\documentclass[final]{siamltex}
\usepackage {epsfig,amsmath,euscript}
\usepackage[latin1]{inputenc}
\usepackage[english]{babel}
\usepackage{color}
\usepackage{amsmath}
\usepackage{amssymb}
\usepackage{ae}
\usepackage{subfigure}
\usepackage[T1]{fontenc}
\usepackage{graphicx}
\usepackage{epstopdf}
\usepackage{longtable}
\usepackage{color}
\definecolor{rred}{rgb}{0.7,0.0,0.2}
\definecolor{bblue}{rgb}{0.2,0.0,0.7}

\newcommand{\secref}[1]{section \ref{sec:#1}}

\newcommand{\seclab}[1]{\label{sec:#1}}
\newcommand{\secsref}[2]{sections \ref{sec:#1} and~\ref{sec:#2}}
\newcommand{\Secsref}[2]{Sections \ref{sec:#1} and~\ref{sec:#2}}
\newcommand{\eqlab}[1]{\label{eq:#1}}
\renewcommand{\eqref}[1]{(\ref{eq:#1})}
\newcommand{\eqsref}[2]{(\ref{eq:#1}) and~(\ref{eq:#2})}

\newcommand{\figref}[1]{Fig.~\ref{fig:#1}}
\newcommand{\figlab}[1]{\label{fig:#1}}
\newcommand{\propref}[1]{Proposition~\ref{proposition:#1}}
\newcommand{\proplab}[1]{\label{proposition:#1}}
\newcommand{\lemmaref}[1]{Lemma~\ref{lemma:#1}}
\newcommand{\lemmalab}[1]{\label{lemma:#1}}
\newcommand{\remref}[1]{Remark~\ref{remark:#1}}
\newcommand{\remlab}[1]{\label{remark:#1}}
\newcommand{\thmref}[1]{Theorem~\ref{theorem:#1}}
\newcommand{\thmlab}[1]{\label{theorem:#1}}
\newcommand{\tablab}[1]{\label{tab:#1}}
\newcommand{\tabref}[1]{Table~\ref{tab:#1}}
\newcommand{\asuref}[1]{Assumption~\ref{assumption:#1}}
\newcommand{\asulab}[1]{\label{assumption:#1}}
\newtheorem{con}{Conjecture}
\newtheorem{asu}{Assumption}

\usepackage{tikz}

\title{Regularizations of two-fold bifurcations in planar piecewise smooth systems using blowup} 

\author{K. Uldall Kristiansen and S. J. Hogan\thanks{K. Uldall Kristiansen: Department of Applied Mathematics and Computer Science, Technical University of Denmark, 2800 Kgs. Lyngby, DK. S. J. Hogan: Department of Engineering Mathematics, University of Bristol, Bristol BS8 1UB, United Kingdom. S.J. Hogan wishes to thank both Danmarks Nationalbank and the Otto M{\o}nsteds Fond for support. In addition, he is extremely grateful to Morten Br{\o}ns for hosting a very productive sabbatical at DTU, Lyngby, Denmark from January to September 2014. }}
\begin{document}
\maketitle 

\begin{abstract}
We use blowup to study the regularization of codimension one two-fold singularities in planar piecewise smooth (PWS) dynamical systems. We focus on singular canards, pseudo-equlibria and limit cycles that can occur in the PWS system. Using the regularization of Sotomayor and Teixeira \cite{Sotomayor96}, we show rigorously how singular canards can persist and how the bifurcation of pseudo-equilibria is related to bifurcations of equilibria in the regularized system. We also show that PWS limit cycles are connected to Hopf bifurcations of the regularization. In addition, we show how regularization can create another type of limit cycle that does not appear to be present in the original PWS system. For both types of limit cycle, we show that the criticality of the Hopf bifurcation that gives rise to periodic orbits is strongly dependent on the precise form of the regularization. Finally, we analyse the limit cycles as locally unique families of periodic orbits of the regularization and connect them, when possible, to limit cycles of the PWS system. We illustrate our analysis with numerical simulations and show how the regularized system can undergo a canard explosion phenomenon.
\end{abstract}



\pagestyle{myheadings}
\thispagestyle{plain}
\section{Introduction}\seclab{Introduction}
Piecewise smooth (PWS) dynamical systems \cite{filippov1988differential, MakarenkovLamb12} are of great significance in applications \cite{Bernardo08}, ranging from problems in mechanics (friction, impact) and biology (genetic regulatory networks) to control engineering \cite{Utkin77}. But, compared to smooth systems \cite{Guckenheimer97}, the study of PWS systems is in its infancy. For example, notions of solution, trajectory, separatrix, topological equivalence and bifurcation, all need revision and extension \cite{filippov1988differential}. Often PWS systems are used as caricatures of smooth systems \cite{Carmona08, Michelson86}, especially if significant amounts of computation are expected. So one of the major challenges of PWS system theory is to see just how close the behaviour of a PWS system is to a suitable smooth system. 

In this paper, we focus on PWS systems in the plane, of the form:
\begin{equation}
\dot{\mathbf{x}}=X^{\pm}(\mathbf{x}), \quad \mathbf{x}\in \Sigma_\pm\subset \mathbb{R}^2, 
\eqlab{pwsdef}
\end{equation}
where the smooth vector fields $X^\pm$, defined on disjoint open regions $\Sigma_\pm$, are smoothly extendable to their common boundary $\Sigma$. The line $\Sigma$ is called the \textit{switching manifold} or \textit{switching boundary}. The union $\Sigma \cup \Sigma_-\cup \Sigma_+$ covers the whole state space.  When the normal components of the vector fields on either side of $\Sigma$ are in \textit{opposition}, a vector field needs to be defined \textit{on} $\Sigma$. The precise choice is not unique and crucially depends on the nature of the problem under consideration. We adopt the widely-used Filippov convention \cite{filippov1988differential}, where a \textit{sliding} vector field is defined on $\Sigma$. In this case, the dynamics is described as \textit{sliding} and the PWS system \eqref{pwsdef}, together with the sliding vector field, constitute a {\it Filippov system}. Such systems possess many phenomena that are not present in smooth systems; grazing and sliding bifurcations, period adding bifurcations and chattering are (almost) ubiquitous in and (virtually) unique to PWS systems.

Sotomayor and Teixeira \cite{Sotomayor96} proposed a regularization of a planar PWS dynamical system, in which the switching manifold $\Sigma$ is replaced by a boundary layer of width $2\epsilon$. Outside the boundary layer, the regularization agrees exactly with the PWS vector fields. Inside the boundary layer, a monotonic function is chosen such that the regularization is at least continuous everywhere. The regualization of PWS systems in $\mathbb{R}^3$ was considered by \cite{Llibre97} and in $\mathbb{R}^n$ by \cite{llibre_sliding_2008}.  

It is natural to ask whether bifurcations in PWS systems are close to bifurcations in a suitable smooth system. But for any regularization, there is a fundamental difficulty when dealing with bifurcations. Fenichel theory \cite{fen1, fen2, fen3, jones_1995}, the main tool used to analyze regularization, requires hyperbolicity, which is lost at a PWS bifurcation. A widely used approach to deal with this loss of hyperbolicity is the \textit{blowup method}, originally due to Dumortier and Roussarie \cite{dumortier_1991, dumortier_1993, dumortier_1996}, and subsequently developed by Krupa and Szmolyan \cite{krupa_extending_2001} in the context of slow-fast systems. 

Buzzi \textit{et al.} \cite{Buzzi06} considered how different PWS phenomena\footnote{For example, they considered crossing, stable and unstable sliding, pseudo-saddle-nodes and two-folds.} in the plane were affected by the regularization method of Sotomayor and Teixeira \cite{Sotomayor96}. A similar study in $\mathbb{R}^3$ was carried out by Llibre \textit{et al.} \cite{Llibre07}. Regularization of PWS systems in $\mathbb{R}^n$ was considered by Llibre \textit{et al.} \cite{llibre_sliding_2008}. These three papers considered the case of one switching manifold separating two different smooth vector fields. Regularization in the case of two intersecting switching manifolds was considered by Llibre \textit{et al.} \cite{Llibre09}, and in the case of surfaces of algebraic variety by Buzzi \textit{et al.} \cite{Buzzi12}. Regularization of codimension one bifurcations in planar PWS systems was considered by De Carvalho and Tonon \cite{DecarvalhoTonon2011}. Common to all of these studies, however, is that they do not deal rigorously with the loss of hyperbolicity at a PWS bifurcation and hence they do not properly unfold the effect of the regularization.  

Recently, Kristiansen and Hogan \cite{KristiansenHogan2015} successfully applied the blowup method of Krupa and Szmolyan \cite{krupa_extending_2001} to study the regularization of both fold and two-fold singularities of PWS dynamical systems in $\mathbb{R}^3$. For two-fold singularities, they showed that the regularized system only fully retains the features of the PWS singular canards when the sliding region does not include a full sector of singular canards. In particular, they showed that every locally unique singular canard persists the regularizing perturbation. For the case of a sector of singular canards, they showed that the regularized system contains a primary canard, provided a certain non-resonance condition holds and they provided numerical evidence for the existence of secondary canards near resonance. Other authors \cite{reves_regularization_2014} have used asymptotic methods to analyze the regularization of a planar PWS fold bifurcation. 

In this paper, we regularize planar codimension one two-fold singularities that occur as the result of collisions of folds (quadratic tangencies) in \textit{both} $X^-$ and $X^+$. We seek to identify PWS bifurcations as smooth bifurcations through regularization. We will study the fate of singular canards, pseudo-equilibria and limit cycles that can occur in our PWS system. We illustrate our analytical results with numerical simulations and show how the regularized system can undergo a canard explosion phenomenon. 



The paper is organized as follows. In \secref{preliminaries}, we set up the problem, define the two-fold singularities we wish to regularize and present our PWS planar system in a normalized form such that the sliding regions retain their character under parameter variation. In \secref{pws}, we describe those properties of the PWS system that we wish to regularize, paying particular attention to singular canards, pseudo-equilibria and limit cycles.  Then in \secref{regularize}, we present a regularized version of our PWS system, using the approach of Sotomayor and Teixeira \cite{Sotomayor96}. Before beginning our analysis, we collect together all our main results in \secref{mainResult}, giving the reader a concise summary of what is to come. In \secref{canards}, we carry out a blowup analysis \cite{krupa_extending_2001} and show how singular canards persist the regularization. Our PWS system can also exhibit pseudo-equilibria, so in \secref{equilibriaregular}, we consider how these unique PWS phenomena survive regularization. In \secref{limitCycles}, we show how limit cycles that are present in the original PWS system behave when regularized. In addition, we show how regularization can create another type of limit cycle that does not appear to be present in the original PWS system. For both types of limit cycle, we show how that the criticality of the Hopf bifurcation that gives rise to periodic orbits is strongly dependent on the precise form of the regularization. Some numerical results are presented in \secref{numerics} to illustrate our analysis. Our conclusions are presented in \secref{conclusions}.

\section{Preliminaries}\seclab{preliminaries}
In this section we set up the problem, define two-fold singularities and present our PWS planar system in a suitable normalized form. Let $\textbf x=(x,y)\in \mathbb R^2$, $\mu\in \mathbb R$. Consider an open set $(\textbf{x},\mu)\in \mathcal U\times \mathcal I$ and a smooth function $f_\mu(\textbf x)$ having $0$ as a regular value for all $\mu\in \mathcal I$. Then $\Sigma \subset \mathcal U$ defined by $\Sigma=f_\mu^{-1}(0)$ is a smooth $1D$ manifold. The manifold $\Sigma$ is our switching boundary. It separates the set $\Sigma_{+} =\{(x,y)\in \mathcal U\vert f_\mu(x,y)>0\}$ from the set $\Sigma_{-}=\{(x,y)\vert f_\mu(x,y)<0\}$. We introduce local coordinates so that $f_\mu(x,y)=y$ and so $\Sigma =\{(x,y)\in \mathcal U\vert y=0\}$. From now on, we suppress the subscript $\mu$.

We consider two smooth vector-fields $X^+$ and $X^-$ that are smooth on $\overline{\Sigma}_+$ and $\overline{\Sigma}_-$, respectively, and define the PWS vector-field $X=(X^-,X^+)$ by
\begin{eqnarray}
  X(\textbf x,\mu) 
  =\left\{ \begin{array}{cc}
                                   X^-(\textbf x,\mu)& \text{for}\quad  \textbf x\in \Sigma_-\\
                                   X^+(\textbf x,\mu)& \text{for}\quad \textbf x\in \Sigma_+
                                  \end{array}\right.\eqlab{Xpm}
\end{eqnarray}
Then
\begin{itemize}
 \item $\Sigma_{cr}\subset \Sigma$ is the \textit{crossing region} where $(X^+ f(x,0,\mu)(X^-f(x,0,\mu)) =X_2^+(x,0,\mu) X_2^-(x,0,\mu) >0$.
 \item $\Sigma_{sl}\subset \Sigma$ is the \textit{sliding region} where $(X^+ f(x,0,\mu))(X^-f(x,0,\mu)) =X_2^+(x,0,\mu) X_2^-(x,0,\mu)<0$.
\end{itemize}
Here $X^{\pm} f(\cdot,\mu)=\nabla f \cdot X^{\pm}(\cdot,\mu)$ denotes the Lie-derivative of $f$ along $X^{\pm}(\cdot,\mu)$. Since $f(x,y)=y$ in our coordinates we have that $X^{\pm} f =X_2^{\pm}$. 

In the sliding region, the vector fields on either side of $\Sigma_{sl}$ point either toward or away from $\Sigma_{sl}$. In this case, in order to have a solution to our system in forward or backward time, we need to define a vector field \textit{on} $\Sigma_{sl}$. There are many possibilities, depending on the problem being considered. One of the most widely adopted definitions is the Filippov convention \cite{filippov1988differential}, in which the {\it sliding vector field} $X_{sl}(\textbf x,\mu)$ is taken to be the convex combination of $X^+$ and $X^-$:
\begin{eqnarray}
 X_{sl}(\textbf x,\mu) = \sigma X^+(\textbf x,\mu)+ (1-\sigma) X^-(\textbf x,\mu),\eqlab{XSliding}
\end{eqnarray}
where $\sigma \in (0,1)$ is such that $X_{sl}(\textbf x,\mu)$ is tangent to $\Sigma_{sl}$. In this case,
\begin{eqnarray}
 \sigma = \frac{X^-f(x,0,\mu)}{X^-f(x,0,\mu)-X^+f(x,0,\mu)}.\eqlab{lambdaSliding}
\end{eqnarray}

The sliding vector field $X_{sl}(\textbf x,\mu)$ can have equilibria (\textit{pseudo-equilibria}, or sometimes \textit{quasi-equilibria}  \cite{filippov1988differential}). Unlike in smooth systems, it is possible for trajectories to reach these pseudo-equilibria in finite time. An orbit of a PWS system can be made up of a concatenation of arcs from $\Sigma$ and $\Sigma_{\pm}$.

\subsection{Two-fold singularities}\seclab{fold}
The boundaries of $\Sigma_{sl}$ and $\Sigma_{cr}$ where $X^+ f = X_2^+=0$ or $X^-f = X_2^-=0$ are singularities called \textit{tangencies}. The simplest tangency is the fold singularity, which is defined as follows.
\begin{definition}
A point $q\in \Sigma$ for $\mu \in \mathcal I$ is a  \textnormal{fold singularity} if 
\begin{eqnarray}
X^{+} f(q,\mu)=0,\quad X^+(X^+f)(q,\mu)\ne 0,\eqlab{foldPlus}
\end{eqnarray}
or if
\begin{eqnarray}
X^{-} f(q,\mu)=0,\quad X^-(X^-f)(q,\mu)\ne 0.\eqlab{foldNegative}
\end{eqnarray}
A fold singularity $q$ with $X^{\pm} f(q,\mu)=0$ is \textnormal{visible} if
 \begin{eqnarray}
 X^{\pm }(X^{\pm } f)(q,\mu) \gtrless 0,\eqlab{visible}
 \end{eqnarray}
and \textnormal{invisible} if
\begin{eqnarray}
 X^{\pm}(X^{\pm} f)(q,\mu) \lessgtr 0.\eqlab{invisible}
\end{eqnarray}
\end{definition}

Note that,  for $\mu$ sufficiently small, the inequalities in \eqsref{foldPlus}{foldNegative} are equivalent to the following
 \begin{align}
  \partial_x X_2^+(q,0)X_1^+(q,0)\ne 0,\,\partial_x X_2^-(q,0)X_1^-(q,0)\ne 0.\eqlab{nondeg}
 \end{align}
 
In this paper, we consider the case of the \textit{two-fold} singularity, when there is a fold singularity in {\it both} $X^{\pm}$. In particular, we suppose that $X^{\pm}$ have tangencies at $q^{\pm}=q^{\pm}(\mu)\in \Sigma$, respectively, which collide for $\mu=0$ at $q=q^{\pm}(0)$ with non-zero velocity. {Hence $(q^+-q^-)'(0)\ne 0$.}

\begin{definition}
 We say that the \textnormal{two-fold} singularity $q$ is 
 \begin{itemize}
  \item \textnormal{visible} if $q^+$ and $q^-$ are both visible;
  \item \textnormal{visible-invisible} if $q^+$ ($q^-$) is visible and $q^-$ ($q^+)$ is invisible;
  \item \textnormal{invisible} if $q^+$ and $q^-$ are both invisible. 
 \end{itemize}
\end{definition}

The three different types of two-fold singularity are shown in \figref{visibleInvisible}.
\begin{figure}
\begin{center}
\includegraphics[width=.95\textwidth]{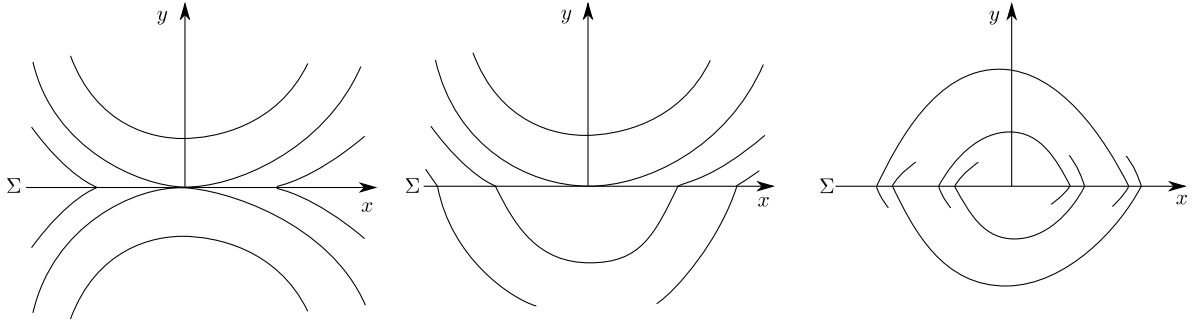}
\caption{The three different types of two-fold singularity studied in this paper: (from L to R) visible, visible-invisible, invisible. Following Filippov \cite{filippov1988differential}, we show neither flow directions nor any sliding vector field.}\figlab{visibleInvisible}
\end{center}
\end{figure}
In the case of a single fold singularity, it is known that both the visible and invisible cases are structurally stable \cite[p. 232]{filippov1988differential}. The regularization of the visible case was studied in \cite{reves_regularization_2014}. 
Filippov \cite[Figs. 58, 59]{filippov1988differential} also considered the case of a single {\it cusp} singularity, which can be either visible or invisible. The cusp singularity is known to be structurally unstable, bifurcating into two tangencies \cite[Figs. 76, 77]{filippov1988differential}, which are on the {\it same} side of $\Sigma$. Kuznetsov {\it et al.} \cite[Fig. 9]{Kuznetsov2003} considered these bifurcations, which they label $DT_{1,2}$, together with the cases we consider here. But we feel that the cusp singularity is best left for future work, as part of the wider picture that includes cusp-fold and two-cusp singularities. 
The two-fold singularities that we consider are shown in Filippov \cite[Figs. 64, 65, 67, 68]{filippov1988differential}, where they are termed {\it type 3 singularities}\footnote{Other type 3 singularities, shown in \cite[Figs. 66, 69, 70, 71]{filippov1988differential}, have codimension greater than one (see \cite[p. 239]{filippov1988differential}). They include cusp-fold and two-cusp singularities.}. There are $7$ different generic cases. These were subsequently called $VV_{1,2}, VI_{1-3}, II_{1,2}$ by Kuznetsov {\it et al.} \cite{Kuznetsov2003}; a notation that we will find useful to adopt. Other authors \cite{Buzzi06, DecarvalhoTonon2011} refer to two-fold singularities as {\it fold-fold singularities}, which can be {\it hyperbolic} (visible), {\it elliptic(al)} (invisible) or {\it parabolic} (visible-invisible). 
Two-folds in $\mathbb{R}^3$ were considered by the present authors in  \cite{KristiansenHogan2015}.

\subsection{Normalized equations}\seclab{normalized}
In this section, we derive a normalized form for the equations near a two-fold singularity at $(q,\mu)=0$ in $\mathbb R^2$. 
By Taylor-expanding $X^{\pm }$, we have, for $y>0$,
\begin{align*}
 \dot x &=X_1^+(0) + \mathcal O(x+y+\mu),\\
 \dot y&=\partial_y X_2^+(0) y+\partial_x X_2^+(0)x+\mathcal O(\vert (x,y)\vert^2+\mu(x+y)),
\end{align*}
and, for $y<0$, 
\begin{align*}
 \dot x &=X_1^-(0) + \mathcal O(x+y+\mu),\\
 \dot y&=\partial_y X_2^-(0) y+\partial_x X_2^-(0)(x-\mu)+\mathcal O(\vert (x-\mu,y)\vert^2+\mu(x-\mu+y)).
\end{align*} 
We now introduce $\tilde x$ and $\tilde t$ where 
\begin{align*}
 x &= \sqrt{\left|\frac{X_1^+(0)}{\partial_y X_2^+(0)}\right|}\tilde x,\\
 t &= \frac{\text{sign}(\partial_y X_2^+(0))}{\sqrt{\vert X_1^+(0) \partial_x X_2^+(0)\vert}}\tilde t,
\end{align*}
which is well-defined, by virtue of \eqref{nondeg}. Then, on dropping tildes, we have, for $y>0$:
\begin{align}
 \dot x &=\delta + \mathcal O(x+y+\mu), \eqlab{normplus}  \\
 \dot y&=x+\mathcal O(y+\mu x+x^2), \nonumber
\end{align}
and, for $y<0$, 
\begin{align}
 \dot x &=\alpha  + \mathcal O(x+y+\mu), \eqlab{normminus}  \\
 \dot y&=-\beta (x-\mu)+\mathcal O(y+\mu (x-\mu)+(x-\mu)^2), \nonumber
\end{align} 
where $\delta = \text{sign}(X_1^+(0) \partial_x X_2^+(0))=\pm 1$. The constants 
\begin{align*}
 \alpha &= \text{sign}(\partial_y X_2^+(0))\frac{X_1^-(0)}{\vert X_1^+(0)\vert},\\
 \beta &= \text{sign}(\partial_y X_2^+(0))\frac{\partial_y X_2^-(0)}{\sqrt{\vert X_1^+(0) \partial_x X_2^+(0)\vert}},
\end{align*}
are non-zero by \eqref{nondeg}. Later on, we will need to include higher order terms in our analysis. We introduce the following coefficients: 
\begin{align}
 \zeta^{\pm},\,\chi^\pm,\,\, \text{and}\,\, \eta^{\pm},\eqlab{zetaeta}
\end{align}
so that \eqref{normplus} becomes for $y>0$:
\begin{align}
 \dot x &=\delta + \zeta^+x+\mathcal O(x^2+y+\mu), \eqlab{normplus2} \\
 \dot y&=x+\eta^+x^2+\chi^+y+\mathcal O(xy+\mu(x+y)+x^3), \nonumber
\end{align}
and \eqref{normminus} becomes for $y<0$:
\begin{align}
 \dot x &=\alpha  + \zeta^-x+\mathcal O(x^2+y+\mu), \eqlab{normminus2}  \\
 \dot y&=-\beta (x-\mu)+\eta^-(x-\mu)^2+\chi^-y+\mathcal O(xy+\mu y+\mu (x-\mu)+(x-\mu)^3). \nonumber
\end{align} 

\begin{remark}\remlab{normalform}
De Carvalho and Tonon \cite{DecarvalhoTonon2014} have given normal forms for codimension one planar PWS vector fields\footnote{De Carvalho and Tonon (private communication) have indicated that they intend to publish a corrigendum to this paper, since their normal forms, as currently stated \cite{DecarvalhoTonon2014}, can not distinguish between all the different planar PWS singularities.}. However, we need \eqref{normplus}, \eqref{normminus} and \eqref{normplus2}, \eqref{normminus2} in this form in order to unfold several of the phenomena studied in this paper. 
\end{remark}

The sliding vector field \eqref{XSliding} is given by
\begin{eqnarray}
 \dot x&=&\sigma X_1^+(x,0,\mu)+(1-\sigma )X_1^-(x,0,\mu),\eqlab{slidingEqns}\\
 \dot y&=&0,\nonumber
\end{eqnarray}
where $\sigma$, defined in \eqref{lambdaSliding}, is given by
\begin{eqnarray}
 \sigma &=& \frac{(-\beta+\mathcal O(x+\mu))(x-\mu)}{(-\beta +\mathcal O(x+\mu))(x-\mu)-(1+\mathcal O(x+\mu))x}.\eqlab{sigma}
\end{eqnarray}
The denominator in \eqref{sigma} is positive for \textit{stable sliding} $\Sigma_{sl}^-$ and negative for \textit{unstable sliding} $\Sigma_{sl}^+$. So if we multiply \eqref{slidingEqns} by the modulus of this denominator, $\vert (-\beta +\mathcal O(x+\mu))(x-\mu)-(1+\mathcal O(x+\mu))x\vert$, corresponding to a transformation of time, we find on $y=0$ that, in $\Sigma_{sl}^{\mp}$,
\begin{align}
\dot x&=\pm (-\beta +\mathcal O(x+\mu))(x-\mu)(\delta +\mathcal O(x+\mu))\pm (-1+\mathcal O(x+\mu))x(\alpha+\mathcal O(x+\mu)),\eqlab{slidingEqns2}\\
 \dot y&=0,\nonumber
\end{align}
Equilibria of \eqref{slidingEqns2} are pseudo-equilibria, which we will study in \secref{pseudo} below. 

Within $\Sigma_{sl}^-$ for $\mu=0$ we find from \eqref{slidingEqns2} that
\begin{align*}
 \dot x = (-\beta\delta +\mathcal O(x))x -(\alpha+\mathcal O(x))x = -(\beta \delta +\alpha)x+\mathcal O(x^2).
\end{align*}

\begin{proposition}\proplab{visibility}
 The fold $q^+=(0,0)$ is visible (invisible) from above if $\delta=1$ ($\delta = -1$), whereas the fold $q^-=(\mu,0)$ is visible (invisible) from below if $\alpha \beta>0$ ($\alpha \beta<0$). Hence the two-fold $q=(0,0)$ for $\mu=0$ is 
 \begin{itemize}
  \item \textnormal{visible} if $\delta = 1$ and $\alpha\beta>0$;
  \item \textnormal{visible-invisible} if $\delta = 1$ ($\delta=-1$) and $\alpha\beta<0$ ($\alpha \beta>0$);
  \item \textnormal{invisible} if $\delta = -1$ and $\alpha\beta<0$.
 \end{itemize}

 We also have that 
 \begin{align}
 \Sigma_{sl}&:\, \beta x(x-\mu)>0,\eqlab{SigmaSl}\\
  \Sigma_{cr}&:\,\beta x(x-\mu)<0,  \eqlab{SigmaCr}
 \end{align}
 for $x$ and $\mu$ sufficiently small. The subset $\Sigma_{sl}^-=\Sigma_{sl}\cap \{x<0\}$ of $\Sigma_{sl}$ is the \textnormal{stable sliding} region whereas the subset $\Sigma_{sl}^+=\Sigma_{sl} \cap \{x>0\}$ of $\Sigma_{sl}$ is the \textnormal{unstable sliding} region. The subset $\Sigma_{cr}^-=\Sigma_{cr}\cap \{x<0\}$ of $\Sigma_{cr}$ is \textnormal{crossing downwards} whereas the subset $\Sigma_{cr}^+=\Sigma_{cr}\cap \{x>0\}$ of $\Sigma_{cr}$ is \textnormal{crossing upwards}.
\end{proposition}
\begin{proof}
 These statements follow from simple computations. For example, to obtain the last part, we note that
 \begin{align*}
  X_2^+(x,0,\mu) X_2^-(x,0,\mu) = (1+\mathcal O(x+\mu))x(-\beta+\mathcal O(x+\mu))(x-\mu),
 \end{align*}
and use the definition of $\Sigma_{sl}$ in \secref{preliminaries} together with \eqref{nondeg}.
\end{proof}

Henceforth, in the visible-invisible case, without loss of generality, we will focus on the case $\delta=1$, $\alpha \beta<0$ so that the fold $q^+=(0,0)$ is visible from above and $q^-=(\mu,0)$ is invisible from below (as in \figref{visibleInvisible}).

Since we perform a local analysis, we restrict attention to $x$ and $\mu$ sufficiently small so that statements \eqsref{SigmaSl}{SigmaCr} in \propref{visibility}, about $\Sigma_{sl}$ and $\Sigma_{cr}$, apply. The advantage of the form \eqsref{normplus}{normminus} of the normalized equations is that the sliding regions $\Sigma_{sl}^{\pm}$ retain their character (stable or unstable) under parameter variation. 

For later convenience we introduce the parameter $\Omega$ defined by $$\Omega \equiv \beta \delta + \alpha.$$

To conclude this section, we state the following assumptions, which we make throughout the rest of the paper.

\begin{asu}\asulab{asu1}
$\delta\ne 0.$
\end{asu}
\begin{asu}\asulab{asu2}
$\alpha\beta \ne 0.$
\end{asu}
\begin{asu}\asulab{asu3}
$\Omega \ne 0.$
\end{asu}

\asuref{asu1} and \asuref{asu2} are the normalized form of \eqref{nondeg}, when combined with \eqsref{normplus}{normminus}. In fact we have already set $\delta = \pm 1$. The significance of \asuref{asu3} will be explained in \remref{omegasingular} below. 

\section{Analysis of the PWS system}\seclab{pws}

In this section, we analyze the planar PWS system \eqsref{normplus}{normminus}, together with \eqsref{slidingEqns}{sigma} whenever we have sliding.  We pay particular attention to singular canards, pseudo-equilibria and limit cycles that can occur in our system. The fold at $q^+=(0,0)$ is fixed, whereas the fold at $q^-=(\mu,0)$ varies with $\mu$ such that the two-fold at $\mu=0$ bifurcates. Both pseudo-equilibria and sliding sections can appear, disappear or change character depending on whether $q^{\pm}$ are visible or invisible. In addition, some of the two-folds at $\mu=0$ possess singular canards, which disappear for $\mu\ne 0$, and at least one two-fold can have a limit cycle. 

\subsection{Singular canards}\seclab{Singularcanards}

Trajectories can go from the attracting sliding region $\Sigma_{sl}^-$ to the repelling sliding region $\Sigma_{sl}^+$, or vice versa, for $\mu=0$. These trajectories, which we call \textit{singular canards} \cite{KristiansenHogan2015}, resemble canards in slow-fast systems \cite{Benoit81}. A singular canard is called a \textit{vrai} singular canard if it goes from the attracting sliding region $\Sigma_{sl}^-$ to the repelling sliding region $\Sigma_{sl}^+$ in forward time. Singular canards that go from the repelling sliding region $\Sigma_{sl}^+$ to the attracting sliding region $\Sigma_{sl}^-$ are called \textit{faux} singular canards. Singular canards can only exist for $\mu=0$ and when there is sliding in both $x<0$ and $x>0$. From \eqref{SigmaSl}, we see that singular canards can only exist for $\beta>0$. 

For the existence of singular canards in our PWS system, it is important to note that, in terms of the original time used in \eqref{slidingEqns}, the two-fold on $\Sigma_{sl}$ can be reached in finite time. A simple calculation using L'H\^opital's rule shows that on $\Sigma_{sl}$, for $\mu=0$, 
\begin{align}
 \lim_{x\rightarrow 0}\dot x= (1+\beta)^{-1} \Omega.\eqlab{xdotxeq0}
\end{align}

There is no singularity at $1+\beta=0$ since we need $\beta>0$ for sliding. So, by \asuref{asu3}, we have a finite value of $\dot x$ on $\Sigma_{sl}$ for $x\rightarrow 0$ when $\mu=0$. Hence it is possible to pass in finite time through $x=0$ (the point separating attracting and repelling sliding regions, if they exist) at a two-fold.

\begin{remark}\remlab{omegasingular}
 The case $\Omega=0$ is degenerate, since $\lim_{x\rightarrow 0}\dot x$ vanishes. Geometrically this case corresponds to the linearized trajectories of $X^{\pm}$ having the same gradient on $\Sigma_{sl}$. We shall not consider this case further (cf. \asuref{asu3} above).
  \end{remark}
  
Hence by \eqref{xdotxeq0} we conclude that singular canards exist in our PWS system. To decide whether they are {vrai} singular canards or faux singular canards, we need to consider the sign of $\dot x$ in \eqref{xdotxeq0}. 
We collect the results in the following proposition:
\begin{proposition}\proplab{canards}
Singular canards in our PWS system exist if and only if $\beta>0$. If $\Omega> 0$ ($\Omega<0$) then the singular canard is a vrai (faux) singular canard.  
\end{proposition}

One of the main objectives of this work is to establish persistence results of these singular canards under regularization. We will focus primarily on the persistence of \textit{vrai} singular canards. 
 
The different types of two-fold, together with their flow directions and  any sliding regions are shown in \figref{betaAlphaDelta}.
Note that the visible two-folds $VV_{1,2}$ and the visible-invisible two-folds $VI_{1-3}$ exist for $\delta=1$, whereas the invisible two-folds $II_{1,2}$ exist for $\delta=-1$. 

\begin{figure}\begin{center}\includegraphics[width=.85\textwidth]{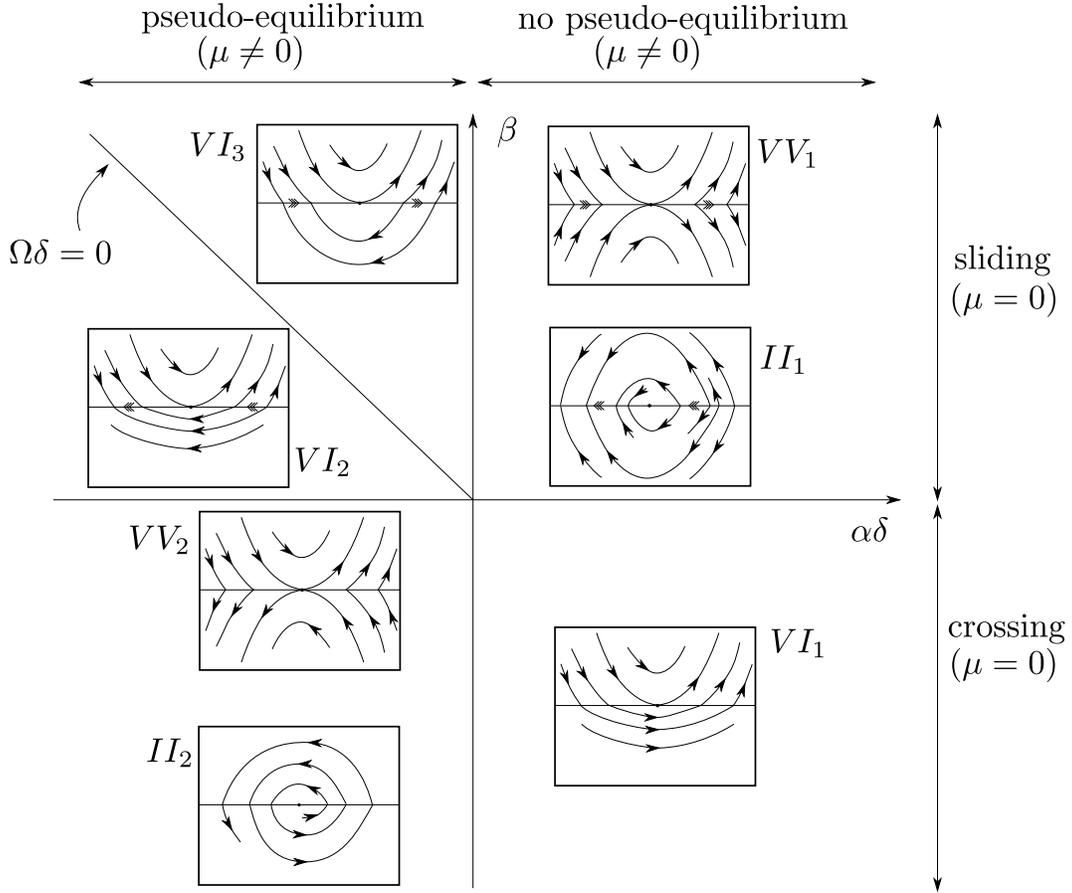}
\caption{The different types of two-fold singularities. Sliding is indicated by triple headed arrows. Cases $II_{1,2}$ occur for $\delta=-1$, the other cases for $\delta=1$.}\figlab{betaAlphaDelta}
                		\end{center}
              \end{figure}

\subsection{Pseudo-equilibria}\seclab{pseudo}

As mentioned in \secref{preliminaries}, the sliding vector field $X_{sl}(\textbf x,\mu)$ itself can have equilibria, called {\it pseudo-equilibria}\footnote{Hence there can be no pseudo-equilibria without a sliding vector field.}. These pseudo-equilibria are not necessarily equilibria of $X^{\pm}(\textbf x,\mu)$. Instead they correspond to the case when $X^{\pm}(\textbf x,\mu)$ on $\Sigma_{sl}$ are linearly dependent. Filippov (\cite{filippov1988differential}, p. 218) terms them {\it type 1} singularities. They comprise three distinct topological classes; a pseudo-node, a pseudo-saddle and a pseudo-saddle-node.
The following proposition describes the existence of pseudo-equilibria in \eqref{slidingEqns2}.
\begin{proposition}\proplab{pwsProposition}
 If
 \begin{align}
 \alpha\delta<0,\eqlab{conditionPseudo}
 \end{align}
then, for $\mu$ sufficiently small,  there exists a pseudo-equilibrium of \eqref{slidingEqns2} at $(x,y)=(x_{ps},0)$, where
 \begin{align}
  x_{ps} = \beta \delta \Omega^{-1} \mu+\mathcal O(\mu^2). \eqlab{xps}
 \end{align}
Also if $\beta \delta \Omega^{-1}\mu<0$ then $(x_{ps},0)\in \Sigma_{sl}^-$ and
\begin{itemize}
 \item for $\Omega<0$: $(x_{ps},0)$ is a pseudo-saddle with local repelling manifold coinciding with $\Sigma_{sl}^-$;
 \item for $\Omega>0$: $(x_{ps},0)$ is a attracting pseudo-node. 
\end{itemize}

\noindent
If $\beta \delta \Omega^{-1}\mu>0$ then $(x_{ps},0)\in \Sigma_{sl}^+$ and
\begin{itemize}
 \item for $\Omega<0$: $(x_{ps},0)$ is a pseudo-saddle with local attracting manifold coinciding with $\Sigma_{sl}^+$;
 \item for $\Omega>0$: $(x_{ps},0)$ is a repelling pseudo-node. 
\end{itemize}

\noindent
If $\alpha\delta>0$, then $(x_{ps},0)$ is not a pseudo-equilibrium.
\end{proposition}
\begin{remark}\remlab{rempsn}
 From \asuref{asu3} we do not consider pseudo-saddle-nodes in our system.
\end{remark}
\begin{proof}
 To find pseudo-equilibria, we set $\dot x=0$ in \eqref{slidingEqns2} to get
 \begin{align*}
  (-\beta\delta+\mathcal O(\mu+x))x+(-\alpha+\mathcal O(\mu+x))x-(-\beta \delta+\mathcal O(\mu))\mu=0,
 \end{align*}
or
\begin{align*}
 (1+\beta^{-1}\delta\alpha +\mathcal O(\mu+x))x= (1+\mathcal O(\mu))\mu.
\end{align*}
Here we have used that $\delta=\pm 1$. Note that $$1+\beta^{-1}\delta\alpha = \beta^{-1} \delta^{-1} \Omega \ne 0,$$ by assumption.  We can therefore solve this equation by the implicit function theorem to obtain
\begin{align*}
 x=x_{ps}\equiv \beta \delta \Omega^{-1} \mu +\mathcal O(\mu^2).
\end{align*}
This is a pseudo-equilibrium if and only if $(x_{ps},0)\in \Sigma_{sl}$. We determine $\Sigma_{sl}$ as follows. Consider first $\beta>0$. Then we have
\begin{align*}
\Sigma_{sl}:&\,x<0\quad \text{or} \quad  x>\mu\quad \text{for}\quad \mu>0,\\
&\,x=0\quad \text{for}\quad \mu=0,\\
&\,x<\mu\quad \text{or}\quad x>0\quad \text{for}\quad \mu<0.
\end{align*}
Thus $\Sigma_{sl}:\,\text{sign}(\mu)x\notin (0,\vert \mu\vert)$. Then since  
\begin{align*}
 \text{sign}(\mu)x_{ps} = \beta \delta \Omega^{-1} \vert \mu \vert +\mathcal O(\mu^2),
 \end{align*}
 we conclude that $(x_{ps},0)\in \Sigma_{sl}$ for $\mu$ sufficiently small, provided $\beta \delta \Omega^{-1}=(1+\beta^{-1}\delta \alpha )^{-1}\notin (0,1)$. Since $\beta>0$, this condition is equivalent to $\alpha\delta<0$. If, on the other hand $\beta \delta \Omega^{-1}=(1+\beta^{-1}\delta \alpha )^{-1}\in (0,1)$, then $\alpha\delta>0$, for $\beta>0$ and so $(x_{ps},0)$ is not a pseudo-equilibrium. 

Next consider $\beta<0$. Then 
\begin{align*}
\Sigma_{sl}:\,\text{sign}(\mu)x\in (0,\vert \mu \vert).
\end{align*}
Hence $(x_{ps},0)\in \Sigma_{sl}$ for $\mu\ne 0$ sufficiently small, provided $\beta \delta \Omega^{-1}=(1+\beta^{-1}\delta\alpha )^{-1}\in (0,1)$. For $\beta<0$, this is equivalent to $\alpha\delta<0$. If, on the other hand $\beta \delta \Omega^{-1}=(1+\beta^{-1}\delta \alpha )^{-1}\notin (0,1)$, then $\alpha\delta>0$, for $\beta<0$ and so $(x_{ps},0)$ is not a pseudo-equilibrium.

We conclude that $(x_{ps},0)$ is a pseudo-equilibrium, where $x_{ps}$ is defined in \eqref{xps}, provided \eqref{conditionPseudo} holds. 

If $x_{ps}<0$ then $(x_{ps},0)\in \Sigma_{sl}^{-}$, the region of stable sliding, and if $x_{ps}>0$ then $(x_{ps},0)\in \Sigma_{sl}^{+}$, the region of unstable sliding (see \propref{visibility}). Combining this with the linearization of \eqref{slidingEqns2} about $(x_{ps},0)$ gives the statements about stability.
\end{proof}
\begin{figure}\begin{center}\includegraphics[width=.90\textwidth]{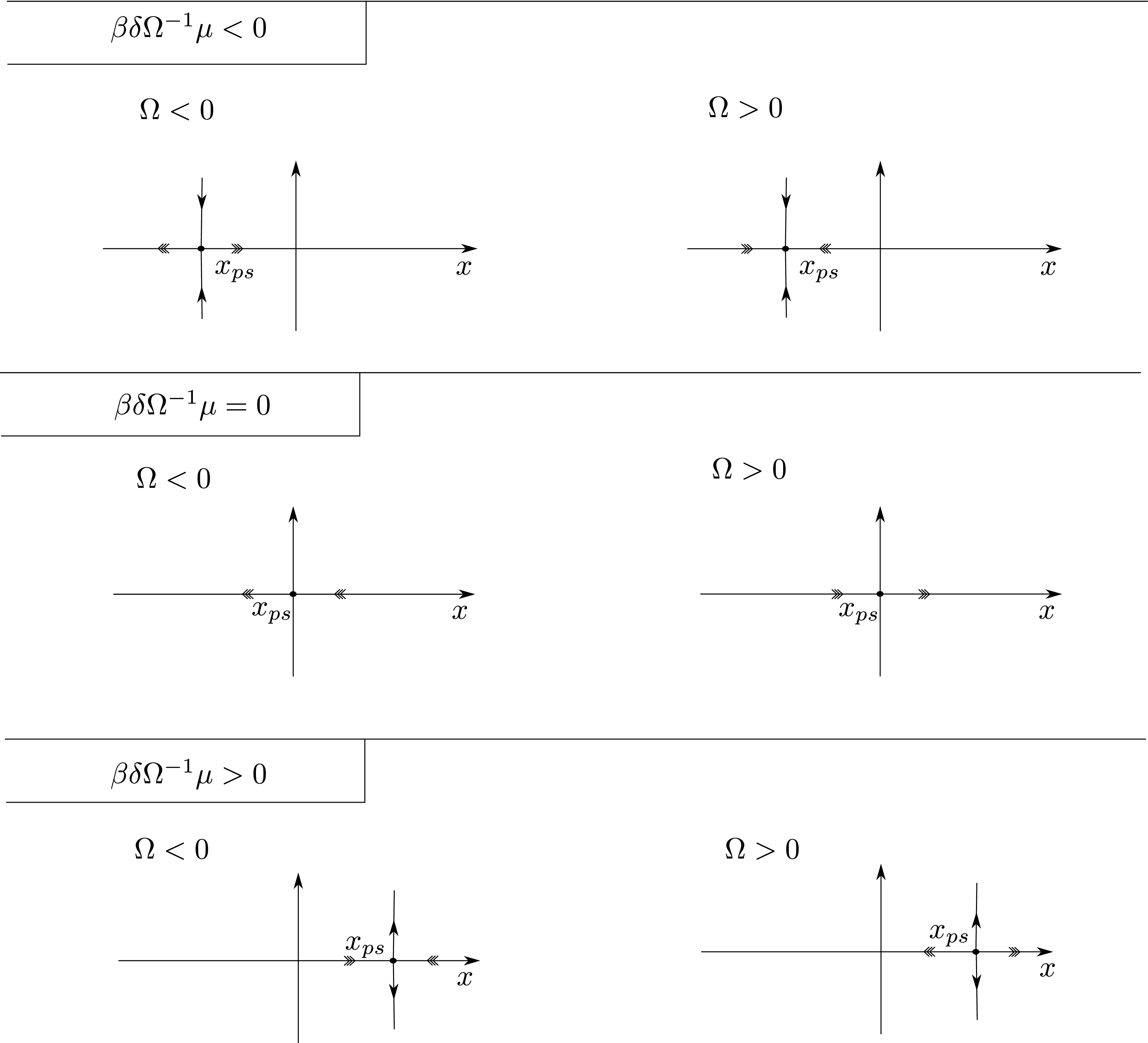}
\caption{Results of \propref{pwsProposition}. Here $\alpha\delta<0$ and $x_{ps}=\beta \delta \Omega^{-1}\mu+\mathcal O(\mu^2)$.}\figlab{prop5_2}
                		\end{center}
              \end{figure}
%

These results are summarized in \figref{prop5_2}. As mentioned earlier, two-fold singularities occur in \eqsref{normplus}{normminus} when $\mu=0$. From \propref{pwsProposition}, it follows that the two-fold singularity can be accompanied by significant changes in the nature of pseudo-equilibria around $\mu=0$. For example, if $\Omega<0$ then $x_{ps}$ is a pseudo-saddle for all $\mu\ne 0$. The difference between $\mu>0$ and $\mu<0$ is, where $\beta \delta \Omega^{-1}\mu<0$, the pseudo-saddle is in $\Sigma_{sl}^-$, which coincides with the associated repelling manifold of $x_{ps}$. But, where $\beta \delta \Omega^{-1}\mu>0$, the pseudo-saddle is in $\Sigma_{sl}^+$, which coincides with the associated attracting manifold of $x_{ps}$. Hence the attracting and repelling directions ``switch'' on passage through the two-fold at $\mu=0$. We can see this behaviour in \cite[Fig. 10]{Kuznetsov2003}, where $\Sigma_{sl}$ is the attracting manifold of the pseudo-saddle of $VV_2$ on one side of the two-fold, whereas it is the repelling manifold on the other. Similarly, when $\Omega>0$ a pseudo-equilibrium goes from being attracting on one side to repelling on the other side. We see this behaviour for example in \cite[Fig. 11]{Kuznetsov2003},  for the $VI_3$ two-fold, where a repelling pseudo-node becomes an attracting pseudo-node.

Another main objective of this paper is to understand how the behaviour of these pseudo-equilibria is modified when our governing equations are regularized.

\subsection{Limit cycles}\seclab{limitCyclesPWS}
A further phenomenon in the two-fold singularity is the existence of limit cycles. For our PWS system, it is clear from \figref{betaAlphaDelta} (see also \cite[Fig. 12]{Kuznetsov2003}) that (local) limit cycles can occur in the invisible case $II_2$ where:
\begin{align}
\delta = -1,\quad \alpha>0,\quad \beta<0.\eqlab{II2cond}
\end{align}
To study periodic orbits in this case, one can introduce a Poincar\'e map $P_0$ which takes $\{(x,0)\vert x>0\}$ into $\{(x,0)\vert x>\mu\}$ under the forward flow of $X^+$ and $X^-$ (see \figref{crossingPOS} and \cite[p. 236]{filippov1988differential}). The map $P_0$ is composed of $\sigma_0^+:\,x_0\mapsto x_1$, the mapping from $(x_0,0),\,x_0>0$, to $(x_1(x_0),0),\,x_1(x_0)<0$, under the forward flow of $X^+$, and $\sigma_0^-:x_1\mapsto x_2$, the mapping from $(x_1,0),\,x_1<\mu$, to $(x_2(x_1),0),\,x_2(x_1)>\mu$, under the forward flow of $X^-$. Clearly $P_0$ is only defined for those $x_0$ for which $x_1<\mu$. For $X^-$ to map $x_1<\mu$ into $x_2>\mu$ we need $\dot x=\alpha +\mathcal O(x+y+\mu)>0$. Therefore $\alpha>0$.
\begin{figure}
\begin{center}
\includegraphics[width=.60\textwidth]{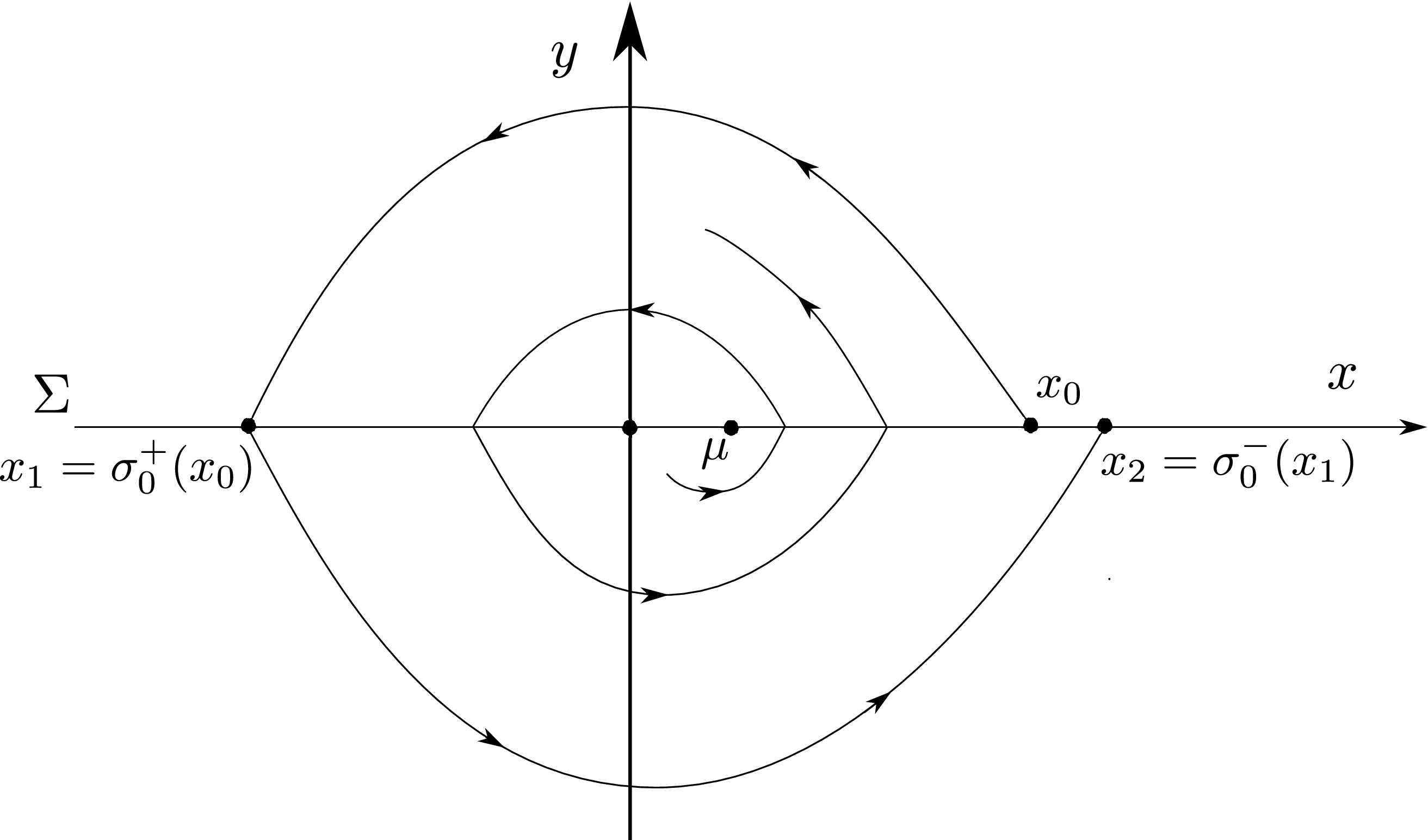}
\caption{The Poincar\'e map $P_0=\sigma_0^-\circ \sigma_0^+$ associated with the invisible two-fold singularity $II_2$, where $\delta=-1$, $\alpha>0$ and $\beta<0$.}\figlab{crossingPOS}
\end{center}
\end{figure}
Hence we have the following lemma.
\begin{lemma}\lemmalab{sigmaMaps}
Consider \eqref{II2cond}. Then
 \begin{align*}
  \sigma_0^+(x_0) &= -x_0+A^+x_0^2+\mathcal O(x_0(\mu+x_0^2)),\\
  \sigma_0^-(x_1) &= -x_1+2\mu +A^-x_1^2+\mathcal O(x_1(\mu + x_1^2)),
 \end{align*}
 where
 \begin{align}
  A^- &= \frac{2}{3\alpha \beta }(\alpha \eta^-+\beta ( \zeta^-+\chi^-)),\eqlab{AN}\\
  A^+ &= -\frac{2}{3}(\eta^++\zeta^++\chi^+).\nonumber
   \end{align}

\end{lemma}
\begin{proof}
See \cite[p. 236]{filippov1988differential}. 
\end{proof}

The Poincar\'e mapping $P_0=\sigma_0^-\circ \sigma_0^+$ therefore takes the following form:
 \begin{align}
  P_0(x_0) = \sigma_0^-(\sigma_0^+(x_0))=x_0+2\mu+(A^--A^+)x_0^2+\mathcal O(x_0^3+\mu x_0),\eqlab{P0Map}
 \end{align}
 for those $x_0$ which satisfy the inequality:
 \begin{align}
 \sigma_0^+(x_0) = -x_0+A^+x_0^2+\mathcal O(x_0^3+x_0\mu)<\mu.\eqlab{x1LeMu}
\end{align}
\begin{proposition}\proplab{propLimitCyclesPWS}
Consider \eqref{II2cond} and suppose that 
\begin{align}
\Delta_{II_2} \equiv A^--A^+=\frac{2}{3\alpha \beta}\left(\alpha (\eta^-+\beta \eta^+)+\beta(\zeta^-+\chi^-+\alpha (\zeta^++\chi^+))\right),\eqlab{Ceqn}
\end{align}
is non-zero. 
Then, for $\mu \Delta_{II_2}^{-1}<0$ sufficiently small, the PWS system has a family of periodic orbits. These periodic orbits correspond to fixed points of $P_0$ of the following form
\begin{align}
 x_0(\mu) = \sqrt{-2\mu \Delta_{II_2}^{-1}}+\mathcal O(\mu).\eqlab{x0mu}
\end{align}
The periodic orbits are attracting for $\Delta_{II_2}<0$ and repelling for $\Delta_{II_2}>0$.  
\end{proposition}
\begin{proof}
 We obtain \eqref{x0mu} by solving the equation $P_0(x_0)=x_0$ for $x_0>0$ using the implicit function theorem. For $\mu$ small, $x_0(\mu)$ is positive and satisfies the inequality \eqref{x1LeMu}. For stability we compute the derivative of \eqref{P0Map} to get:
 \begin{align*}
  P_0'(x_0(\mu)) &= 
  1+2\Delta_{II_2}x_0(\mu)+\mathcal O(\mu).
 \end{align*}
  If $\Delta_{II_2}<0$ ($\Delta_{II_2}>0$) then $0<P'(0)<1$ ($P'(0)>1$) for $\mu$ sufficiently small. The result follows.
\end{proof}

{Finally in this subsection, we note that the visible-invisible case $VI_3$ does not appear to have limit cycles. However, by straightening out the flow within $\Sigma^-$, so that $\Sigma$ becomes curved and quadratic at the fold, the $VI_3$ case clearly resembles the classical slow-fast curved critical manifold with \textit{singular cycles}. This similarity can be seen in \figref{VI3vdP} where we illustrate (a) the $VI_3$ two-fold and (b) the slow-fast equivalent as it appears, for example, in the van der Pol system (see also \cite[Fig. 5]{krupa_relaxation_2001}). 
There are two singular canard cycles in \figref{VI3vdP} (b), shown as thick curves, each composed of fast segments (with triple-headed arrows) and slow segments on the curved slow manifold. For $\epsilon$ sufficiently small, this slow-fast system is known to produce a \textit{canard explosion phenomenon} \cite{krupa_relaxation_2001}, in which the singular cycles become limits of a family of rapidly increasing periodic orbits as $\epsilon\rightarrow 0$. However, when considering \figref{VI3vdP}, it is important to highlight that there is no time scale separation in the PWS system. Nevertheless, we shall see that the \textit{regularized} PWS system possesses a hidden slow-fast structure near the discontinuity set. This will allow us identify the \textit{singular cycles} in \figref{VI3vdP} (a) as limits of periodic orbits of the regularization, and hence strengthen the connection between (a) and (b) further. The singular cycles in \figref{VI3vdP} (a) are also illustrated as thick curves, but in comparison to the cycles in (b), they are composed of an orbit segment of $X^-$, within $\Sigma^-$, and a sliding segment on $\Sigma$. }

Another objective of this paper to understand the existence of limit cycles under regularization. Further details are presented in \secref{limitCycles}. 

\begin{figure}
\begin{center}
\includegraphics[width=.80\textwidth]{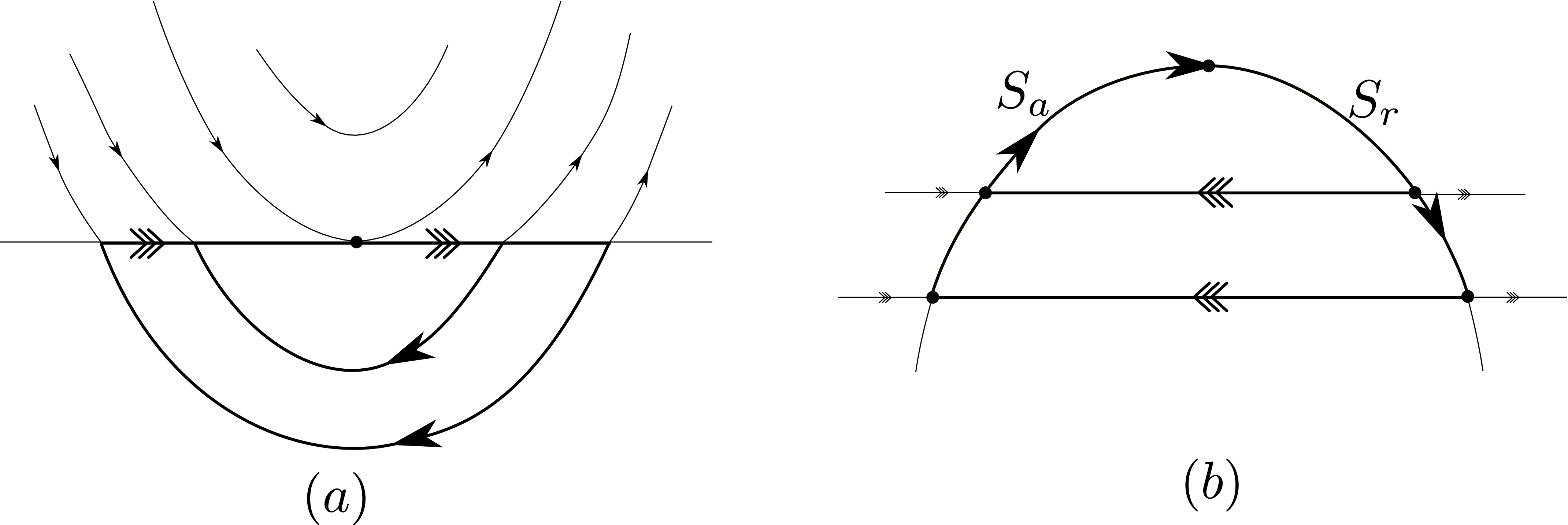}
\caption{In (a): the $VI3$ PWS two-fold bifurcation and in (b): the classical curved critical manifold e.g. appearing in the slow-fast van der Pol system. In (b) $S_a$ and $S_r$ denote the attracting and repelling part of the critical manifold. Note also the different use of triple-headed arrows. In (a) triple-headed arrows are used to indicate sliding whereas they in (b) are used to indicate fast orbit segments described by a set of layer equations. The thick curves illustrate \textit{singular} closed curves which for case (a) will be become periodic orbits upon regularization.}\figlab{VI3vdP}
\end{center}
\end{figure}

\subsection{Summary of two-fold properties}\seclab{summary}
We conclude this section with \tabref{tblPWS} which summarizes properties of the seven two-folds, for future reference.\\

\begin{center}
\begin{longtable}{| c | c | c | c | c | c | c | c |}
\caption{1: type of two-fold singularity [visible (VV), visible-invisible (VI) and invisible (II), following \cite{Kuznetsov2003}]. 2: value of $\delta$. 3: sign of $\alpha$. 4: sign of $\beta$. 5: sign of $\Omega$. 6: type of singular canard (x=no canard). 7: type of pseudo-equilibrium after bifurcation (PS=pseudo-saddle, PN=pseudo-node, x=no pseudo-equilibrium). 8: possibility of limit cycle.}\tablab{tblPWS} \\
\hline
1 & 2 & 3 & 4 & 5 & 6 & 7 & 8\\
\hline
 \hline
Two-fold type & $\delta$ & $\alpha$ & $\beta$ & $\Omega$ & Singular canard & Pseudo-equilibrium & Limit cycle \\
\hline
\hline
$VV_1$ & 1 & + & + & + & {vrai} & x & no\\
\hline
$VV_2$ & 1 & - & - & - & x & PS & no\\
\hline
$VI_1$ & 1 & + & - & $\pm$ & x & x & no\\
\hline
$VI_2$ & 1 & - & + & - & faux & PS & no\\
\hline
$VI_3$ & 1 & - & + & + & {vrai} & PN & -\footnote{The regularized two-fold $VI_3$ \textit{does} possess a limit cycle. See Theorem~\ref{mainLimitCycles}. }\\
\hline
$II_1$ & -1 & - & + & - & faux & x & no\\
\hline
$II_2$ & -1 & + & - & + & x & PN & yes\\
 \hline
 \end{longtable}
\end{center}

\section{Regularization}\seclab{regularize}
It is natural to ask how the results in \secref{pws} are affected by regularization. For example, there is the question of the persistence of the singular canards (\secref{Singularcanards}). Then, of course, pseudo-equilibria (\secref{pseudo}) can not exist in a smooth system, so we will need to look at the existence and behaviour of equilibria in the regularization. Finally there is the need to understand both the fate of the limit cycles (\secref{limitCyclesPWS}) in the invisible two-fold $II_2$ under regularization and the possibility of other limit cycles which may appear in the regularized system.

There is a number of ways that the original PWS vector field $X=(X^-,X^+)$ can be regularized. We follow the approach of Sotomayor and Teixeira \cite{Sotomayor96}. We define a $C^k$-function ($1\le k\le \infty$) $\phi(s)$ which satisfies:
\begin{eqnarray}
 \phi(s)=\left\{\begin{array}{cc}
                 1 & \text{for}\quad s\ge 1,\\
                 \in (-1,1)& \text{for}\quad s\in (-1,1),\\
                 -1 & \text{for}\quad s\le -1,\\
                \end{array}\right.\eqlab{phiFunc}
 \end{eqnarray}
 where 
 \begin{align}
 \phi'(s)>0 \quad \text{within}\quad s\in (-1,1).\eqlab{phiProperties}
 \end{align}
 Then for $\epsilon>0$, the regularized vector-field $X_\epsilon (\textbf x,\mu)$ is then given by
\begin{eqnarray}
 X_\epsilon (\textbf x,\mu) =
 \frac12 X^+(\textbf x,\mu)(1+\phi(\epsilon^{-1} y)) +\frac12 X^-(\textbf x,\mu)(1-\phi(\epsilon^{-1}y)).\eqlab{Xeps}
\end{eqnarray}
Note that 
\begin{align}
X_\epsilon(\textbf x,\mu)=X^{\pm}(\textbf x,\mu) \quad \text{for}\quad  y\gtrless \pm \epsilon.\eqlab{SotomayorProperty}
\end{align}
The region $y\in (-\epsilon,\epsilon)$ is the \textit{region of regularization}. The system \eqref{Xeps} has a hidden slow-fast stucture which is unfolded by the re-scaling (see also \cite{reves_regularization_2014}) 
\begin{align}
\hat y=\epsilon^{-1} y.\eqlab{yTilde}
\end{align}
In terms of this re-scaling the region of regularization becomes $\hat y \in (-1,1)$. 
Then, using \eqsref{normplus}{normminus}, the regularized system \eqref{Xeps} becomes
\begin{eqnarray}
\dot{x} &=&\frac{1}{2} \left ( (\delta + \mathcal O(x+\mu+\epsilon \hat y))(1+\phi(\hat y))+(\alpha  + \mathcal O(x+\mu+\epsilon \hat y))(1-\phi(\hat y))\right ),\eqlab{slowslowFastEquations1}\\
\epsilon\dot{\hat y} &=&\frac{1}{2} \left ( (x+\mathcal O(\mu x+x^2+\epsilon \hat y))(1+\phi(\hat y))+(-\beta (x-\mu)+\mathcal O(\mu (x-\mu)+(x-\mu)^2+\epsilon \hat y))(1-\phi(\hat y))\right ).\nonumber
 \end{eqnarray}
This is the \textit{slow system} of a slow-fast system. The $\hat y$ variable is fast with $\mathcal O(\epsilon^{-1})$ velocities and $x$ is the slow variable with $\mathcal O(1)$ velocities.  The limit $\epsilon=0$:
  \begin{eqnarray}
\dot{x} & = &\frac{1}{2} \left ( (\delta + \mathcal O(x+\mu))(1+\phi(\hat y))+(\alpha  + \mathcal O(x+\mu))(1-\phi(\hat y))\right ),\eqlab{ReducedProblemY2}\\
0 & = & \frac{1}{2} \left ( (x+\mathcal O(\mu x+x^2))(1+\phi(\hat y))+(-\beta (x-\mu)+\mathcal O(\mu (x-\mu)+(x-\mu)^2))(1-\phi(\hat y))\right ). \nonumber
 \end{eqnarray}
is called the \textit{reduced problem}. 
When we re-scale time according to $\tau = \frac{1}{2\epsilon} t$, \eqref{slowslowFastEquations1} becomes:
\begin{eqnarray}
x' 
&=&\epsilon \left ( (\delta + \mathcal O(x+\mu+\epsilon \hat y))(1+\phi(\hat y))+(\alpha  + \mathcal O(x+\mu+\epsilon \hat y))(1-\phi(\hat y))\right ),\eqlab{slowFastEquations1}\\
\hat y'
 &=&(x+\mathcal O(\mu x+x^2+\epsilon \hat y))(1+\phi(\hat y))+(-\beta (x-\mu)+\mathcal O(\mu (x-\mu)+(x-\mu)^2+\epsilon \hat y))(1-\phi(\hat y)),\nonumber
 \end{eqnarray}
with $()'=\frac{d}{d\tau}$. This is the \textit{fast system} of a slow-fast system. The limit $\epsilon=0$:
  \begin{eqnarray}
x' & = & 0, \eqlab{LayerProblemY2}\\
\hat y'&=&(x+\mathcal O(\mu x+x^2))(1+\phi(\hat y))+(-\beta (x-\mu)+\mathcal O(\mu (x-\mu)+(x-\mu)^2))(1-\phi(\hat y)). \nonumber
 \end{eqnarray}
is called the \textit{layer problem}. Time $\tau$ is the \textit{fast time} and time $t$ is the \textit{slow time}. 
\begin{remark}\remlab{hatyRemark}
{The re-scaling used in \eqref{yTilde} is necessary to identify the slow-fast structure hidden in \eqref{Xeps}. It is not due to loss of hyperbolicity and so the use of \eqref{yTilde} is different from the scaling used later in \secref{canards} in connection with the blowup of a nonhyperbolic line. We use \eqref{yTilde} to cover $\hat y\in (-1,1)$. Outside this region we will use \eqref{SotomayorProperty} and the exact PWS analysis to describe the dynamics.

In \cite{Llibre07,llibre_sliding_2008,Llibre09}, the re-scaling \eqref{yTilde} is replaced by a transformation $y=r\bar y,\,\epsilon =r\bar \epsilon$, $(\bar y,\bar \epsilon)\in S^1$, making it possible to illustrate, in one diagram, the dynamics of $X^{\pm}$ outside the region of regularization and the slow-fast dynamics within (by inserting the cylinder $(x,(\bar y,\bar \epsilon))\in \mathbb R\times S^1$ at $y=0$). However it does nothing to address any loss of hyperbolicity that may occur.

}
\end{remark}

In this paper, we apply and extend Fenichel theory \cite{fen1, fen2, fen3, jones_1995} of singular perturbations to study these regularized equations \eqref{slowslowFastEquations1}-\eqref{LayerProblemY2}, allowing us to go from a description of the singular limit $\epsilon=0$ to a description for $\epsilon>0$. This approach has the advantage that it is geometric so, for example, we are able to solve persistence problems by invoking transversality. It will also be important to identify the singular limit $\epsilon=0$ with the original PWS system.

The key to the subsequent analysis in this paper is the following result ({a related result is given in \cite[Theorem 1.1]{llibre_sliding_2008}, in a slightly different form}):
{
\begin{theorem}\thmlab{criticalManifold}
 There exist critical manifolds $S_{a,r}$ and $\tilde q$, given by
 \begin{align}
 S_a&:\quad \frac{1-\phi(\hat y)}{1+\phi(\hat y)}=\frac{x+\mathcal O(x^2+\mu x)}{\beta (x-\mu) +\mathcal O(\mu(x-\mu)+(x-\mu)^2)},\,(x,0)\in {\Sigma}_{sl}^-,\eqlab{Sary2}\\
 S_r&:\quad \frac{1-\phi(\hat y)}{1+\phi(\hat y)}=\frac{x+\mathcal O(x^2+\mu x)}{\beta (x-\mu) +\mathcal O(\mu(x-\mu)+(x-\mu)^2)},\,(x,0)\in {\Sigma}_{sl}^+,\nonumber\\
 \tilde q&:\quad x=0,\,\mu=0,\,\hat y\in (-1,1),\eqlab{qtilde}
 \end{align}
which are normally attracting, normally repelling and normally non-hyperbolic, respectively. On the critical manifolds $S_{a,r}$, the motion of the slow variable $x$ is described by the reduced problem \eqref{ReducedProblemY2}, which coincides with the sliding equations \eqsref{slidingEqns}{sigma}. 
\end{theorem}
\begin{proof}
This follows from simple computations. Recall that the critical manifolds are fixed points of the layer equations \eqref{LayerProblemY2}. Each manifold is normally hyperbolic if the linearization has one non-zero eigenvalue. Otherwise it is non-hyperbolic. 
\end{proof}
\begin{remark}
 The non-hyperbolic line $\tilde q$ of the regularized system corresponds to the two-fold $q$ of the PWS system. 
 \end{remark}
 
 As in our previous paper \cite{KristiansenHogan2015} it is useful to introduce the following function $w=w(\hat y)$:
\begin{align}
w(\hat y)=\frac{1-\phi(\hat y)}{1+\phi(\hat y)},\eqlab{weqn}
\end{align}
which appears on the left hand side of \eqref{Sary2}, since this reduces the complexity of subsequent expressions. 
For $\hat y\in (-1,1)$ we have $w \in (0,\infty)$. Also
 \begin{align}
  w'(\hat y)= -\frac{2\phi'(\hat y)}{(1+\phi(\hat y))^2}<0,\eqlab{wPrime}
 \end{align}
within $\hat y\in (-1,1)$, and the critical manifolds $S_{a,r}$ are therefore graphs of $\hat y$ over $\Sigma_{sl}^{\mp}$:
\begin{align*}
 S_a&:\quad \hat y=w^{-1}\left(\frac{x+\mathcal O(x^2+\mu x)}{\beta (x-\mu) +\mathcal O(\mu(x-\mu)+(x-\mu)^2)}\right),\,(x,0)\in {\Sigma}_{sl}^-,\\
 S_r&:\quad \hat y=w^{-1}\left(\frac{x+\mathcal O(x^2+\mu x)}{\beta (x-\mu) +\mathcal O(\mu(x-\mu)+(x-\mu)^2)}\right),\,(x,0)\in {\Sigma}_{sl}^+,\nonumber
 \end{align*}
 where $$w^{-1}(z) = \phi^{-1}\left(\frac{1-z}{1+z}\right),\quad z\in (0,\infty)$$ is the inverse of $w$.
%
}

In the sequel, we will also need the \textit{extended problem}, which is given by
\begin{align}
x' &=\epsilon \left ( (\delta + \zeta^+ x+\mathcal O(x^2+\mu+\epsilon \hat y))(1+\phi(\hat y))+(\alpha  + \zeta^- x + \mathcal O(x^2+\mu+\epsilon \hat y))(1-\phi(\hat y))\right ),\eqlab{ExtendedProblem1}\\
\hat y' &=(x+\eta^+ x^2 +\epsilon \chi^- \hat y + \mathcal O(\epsilon x\hat y+\mu(x+\epsilon \hat y)+x^3))(1+\phi(\hat y))\nonumber\\
&+(-\beta (x-\mu)+\eta^- (x-\mu)^2 + \epsilon \chi^- \hat y+\mathcal O(\epsilon x\hat y+\mu \epsilon \hat y + \mu(x-\mu)+(x-\mu)^3))(1-\phi(\hat y)),\nonumber\\
\epsilon' & =  0,\nonumber\\
\mu' & = 0.\nonumber
 \end{align}
 Here we have used \eqsref{normplus2}{normminus2}. 
 
 Finally, when studying canards in \secref{canards}, it will also be useful to introduce a new time $\tilde \tau$, defined by
 \begin{align}
  d\tilde \tau = (1+\phi(\hat y))d\tau,\eqlab{tildeTau} 
 \end{align}
so that our extended problem \eqref{ExtendedProblem1}, now only defined within $\hat y\in (-1,1)$, becomes:
\begin{align}
x' &=\epsilon \left (\delta + \zeta^+ x+\mathcal O(x^2+\mu+\epsilon \hat y)+(\alpha  + \zeta^- x + \mathcal O(x^2+\mu+\epsilon \hat y))w(\hat y)\right ),\eqlab{ExtendedProblem1W}\\
\hat y' &=x+\eta^+ x^2 +\epsilon \chi^- \hat y + \mathcal O(\epsilon x\hat y+\mu(x+\epsilon \hat y)+x^3)\nonumber\\
&+(-\beta (x-\mu)+\eta^- (x-\mu)^2 + \chi^- y+\mathcal O(\epsilon x\hat y+\mu \epsilon \hat y + \mu(x-\mu)+(x-\mu)^3))w(\hat y),\nonumber\\
\epsilon' & =  0,\nonumber\\
\mu' & = 0.\nonumber
 \end{align}
 where, by abuse of notation, we take $()'=\frac{d}{d\tilde \tau}$ in \secref{canards} only. Note that $w(\hat y)$ appears on the right hand sides of equations \eqref{ExtendedProblem1W}. 

\begin{remark}\remlab{S0vsSigmaSl}
 If we return to our original $y$ variable using \eqref{yTilde} then the critical manifolds $S_{a,r}$ become graphs $y=\epsilon h(x,\mu)$ which, for $\epsilon=0$, collapse to $\overline{\Sigma}_{sl}$.
\end{remark}
{
\begin{remark} 
The critical manifolds $S_{a,r}$ and $\tilde q$ are shown in \figref{Sar} for $\beta>0$. It is clear that a singular canard exists for $\mu=0$ only.
\end{remark}
}
\begin{figure}\begin{center}\includegraphics[width=.65\textwidth]{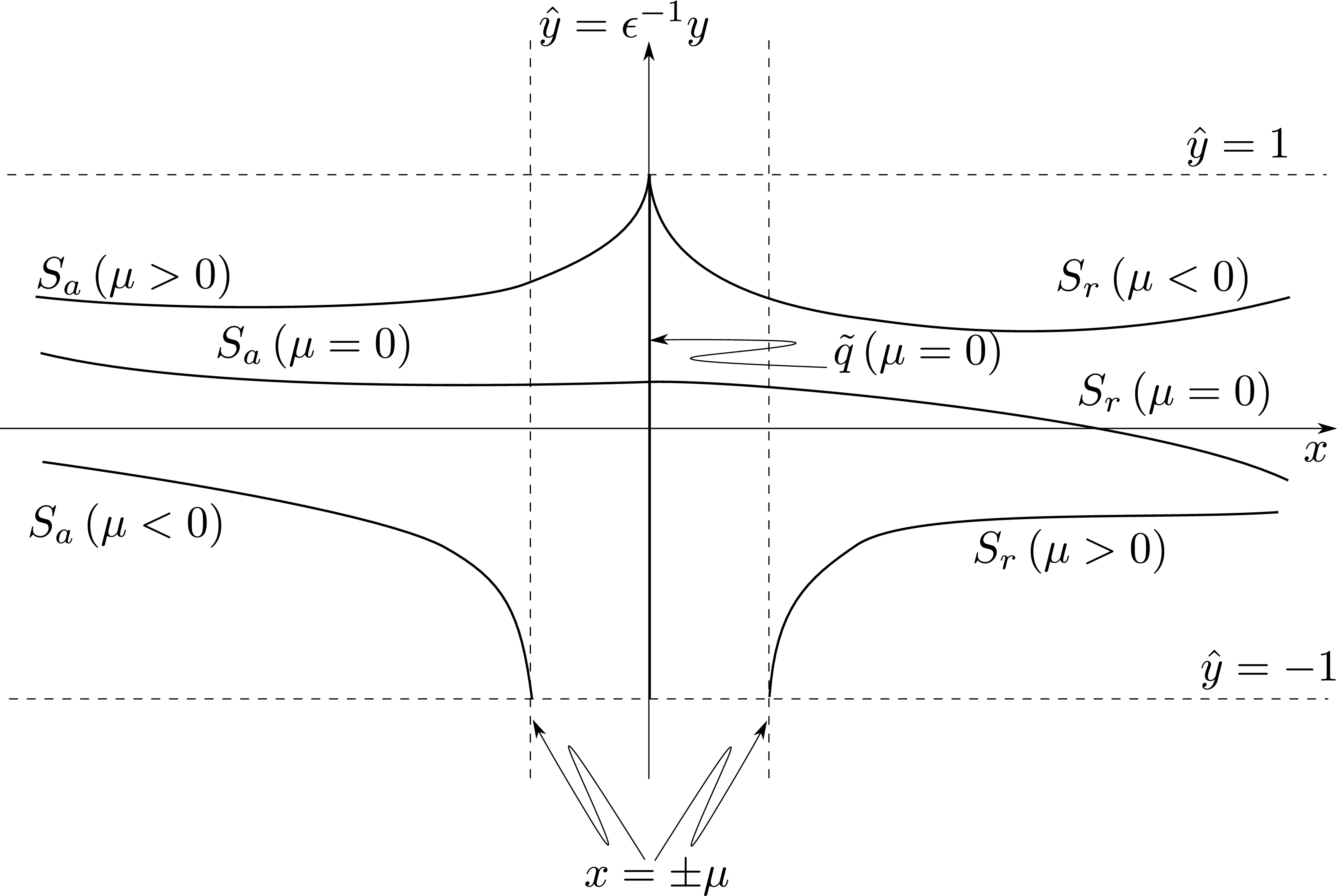}
\caption{The critical manifolds $S_{a,r}$ for different values of $\mu$ together with the non-hyperbolic line $\tilde q$ for $\beta>0$. {The critical manifolds agree with the sliding regions of the PWS smooth system. In particular, stable (unstable) sliding becomes an attracting (repelling) critical manifold of the regularization carrying slow dynamics described by the sliding equations.}}\figlab{Sar}
                		\end{center}
              \end{figure}

\subsection{Fenichel theory}
Consider compact sets $I^\pm$ contained within $\Sigma_{sl}^{\pm}$, respectively. Then according to Fenichel theory  \cite{fen1, fen2, fen3, jones_1995}, there exist slow manifolds of \eqref{Xeps} of the form
\begin{align}
 S_{a,\epsilon}:\quad \hat y = w^{-1}\left(\frac{x+\mathcal O(x^2+\mu x)}{\beta (x-\mu) +\mathcal O(\mu(x-\mu)+(x-\mu)^2)}\right)+\mathcal O(\epsilon),\quad (x,0)\in I^-\subset \Sigma_{sl}^-,\eqlab{Saeps}\\
 S_{r,\epsilon}:\quad \hat y =w^{-1}\left( \frac{x+\mathcal O(x^2+\mu x)}{\beta (x-\mu) +\mathcal O(\mu(x-\mu)+(x-\mu)^2)}\right)+\mathcal O(\epsilon),\quad (x,0)\in I^+\subset  \Sigma_{sl}^+.\eqlab{Sreps}
\end{align}
The flow on $S_{a,\epsilon}$ and $S_{r,\epsilon}$ is therefore $\mathcal O(\epsilon)$-close to the flow of the sliding equations. 
 The potential non-uniqueness of $S_{a,\epsilon}$ and $S_{r,\epsilon}$ only manifests itself in $\mathcal O(e^{-c/\epsilon})$ small deviations. Henceforth we shall fix copies of $S_{a,\epsilon}$ and $S_{r,\epsilon}$. 

\subsection{Loss of hyperbolicity}\seclab{loss}
From \thmref{criticalManifold}, it follows that the analysis in \secref{pws} can be directly applied to singular canards on the critical manifold for the limiting regularized system. 
In contrast, a (maximal) {canard} of the regularized system \eqref{slowFastEquations1} appears as an intersection of the extension by the flow of the Fenichel slow manifolds $S_{a,\epsilon}$ and $S_{r,\epsilon}$ to a vicinity of the non-hyperbolic line $\tilde q$. A canard of the regularized system is called a \textit{vrai} (\textit{faux}) canard if it goes from $S_{a,\epsilon}$ ($S_{r,\epsilon}$) to $S_{r,\epsilon}$ ($S_{a,\epsilon}$) in forward time. 
\begin{remark}\remlab{vraiCanards}
 {Our focus is mainly on \textit{vrai} canards, usually referred to as canards \cite{krupa_extending_2001,szmolyan_canards_2001}, a convention we shall adopt.}
\end{remark}
%

To study the extensions of $S_{a,\epsilon}$ and $S_{r,\epsilon}$ near $\tilde q$, we will use the blowup method of Dumortier and Roussarie \cite{dumortier_1991, dumortier_1993, dumortier_1996}, in the formulation of Krupa and Szmolyan \cite{krupa_extending_2001}. 

Consider the following sections:
\begin{align}
\Lambda^\mp :&\quad x=\mp \rho,\eqlab{Lambdapm}
\end{align}
where $\rho$ is a small positive constant. Then Fenichel's slow manifolds $S_{a,\epsilon}$ and $S_{r,\epsilon}$ intersect $\Lambda^{\mp}$, respectively, in
\begin{align}
\hat y&= w^{-1}(\beta^{-1}) + \mathcal O_{a,r}(\rho+\mu +\epsilon),\eqlab{SarIntersection}
 \end{align}
from \eqsref{Saeps}{Sreps}. This will become useful later when extending $S_{a,\epsilon}$ and $S_{r,\epsilon}$.
 
We are also interested in studying the regularization of pseudo-equilibria and limit cycles. The non-hyperbolicity of the line $\tilde q$ defined in \eqref{qtilde} will again complicate the analysis and we will also use the blowup method to obtain an accurate and complete description of the regularization of these phenomena.

\section{Main results}\seclab{mainResult}
In this section, we anticipate our main regularization results, in a form convenient for the reader. The connection between sections of the paper devoted to PWS results and to regularization results is shown in \tabref{tblMR}:
\begin{table}[h!]
\caption{Connection between phenomena in the PWS and regularized systems}\tablab{tblMR}
\centering
\begin{tabular}{| c | c |}
 \hline
PWS phenomenon & Regularized phenomenon \\
\hline
\hline
Singular canards (\secref{Singularcanards}) & Canards (\secref{canards})\\
\hline
Pseudo-equilibria (\secref{pseudo}) & Equilibria (\secref{equilibriaregular})\\
\hline
Limit cycles (\secref{limitCyclesPWS}) & Limit cycles (\secref{limitCycles})\\
\hline
\end{tabular}
\end{table}

\subsection{Canards}
From \secref{Singularcanards}, singular canards of the PWS system can only exist for $\mu=0$. From \propref{canards}, singular canards exist for 
\begin{align}
  \beta>0\quad \text{and}\quad \Omega>0.\eqlab{canardAssumption}
 \end{align}
Canards of the regularized system \eqref{Xeps} are considered in \secref{canards}. The main result of that section is that singular canards survive the regularization in the following sense:

\textsc{theorem} \ref{canardTheorem}
\textit{Assuming \eqref{canardAssumption}, the regularized system \eqref{Xeps}  has a maximal canard at $$\mu =\sqrt{\epsilon} \mu_{2,c}(\sqrt{\epsilon})=\mathcal O(\epsilon),$$ where $\mu_{2,c}$ is given by \eqref{muc}, which is $\mathcal O(\sqrt{\epsilon})$-close to the singular canard.}
\begin{remark}\remlab{maximal}
{Canards are non-unique because the slow manifolds are non-unique. But often in the literature, \textit{any} solution following $S_{r,\epsilon}$ for an $O(1)$-distance is called a canard. Hence the canard in Theorem (\ref{canardTheorem}), which is obtained geometrically as the transverse intersection of fixed copies of $S_{a,\epsilon}$ and $S_{r,\epsilon}$, is referred to as the \textit{maximal canard}.}
\end{remark}

\subsection{Equilibria}
From \secref{pseudo}, pseudo-equilibria of the PWS system can only exist for
\begin{align}
 \alpha\delta<0.\eqlab{eqAssymption}
\end{align}
Equilibria of the regularized system \eqref{Xeps} are considered in \secref{equilibriaregular}. We have two main results. Our first result is Theorem~\ref{mainHopf}, which shows that, for $\Omega<0$, the equilibria are saddles and, for $\Omega>0$, there is a Hopf bifurcation at $\mu=\mathcal O(\sqrt{\epsilon})$, after which the equilibria become foci, and then nodes, as $\mu$ varies. 

\textsc{theorem} \ref{mainHopf}
\textit{Assuming \eqref{eqAssymption}, the regularized system  \eqref{Xeps} has a smooth and locally unique family of equilibria
\begin{align*}
(x,y)=(x^*,y^*)(\mu,\sqrt{\epsilon}) \equiv \left(\frac{\mu}{1+\beta^{-1} \delta \alpha}+\mathcal O(\mu^2+\epsilon),\mathcal O(\epsilon)\right)
\end{align*}
where $(x^*(\mu,0),y^*(\mu,0))=(x^*(\mu,0),0)$ agrees with the family of pseudo-equilibria for the PWS system (see \propref{pwsProposition}) and, in particular, $\partial_\mu x^*(\mu,0)\ne 0$. For
\begin{itemize}
 \item[$\Omega<0$:] The family of equilibria consists of saddles and does not undergo any bifurcation. 
 \item[$\Omega>0$:] The family of equilibria undergoes a Hopf bifurcation at $\mu=\sqrt{\epsilon}\mu_{2,H}(\sqrt{\epsilon})=\mathcal O(\epsilon)$, where $\mu_{2,H}$ is given by \eqref{mu2H}. The first Lyapunov coefficient is given by $a = a_2 \sqrt{\epsilon}+\mathcal O(\epsilon)$
where $a_2$, given by \eqref{K2}, depends upon the regularization function $\phi$.
 \end{itemize}
}

Our second result for equilibria of the regularized system \eqref{Xeps}, Theorem~\ref{K2sign}, is perhaps surprising, in that it shows that the criticality of the Hopf bifurcation depends on the regularization function $\phi$.

\textsc{theorem} \ref{K2sign}
 \textit{Suppose that $\Omega>0$ so that, from Theorem~\ref{mainHopf}, there exists a Hopf bifurcation. Then for $\epsilon$ sufficiently small, provided
  \begin{align}
\delta (\zeta^-+\chi^-)-\alpha (\zeta^++\chi^+)\ne 0, \tag*{\eqref{assumption1}}
  \end{align}
  where $\zeta^\mp, \chi^\mp$ are the higher order coefficients in \eqref{zetaeta}, the first Lyapunov coefficient $a$ can be positive, zero or negative, depending on the regularization function $\phi$.}

\subsection{Limit cycles}
From \secref{limitCyclesPWS}, limit cycles of the PWS system can exist for the case $II_2$, which occurs when 
\begin{align}
 \delta = -1,\quad \alpha>0,\quad \beta<0.\eqlab{limitCyclePWSAssumption}
 \end{align}
Limit cycles of the regularized system \eqref{Xeps} are considered in \secref{limitCycles} where, following Theorem~\ref{mainHopf} above, we find limit cycles provided:
\begin{align}
\Omega>0\quad  \text{and}\quad \alpha\delta<0.\eqlab{limitCycleAssumption}
\end{align}
This leads to the regularization of $II_2$ and of $VI_3$, which we denote by $II_2^\epsilon$ and $VI_3^\epsilon$, respectively. We have two main results for limit cycles of the regularized system. The first main result Theorem~\ref{mainLimitCycles} shows how small amplitude periodic orbits due the Hopf bifurcation in Theorem~\ref{mainHopf} can be connected to $\mathcal O(1)$ (with respect to $\epsilon$) amplitude periodic orbits.

\textsc{theorem} \ref{mainLimitCycles} 
\textit{For $\epsilon$ sufficiently small:
 \begin{itemize}
  \item[$II_2^\epsilon$:] There exists a $C^k$-smooth family of locally unique periodic orbits of the regularized system \eqref{Xeps} that is due to the Hopf bifurcation in Theorem~\ref{mainHopf}. If $a_2< 0$ ($a_2>0$) where $a_2$ is the first Lyapunov coefficient as defined in \eqref{K2}, then the periodic orbits are attracting (repelling) near the Hopf bifurcation. If 
  \begin{align}
  \Delta_{II_2}=\frac{2}{3\alpha \beta} \left(\alpha (\eta^-+\beta \eta^+)+\beta(\zeta^-+\chi^-+\alpha (\zeta^++\chi^+))\right) \tag*{\eqref{Ceqn}},
  \end{align} is negative (positive) then the periodic orbits for $\Delta_{II_2}^{-1} \mu \le - c\sqrt{\epsilon}$, $c>0$ sufficiently large, are attracting (repelling).
  The periodic orbits for $\mu=\mathcal O(1)$ (with respect to $\epsilon$) are continuously $\mathcal O(\epsilon)$-close to the PWS periodic orbits in \propref{propLimitCyclesPWS}.
  \item[$VI_3^\epsilon$:] There exists a $C^k$-smooth family of \textnormal{small} periodic orbits of the regularized system \eqref{Xeps} that is due to the Hopf bifurcation in Theorem~\ref{mainHopf}. There also exists a $C^k$-smooth family of periodic orbits that are $\mathcal O(1)$ (with respect to $\epsilon$) in amplitude and which undergo a canard explosion, where the amplitude changes by $\mathcal O(1)$ within an exponentially small parameter regime around the canard value $\mu=\sqrt{\epsilon}\mu_{2,c}(\sqrt{\epsilon})$ (see Theorem~\ref{canardTheorem}). If $a_2< 0$ ($a_2>0$) where $a_2$ is the first Lyapunov coefficient as defined in \eqref{K2}, then the periodic orbits are attracting (repelling) near the Hopf bifurcation. If 
 \begin{align}
 \Delta_{VI_3}=\frac{2}{3\alpha \beta}\left(\alpha (\eta^-+ \beta \eta^+)+\beta(\beta+1)(\zeta^-+\chi^-)\right) \tag*{\eqref{Beqn}}
 \end{align} 
 is negative (positive) then the $\mathcal O(1)$-periodic orbits are attracting (repelling).
 \end{itemize}
}


We conjecture on the connection of the two families of periodic orbits in $VI_3^\epsilon$:

\textsc{Conjecture} \ref{mainConjecture1}
\textit{The two families of periodic orbits in $VI_3^\epsilon$  belong to the same family of locally unique periodic orbits.}

The second main result for limit cycles of the regularized system \eqref{Xeps} shows how an open set of regularization functions can induce at least one saddle-node bifurcation in the periodic orbits of Theorem~\ref{mainLimitCycles}. 

\textsc{theorem} \ref{saddleNode}
 \textit{Suppose \eqref{assumption1} and \eqref{limitCycleAssumption}. Then for $\epsilon$ sufficiently small:
  \begin{itemize}
   \item[$II_2^\epsilon$:] There exists an open set of regularization functions such that the periodic orbits in Theorem~\ref{mainLimitCycles}, case $II_2^\epsilon$, undergo at least one saddle-node bifurcation.
   \item[$VI_3^\epsilon$:] Suppose, in addition, that Conjecture~\ref{mainConjecture1} holds. Then there exists an open set of regularization functions such that the periodic orbits in Theorem~\ref{mainLimitCycles}, case $VI_3^\epsilon$, undergo at least one saddle-node bifurcation.
  \end{itemize}
}

In \secref{numerics}, we illustrate the results in Theorem~\ref{saddleNode} by applying two different regularization functions to two model systems of $II_2^\epsilon$ and $VI_3^\epsilon$. The regularization functions are such that each case, $II_2^\epsilon$ or $VI_3^\epsilon$, will have a saddle-node bifurcation for only one of the regularization functions.

\section{On the existence of canards}\seclab{canards}
In \secref{Singularcanards}, \propref{canards}, we showed that our PWS system \eqsref{normplus}{normminus} together with the sliding vector field \eqsref{slidingEqns}{sigma} contains singular canards for $\beta>0$. In this section, we aim to discover the fate of these singular canards in the regularized system. {To do this, we focus on dynamics in the region of regularization $y\in (-\epsilon,\epsilon)$, or $\hat y\in (-1,1)$, as described by equations \eqref{ExtendedProblem1W}, which are written in terms of the new time $\tilde \tau$, defined in \eqref{tildeTau}. We work with $\tilde \tau$ in this section only.}

As discussed in \secref{loss}, Fenichel theory breaks down on the non hyperbolic line $\tilde q$, defined in \eqref{qtilde}. We use the blowup method \cite{dumortier_1991,dumortier_1993,dumortier_1996} to deal with this line. We introduce the \textit{quasi-homogeneous blowup}, given by $$(x,\epsilon,\mu)=(r^{a_1} \overline x,r^{a_2} \overline \epsilon,r^{a_3}\overline{\mu}),$$ where the number $r$ is called the \textit{exceptional divisor}. By this transformation, the line $\tilde q$ is blown up to a cylinder $$S^2\times (-1,1) = \{((\bar x, \bar \epsilon, \bar \mu),\hat y) | \bar x^2 + \bar \epsilon^2 + \bar \mu^2 =1,\,\hat y\in (-1,1)\}.$$ When $r=0$, the blown-up coordinates collapse to the non-hyperbolic line $\tilde q$. 

The weights $(a_1,a_2,a_3)$ are chosen so that the vector field written as a function of the blowup coordinates has a power of the exceptional divisor as a common factor. By transforming time using this common factor, it is then possible to remove the exceptional divisor and so de-trivialize the vector-field on $\tilde q$. By substituting the quasi-homogeneous blowup into \eqref{ExtendedProblem1W} and removing the exceptional divisor, it turns out that $(a_1,a_2,a_3)=(1,2,1)$. So we have the following blowup of $\tilde q$: $$x=r\bar x,\,\epsilon = r^2 \bar \epsilon,\,\mu=r \bar \mu$$ with $r\ge 0$, $(\bar x,\bar \epsilon,\bar \mu)\in S^2$. Note that this blowup does not depend on $\hat y$. The new phase space is therefore $$((\bar x,\bar \epsilon,\bar \mu),r,\hat y)\in S^2\times \overline{\mathbb R}_+\times (-1,1).$$

To describe the dynamics on the blowup space we consider the following charts:
 \begin{align}
  \text{chart}\quad &\kappa_1:\,\bar x=-1,\, x=-r_1,\,\epsilon = r_1^2 \epsilon_1,\,\mu = r_1 \mu_1\eqlab{eqChartKappa1}\\
  \text{chart}\quad &\kappa_2:\,\bar \epsilon=1,\, x=r_2 x_2,\,\epsilon = r_2^2,\,\mu = r_2 \mu_2,\eqlab{eqChartKappa2}\\
  \text{chart}\quad &\kappa_3:\,\bar x=1,\quad x=r_3,\, \,\epsilon = r_3^2 \epsilon_3,\,\mu = r_3 \mu_3.\eqlab{eqChartKappa3}
 \end{align}
The chart $\kappa_2$ is called the scaling chart or {\it family rescaling} chart \cite{dumortier_1991,dumortier_1993,dumortier_1996}. The charts $\kappa_{1,3}$ are called {\it phase directional} charts. The point $(x,\epsilon,\mu)=(0,0,0)$ has been blown up into the planes $r_i=0 : i=1,2,3$. 
{The two-sphere $S^2$ and the charts $\kappa_{1-3}$ are shown in \figref{s2AndCharts}.} We adopt the convention that the subscript $n$ of each  quantity is used when we are working in chart $\kappa_n$.
\begin{figure}\begin{center}\includegraphics[width=.65\textwidth]{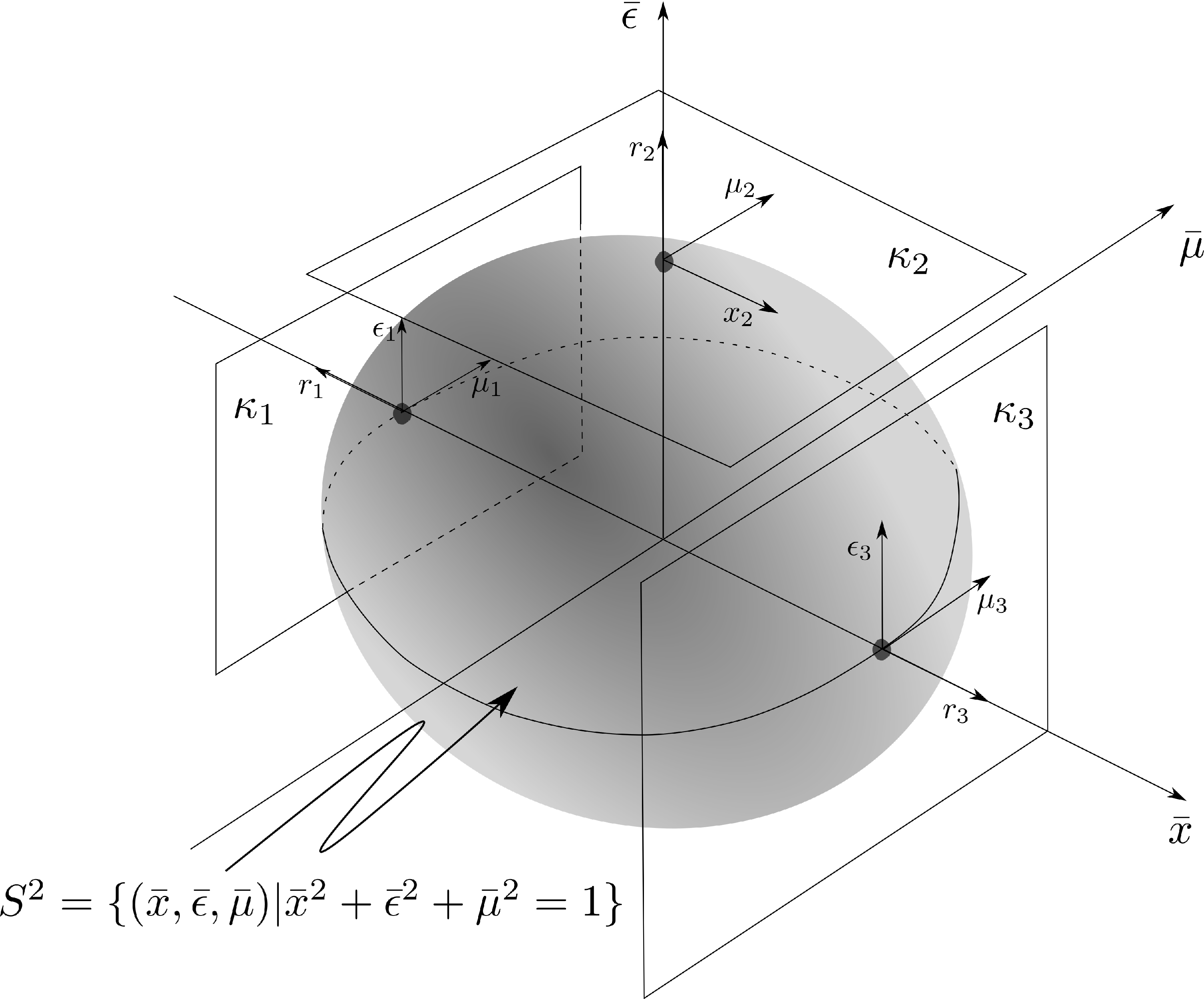}
\caption{The two-sphere $S^2$ and the charts $\kappa_{1-3}$.}\figlab{s2AndCharts}
                		\end{center}
              \end{figure}



The following coordinate change between charts $\kappa_1$ and $\kappa_2$ will be important in what follows:
\begin{eqnarray}
\kappa_{21}&:(r_1,\hat y,\epsilon_1,\mu_1)\mapsto (x_2,\hat y,\mu_2,r_2) =&(-\epsilon_1^{-1/2},\,\hat y,\,{\epsilon_1}^{-1/2}\mu_1,\,r_1 \sqrt{\epsilon_1}),\eqlab{kappa12}\\
\kappa_{12}&:(x_2,\hat y,r_2,\mu_2)\mapsto (r_1,\hat y,\epsilon_1,\mu_1)=&(-r_2x_2,\,\hat y,\,x_2^{-2},-x_2^{-1}\mu_2),\eqlab{kappa21}
\end{eqnarray}
defined for $x_2<0$ and $\epsilon_1>0$ respectively. 
The change between charts $\kappa_2$ and $\kappa_3$ is given by:
\begin{eqnarray}
\kappa_{32}&:(x_2,\hat y,r_2,\mu_2)\mapsto (r_3,\hat y,\epsilon_3,\mu_3)=&(r_2x_2,\,\hat y,\,x_2^{-2},x_2^{-1}\mu_2),\eqlab{kappa23}\\
\kappa_{23}&:(r_3,\hat y,\epsilon_3,\mu_3)\mapsto (x_2,\hat y,\mu_2,r_2) =&(\epsilon_3^{-1/2},\,\hat y,\,{\epsilon_3}^{-1/2}\mu_3,\,r_3 \sqrt{\epsilon_3}),\eqlab{kappa32}
\end{eqnarray}
defined for $x_2>0$ and $\epsilon_3>0$ respectively. 

We now describe the dynamics in each chart, beginning with the chart $\kappa_2$. \Secsref{chartKappa1}{chartKappa3} describe the dynamics in charts $\kappa_{1,3}$. Finally, \secref{combination} combines the results from charts $\kappa_{1-3}$ to prove Theorem~\ref{canardTheorem}.
\subsection{Chart $\kappa_2$}\seclab{chartKappa2}
The extended problem \eqref{ExtendedProblem1W} written in chart $\kappa_2$ becomes 
\begin{align}
 \dot x_2 &=\delta +\alpha w(\hat y)+ r_2(\zeta^+x_2+\zeta^-x_2 w(\hat y))+\mathcal O(r_2(\mu_2+r_2)),\eqlab{eqnChartK2}\\
 \dot{\hat y} &=x_2+r_2\left(\eta^+x_2^2+\chi^+ \hat y\right),\nonumber\\
 &+\left(-\beta (x_2-\mu_2)+r_2\left(\eta^-(x_2-\mu_2)^2+\chi^-\hat y\right)\right)w(\hat y)+\mathcal O(r_2(\mu_2+r_2))\big),\nonumber\\
 \dot r_2 &=0,\nonumber\\
 \dot \mu_2&=0,\nonumber
\end{align}
where we have multiplied time by $r_2$. 
For $\beta>0$ there is an invariant line $l_{2}$ given by: 
\begin{align}
l_{2}:\,x_2\in \mathbb R,\, \hat y=\hat y_{c}, \,r_2=0, \, \mu_2=0,\eqlab{l2Solution}
\end{align}
where
\begin{align}
 \hat y_{c} = w^{-1}(\beta^{-1}) = \phi^{-1}\left(\frac{\beta -1}{\beta+1}\right)\in (-1,1),\eqlab{phiy2c}
\end{align}
carrying the \textit{special} solution:
\begin{align*}
 x_2 = \beta^{-1}\Omega t_2,\, \hat y=\hat y_{c},\,r_2=0,\,\mu_2=0,
\end{align*}
where $t_2$ is the time used in \eqref{eqnChartK2}.

Consider the following sections 
\begin{align}
\Lambda_2^{\mp}=\{(x_2,\hat y,r_2,\mu_2)\vert \,x_2 = \mp \nu^{-1}\},\eqlab{Lambda2pm}
\end{align}
where $\nu$ is small and positive. The line $l_2$ intersects $\Lambda_2^{\mp}$ in 
\begin{align}
 l_2 \cap \Lambda_2^{\mp}:\, x_2 = \mp \nu^{-1},\,\hat y=\hat y_c,\,r_2=0,\,\mu_2=0.\eqlab{l2Lambdapm}
\end{align}
The line $l_2$ and the sections $\Lambda_2^{\mp}$ are shown in \figref{kappa2}.

\begin{figure}\begin{center}\includegraphics[width=.85\textwidth]{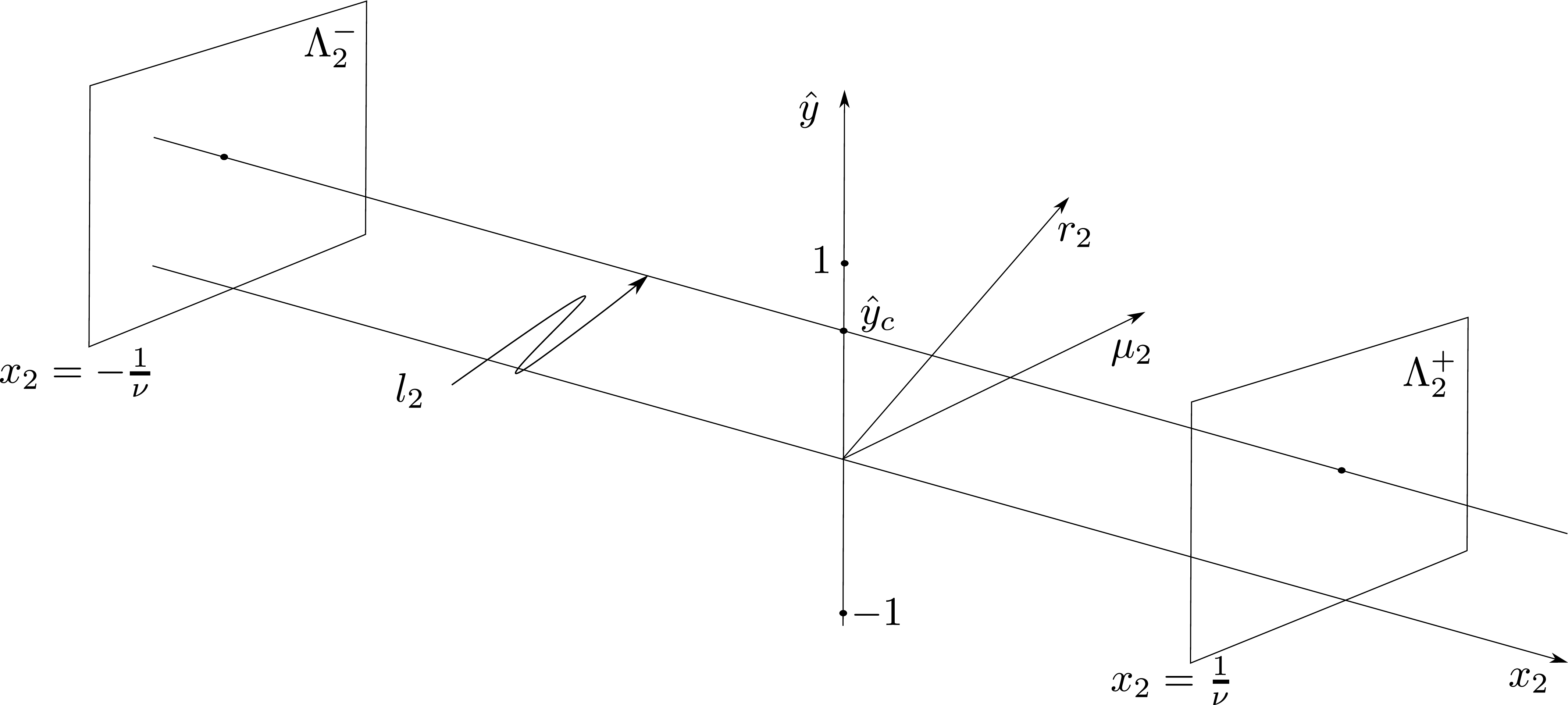}
\caption{The line $l_2$ and the sections $\Lambda_2^{\pm}$.}\figlab{kappa2}
                		\end{center}
              \end{figure}


%
%



\subsection{Chart $\kappa_1$}\seclab{chartKappa1}
We now describe the dynamics in chart $\kappa_1$ and relate them to the dynamics in chart $\kappa_2$. The extended problem \eqref{ExtendedProblem1W} written in chart $\kappa_1$ becomes
\begin{align}
\dot r_1 &=-r_1 \epsilon_1 F_1(r_1,\hat y,\epsilon_1,\mu_1),\eqlab{EqsChartKappa1}\\
\dot{\hat y} &= -1+\beta(1+ \mu_1)w(\hat y)+\mathcal O(r_1),\nonumber\\
\dot \epsilon_1 &=2\epsilon_1^2 F_1(r_1,\hat y,\epsilon_1,\mu_1),\nonumber\\
\dot \mu_1 &= \epsilon_1 \mu_1 F_1(r_1,\hat y,\epsilon_1,\mu_1),\nonumber
\end{align}
where we have divided the vector-field by $r_1$ and set $F_1(r_1,\hat y,\epsilon_1,\mu_1) = \delta+\alpha w(\hat y)+\mathcal O(r_1)$. The $\mathcal O(r_1)$-terms include the constants $\zeta^\pm,\,\eta^\pm,\,\chi^\pm$ from \eqref{zetaeta} but they will not play a role in this section and we therefore suppress them.
The line $r_1=0,\,\hat y=w^{-1}\left(\frac{1}{\beta(1+\mu_1)}\right),\,\epsilon_1=0$ is a line of fixed points, provided $\beta(1+\mu_1)>0$, since $w(\hat y)\in (0,\infty)$. In particular, this line includes the point $r_1=0,\,\hat y=\hat y_c,\,\epsilon_1=0,\,\mu_1=0$ because we focus on $\beta>0$. 
There exists an attracting center manifold $M_{a,1}$ of this line, given by:
$$M_{a,1}:\,\hat y=w^{-1}\left(\frac{1}{\beta(1+\mu_1)}\right)+\mathcal O(r_1+\epsilon_1)$$ 
and, within $r_1=0$, a center manifold $C_{a,1}$, given by: 
$$C_{a,1}:\,\hat y=w^{-1}\left(\frac{1}{\beta(1+\mu_1)}\right)+\mathcal O(\epsilon_1).$$ 
Within $C_{a,1}$, there exists an invariant line:
\begin{align*}
l_{a,1}:\, r_1=0,\,\hat y = \hat y_c,\,\epsilon_1\ge 0,\,\mu_1=0.
\end{align*}
Recall \eqref{phiy2c}. 
The center manifold $C_{a,1}$ is overflowing and hence unique near $l_{a,1}$ if $\dot \epsilon_1>0$ for $\epsilon_1>0$. From \eqref{EqsChartKappa1},  we therefore have $$\dot \epsilon_1 =2\epsilon_1^2 (\delta+\alpha w(\hat y))=2\epsilon_1^2 \beta^{-1}\Omega+\mathcal O(\epsilon_1),$$ with $r_1=0$ and $w(\hat y)=\beta^{-1}+\mathcal O(\epsilon_1)$ near $l_{a,1}$. Hence since $\beta>0$, $\dot \epsilon_1>0$ for $\epsilon_1>0$ if $\Omega>0$. For $\Omega<0$, we have $\dot \epsilon_1<0$ and so $C_{a,1}$ is non-unique near $l_{a,1}$ in this case.

The manifold $M_{a,1}$ has invariant foliations, which we denote by $M_{a,1}(\epsilon)$, with $\epsilon=r_1^2\epsilon_1 =\text{const}$. The sub-manifold $M_{a,1}(\epsilon)$ intersects the section $r_1=\rho$, corresponding to $\Lambda^-$ \eqref{Lambdapm}, in 
\begin{align*}
\hat y=w^{-1}(\beta^{-1})+\mathcal O(\rho+\mu+\epsilon)=\hat y_c+\mathcal O(\rho+\mu+\epsilon),
\end{align*}
which agrees with \eqref{SarIntersection} since $\mu=r_1\mu_1=\rho \mu_1$. Hence $M_{a,1}(\epsilon)$ is the extension of $S_{a,\epsilon}$ into chart $\kappa_1$ near the line $\tilde q$ (ignoring exponentially small terms). 

In order to relate the dynamics in chart $\kappa_1$ to the dynamics in chart $\kappa_2$, we use the coordinate change $\kappa_{21}$ in \eqref{kappa21}. The section $\Lambda_2^{-}$ defined in \eqref{Lambda2pm} then becomes 
\begin{align}
 \Lambda_1^-:\quad \epsilon_1 = \nu^2.\eqlab{Lambda2pmChartK1}
\end{align}
The manifold $C_{a,1}$ intersects $\Lambda_1^-$ in
\begin{align*}
 C_{a,1}\cap \Lambda_1^-:\,\epsilon_1=\nu^2,\,r_1=0,\hat y=\hat y_c+\mathcal O(\nu^2+\mu_1)
\end{align*}
and hence, by the conservation of $\epsilon$, we conclude that the intersection of $M_{a,1}(\epsilon)$ with $\Lambda_1^-$ is $\mathcal O(\sqrt{\epsilon})$-close to $C_{a,1}$.

The line $l_{a,1}$ and the section $\Lambda_1^-$ are shown in \figref{kappa1}.
\begin{figure}\begin{center}\includegraphics[width=.65\textwidth]{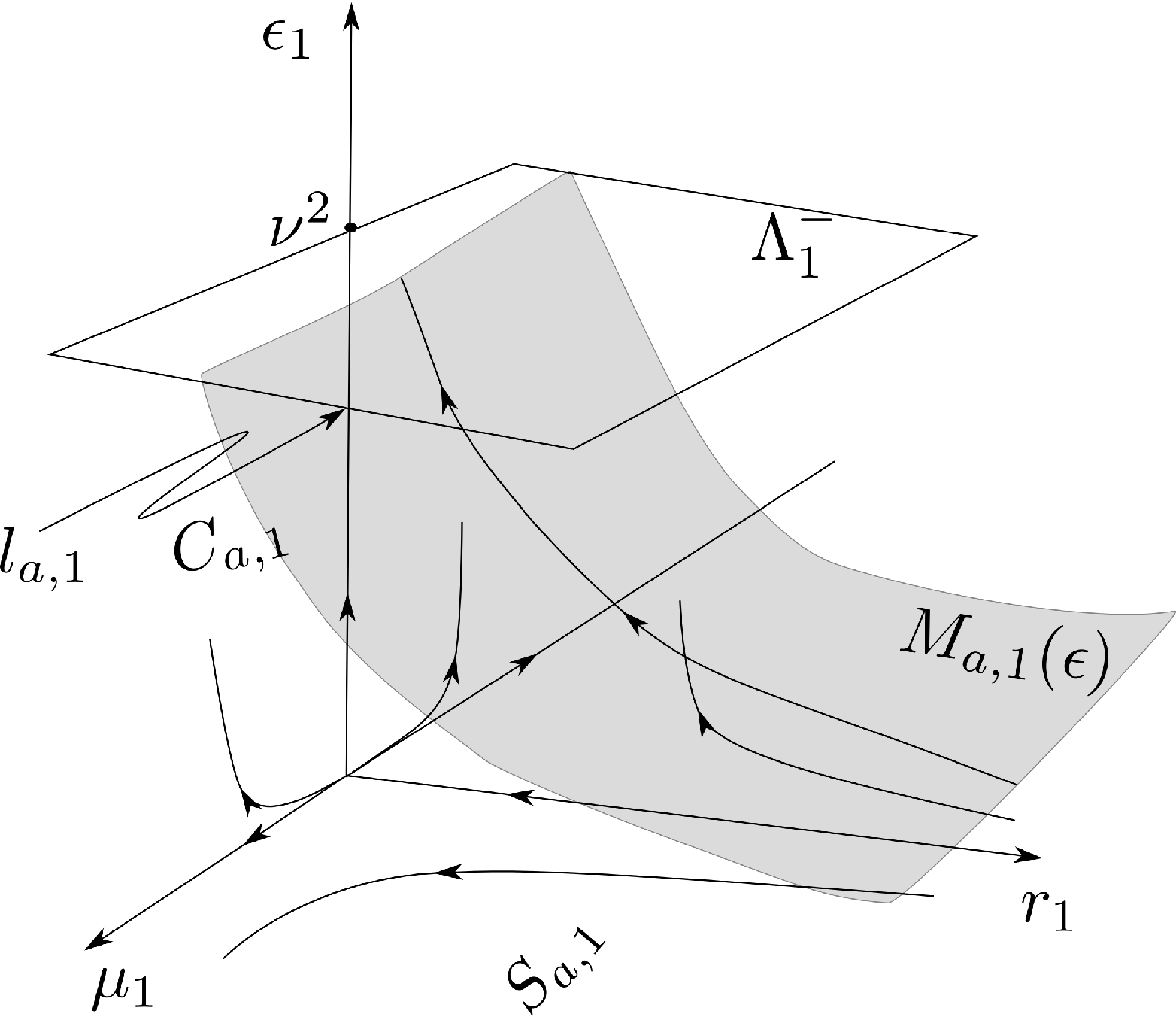}
\caption{The line $l_{a,1}$ and the section $\Lambda_1^-$.}\figlab{kappa1}
                		\end{center}
              \end{figure}

\subsection{Chart $\kappa_3$}\seclab{chartKappa3}
The extended problem \eqref{ExtendedProblem1W} written in chart $\kappa_3$ becomes
\begin{align}
 \dot r_3 &=r_3 \epsilon_3F_3(r_3,\hat y,\epsilon_3,\mu_3),\eqlab{EqsChartKappa3}\\
 \dot{\hat y} &=1-\beta(1-\mu_3)w(\hat y)+\mathcal O(r_3),\nonumber\\
 \dot \epsilon_3 &=-2\epsilon_3^2F_3(r_3,\hat y,\epsilon_3,\mu_3),\nonumber\\
 \dot \mu_3&=-\epsilon_3\mu_3 F_3(r_3,\hat y,\epsilon_3,\mu_3),\nonumber
 \end{align}
where we have divided the vector-field by $r_3$ and set $F_3(r_3,\hat y,\epsilon_3,\mu_2) = \delta+\alpha w(\hat y)+\mathcal O(r_3)$. 
The analysis of the dynamics in this chart is very similar to the analysis in chart $\kappa_1$. Therefore we state the main results only.
 
There exists a repelling center manifold $M_{r,3}$ of $r_3=0,\,\hat y=w^{-1}\left(\frac{1}{\beta(1-\mu_3)}\right),\,\epsilon_3=0$ given by
$$M_{r,3}:\,\hat y=w^{-1}\left(\frac{1}{\beta(1-\mu_3)}\right)+\mathcal O(r_3+\epsilon_3)$$
and a center manifold $C_{r,3}$ within $r_3=0$ which takes the form 
$$C_{r,3}:\,\hat y=w^{-1}\left(\frac{1}{\beta(1-\mu_3)}\right)+\mathcal O(\epsilon_3). $$ 
Within $C_{r,3}$ there exists an invariant line
\begin{align*}
l_{r,3}:\, r_3=0,\,\hat y = \hat y_c,\,\epsilon_3\ge 0,\,\mu_3 =0.
\end{align*}
The center manifold $C_{r,3}$ is overflowing and hence unique near $l_{r,3}$ for $\Omega>0$. If $\Omega<0$ then $C_{r,3}$ is non-unique near $l_{a,3}$.
The sub-manifold $M_{r,3}(\epsilon)=M_{r,3}\cap \{r_3^2\epsilon_3=\epsilon\}$ is the extension of $S_{r,\epsilon}$ into chart $\kappa_3$ near the line $\tilde q$. 
Also $M_{r,3}(\epsilon)$ intersects 
\begin{align}
 \Lambda_3^+:\quad \epsilon_3 = \nu^2,\eqlab{Lambda3Plus}
\end{align}
which corresponds to $\Lambda_2^+$ through the coordinate change $\kappa_{23}$ in \eqref{kappa23}, $\mathcal O(\sqrt{\epsilon})$-close to the intersection of $C_{r,3}$ with \eqref{Lambda3Plus}
\begin{align*}
 C_{r,3}\cap \Lambda_3^+:\,\epsilon_3=\nu^2,\,r_3=0,\hat y=\hat y_c+\mathcal O(\nu^2+\mu_3).
\end{align*}
The line $l_{r,3}$ and the section $\Lambda_3^+$ are shown in \figref{kappa3}.
\begin{figure}\begin{center}\includegraphics[width=.65\textwidth]{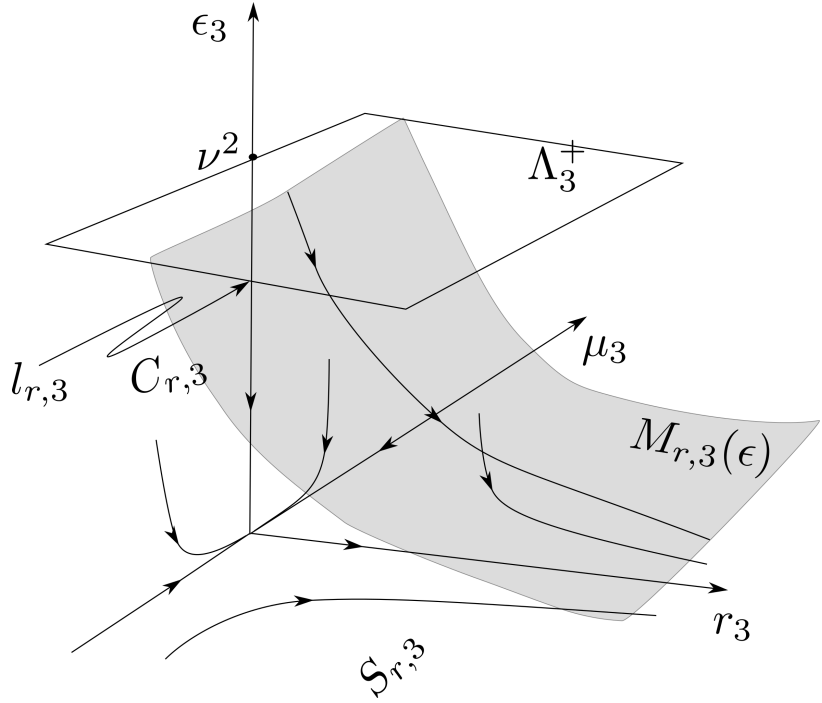}
\caption{The line $l_{r,3}$ and the section $\Lambda_3^+$.}\figlab{kappa3}
                		\end{center}
              \end{figure}

\subsection{Persistence of canards}\seclab{combination}
In this section, we combine results from the three charts $\kappa_{1-3}$ to show how singular canards in our PWS system can survive the regularization. For singular canards, there were two cases to consider: $\Omega>0$ (vrai singular canards) and $\Omega<0$ (faux singular canards). We now consider the same two cases for canards in the regularized system.

For $\Omega>0$, $C_{a,1}$ and $C_{r,3}$ are unique as center manifolds. By transforming $l_2 \cap \Lambda_2^{\mp}$ of \eqref{l2Lambdapm} into charts $\kappa_1$ and $\kappa_3$ we conclude that $\kappa_{12}(l_2)=l_{a,1}\subset C_{a,1}$ for $x_2\ll 0$ and $\kappa_{32}(l_2)=l_{r,3}\subset C_{r,3}$ for $x_2\gg 0$. It is therefore possible to extend $C_{a,1}$ and $C_{r,3}$ into $\kappa_2$ as invariant manifolds $C_{a,2}$ and $C_{r,2}$ by using the invariant line $l_2$ as a guide. Then from the analyses in \secref{chartKappa1} and \secref{chartKappa3}, we conclude that the manifolds $M_{a,2}(\epsilon) = \kappa_{21}(M_{a,1}(\epsilon))$ and $M_{r,2}(\epsilon) = \kappa_{23}(M_{r,3}(\epsilon))$ intersect $\Lambda_{2}^\mp$ a distance $\mathcal O(r_2)$-close  to the intersections of $C_{a,2}$ and $C_{r,2}$, respectively, where $r_2=\sqrt{\epsilon}$ by \eqref{eqChartKappa2}. By applying regular perturbation theory in chart $\kappa_2$ the manifolds $M_{a,2}(\epsilon)$ and $M_{r,2}(\epsilon)$ can be continued all the way up to $x_2=0$ where they remain $\mathcal O(r_2)$-close to $C_{a,2}$ and $C_{r,2}$, respectively. 

The manifolds $C_{a,2}$ and $C_{r,2}$ intersect along $l_2$. We now investigate whether the intersection is transverse. If so, we can conclude that $M_{a,2}(\epsilon)$ and $M_{r,2}(\epsilon)$ also intersect transversally $\mathcal O(r_2)$-close to $l_2$, since $M_{a,2}(\epsilon)$ and $M_{r,2}(\epsilon)$ are $\mathcal O(r_2)$-close to $C_{a,2}$ and $C_{r,2}$, respectively. 
We first eliminate time from \eqref{eqnChartK2} by division with $\dot x_2$ and re-writting our equations in terms of $w$ instead of $\hat y$ using \eqref{weqn}. 
Then we let $(u(x_2),\tilde \mu_2(x_2),\tilde r_2(x_2))$ denote the variations of $(w,\mu_2,r_2)$ about $l_2:\,(w,\mu_2,r_2)=(\beta^{-1},0,0)$. This gives
\begin{align}
 \frac{du}{dx_2}& = \lambda(\beta x_2 u-\tilde \mu_2+\tilde r_2G_2(x_2,u)),\eqlab{var}\\
 \frac{d\tilde \mu_2}{dx_2}&=0,\nonumber\\
 \frac{d\tilde r_2}{dx_2} &=0,\nonumber
\end{align}
where 
\begin{align}
 \lambda &= \frac{(\beta+1)^2 \phi_{1,c}}{2\beta\Omega },\quad
 G_2(x_2) = \beta^{-1} \left (x_2^2 (\eta^-+\beta \eta^+)- \hat y_{c} (\chi^-+\beta \chi^+)\right),\eqlab{lambda}
\end{align}
and
\begin{align*}
 \phi_{1,c} \equiv \phi'(\hat y_{c}) >0.
\end{align*}

For simplicity we now drop the tildes from $\mu_2$ and $r_2$. 
We will apply the following lemma to \eqref{var}.
\begin{lemma}\lemmalab{szmolyan_canards_2001}
 (\cite[Proposition 4.2]{szmolyan_canards_2001}) $C_{a,2}$ and $C_{r,2}$ intersect transversally along $l_2$ if and only if there exists no non-zero solution of \eqref{var}$\vert_{r_2=0}$ which has algebraic growth for both $x_2\rightarrow \pm \infty$.
\end{lemma}
\begin{proof}
 Variations within the tangent spaces $T_0 C_{a,2}$ and $T_0 C_{r,2}$ are characterized by algebraic growth in the past ($x_2\rightarrow -\infty$) and in the future ($x_2\rightarrow \infty$), respectively. Since $C_{a,2}$ and $C_{r,2}$ are unique, variations normal to $T_0 C_{a,2}$ and $T_0 C_{r,2}$ will be characterized by exponential growth in the past and in the future, respectively. 
\end{proof}

The following lemma describes the properties of the solutions of \eqref{var}$_{r_2=0}$ necessary to invoke \lemmaref{szmolyan_canards_2001}.
\begin{lemma}\lemmalab{asymptotics}
 If $\Omega>0$ then \eqref{var}$_{r_2=0}$ has two linearly independent solutions 
 $$\gamma^-=(u^-,\mu_{2}^-)=\left(\sqrt{\frac{\pi \lambda}{2\beta}}(1-\textnormal{erf}\left(\sqrt{\frac{\lambda\beta }{2}} x_2\right))e^{\frac{ -\lambda \beta x_2^2}{2}},1\right),$$
 and 
 $$\gamma^+=(u^+,\mu_{2}^+)=\left(\sqrt{\frac{\pi \lambda}{\beta}}(1+ \textnormal{erf}\left(\sqrt{\frac{\lambda\beta }{2}} x_2\right))e^{\frac{ -\lambda \beta x_2^2}{2}},1\right).$$
 The solutions $\gamma^{\mp}$ grow exponentially as $x_2\rightarrow \pm \infty$ but the growth is only algebraic as $x_2\rightarrow \mp \infty$, respectively. 
 
If $\Omega<0$, then neither of the solutions of \eqref{var}$_{r_2=0}$ grows exponentially as $x_2\rightarrow \pm \infty$.
\end{lemma}
\begin{proof}
For $\Omega>0$ the result follows from the asymptotic behaviour of the error-function:
\begin{align*}
 \text{erf}(z) = \frac{2}{\sqrt{\pi}}\int_0^z e^{-t^2}dt.
\end{align*}
 For $\Omega<0$ we use that $\int_0^{z} e^{t^2-z^2}dt =\mathcal O(z^{-1})$ for $z\rightarrow \pm \infty$.
\end{proof}

Therefore by \lemmaref{szmolyan_canards_2001}, the manifolds $C_{a,2}$ and $C_{r,2}$ intersect transversally for $\Omega>0$ along $l_2$. Hence $M_{a,2}(\epsilon)$ and $M_{r,2}(\epsilon)$ intersect transversally $\mathcal O(\sqrt{\epsilon})$-close to $r_2$ for $\mu_2=\mu_{2,c}$ where
\begin{align}
 \mu_{2,c} = \mathcal O(r_2) \eqlab{mu2Canard}
\end{align}
and $r_2=\sqrt{\epsilon}$ sufficiently small. In fact, the following lemma allow us to calculate $\mu_{2,c}$ to lowest order:
\begin{lemma}\lemmalab{AlgebraicGrowthSolR2}
 Consider $\Omega>0$, then 
 the only solution of \eqref{var} which has algebraic growth for $x_2\rightarrow \pm \infty$ is
\begin{align*}
 u(x_2) &=-r_2\beta^{-2} (\eta^-+\beta \eta^+)x_2,
\end{align*}
for
\begin{align}
\mu_2 =  -\left( \frac{2\Omega }{(\beta+1)^2 \phi_{1,c}} (\eta^-+\beta \eta^+)+\hat y_{c} (\chi^-+\beta \chi^+)\right)\beta^{-1} r_2.\eqlab{mu2r2}
\end{align}
\end{lemma}
\begin{proof}
The complete solution of \eqref{var} is 
\begin{align}
 u(x_2) &= 
  \bigg(-\left(r_2\beta^{-2} \left(\eta^-+\beta \eta^+\right) +\lambda (\mu_2+r_2\beta^{-1} \hat y_{c} (\chi^-+\beta \chi^+)) \right)\sqrt{\frac{\pi}{2 \lambda\beta}} \text{erf}\left(\sqrt{\frac{\lambda\beta }{2}} x_2\right)+u(0)\bigg)e^{\frac{\lambda \beta x_2^2}{2}}\nonumber\\
  &-r_2\beta^{-2} (\eta^-+\beta \eta^+)x_2.  \eqlab{uSolution2nd}
\end{align}
Setting $u(0)=0$ and \eqref{mu2r2} the result follows from \lemmaref{asymptotics}.
\end{proof}

Then following \cite[Prop. 3.1]{krupa_extending2_2001} and using \eqref{mu2r2} we obtain
\begin{align}
  \mu_{2,c}&=-\left(\frac{2 \Omega }{(\beta+1)^2 \phi_{1,c}} (\eta^-+\beta \eta^+) +\hat y_{c}(\chi^-+\beta \chi^+) \right)\beta^{-1} r_2 +\mathcal O(r_2^2),\eqlab{muc}
%
%
%
\end{align}
and hence we have the following main theorem.
\begin{theorem}\label{canardTheorem}
Assuming \eqref{canardAssumption}, the regularized system \eqref{Xeps}  has a maximal canard at $$\mu =\sqrt{\epsilon} \mu_{2,c}(\sqrt{\epsilon})=\mathcal O(\epsilon),$$ where $\mu_{2,c}$ is given by \eqref{muc}, which is $\mathcal O(\sqrt{\epsilon})$-close to the singular canard.
\end{theorem}
\begin{proof}
 The statements follow from the analysis above. 
\end{proof}

 In the case where $\Omega<0$ then $C_{a,1}$ and $C_{r,3}$ are both non-unique and there exists a faux canard. The proof of this is very similar to the proof of faux canards in $\mathbb R^3$ \cite{KristiansenHogan2015,szmolyan_canards_2001}.

\section{Equilibria of the regularized system}\seclab{equilibriaregular}
We now revert to describing the dynamics in terms of the fast time $\tau$ in \eqref{tildeTau} for the remainder of the paper. In \secref{pseudo}, pseudo-equilibria of the PWS system were shown to exist for $\alpha\delta<0$ only. In \thmref{criticalManifold}, we showed that the equations of the reduced problem \eqref{ReducedProblemY2} agree with the sliding equations \eqsref{slidingEqns}{sigma}. Hence, as a consequence of the implicit function theorem, pseudo-equilibria of the sliding vector field $X_{sl}(\textbf x,\mu)$ for $\vert \mu\vert >\mu_0> 0$ perturb to locally unique equilibria on the slow manifold of the regularized problem for $0<\epsilon\le \epsilon_0(\mu_0)$ sufficiently small. Indeed, from our fast system \eqref{ExtendedProblem1W}, we find an equilibrium at
\begin{align}
 x^*&= \beta \delta\Omega^{-1} \mu+\mathcal O(\mu^2+\epsilon), \eqlab{EqSlowManifold}\\
  \hat y^* &= w^{-1}(-\alpha^{-1}\delta)+\mathcal O(\mu+\epsilon).\nonumber
\end{align}
Since we require $w(\hat y)>0$ by \eqref{weqn}, $(x^*,\hat y^*)$ is an equilibrium provided $\alpha \delta <0$. This is in accordance with equation \eqref{conditionPseudo} for the PWS system. 
The stability of this equilibrium is described by \propref{pwsProposition}. However, for $\mu =0$, $\epsilon=0$, the equilibrium lies on the non-hyperbolic line $\tilde q$, defined in \eqref{qtilde}, of the critical manifold. Hence Fenichel theory can not give a description of the equilibrium for $\epsilon\ne 0$. 
To accurately follow the equilibria near $\mu=0$ we need to consider the extended equations in chart $\kappa_2$, equations \eqref{eqnChartK2}. 
\subsection{Chart $\kappa_2$}
From \eqref{eqnChartK2}, we find the following equilibrium:
\begin{align}
 (x_2^*,\hat y^*) = \left(\beta \delta\Omega^{-1}\mu_2,\,w^{-1}(-\alpha^{-1} \delta)\right)+\mathcal O(r_2).\eqlab{eqx2w}
\end{align}
The equilibrium $(x_2^*,\hat y^*)$ intersects the section $\Lambda_2^-$, defined in \eqref{Lambda2pm}, when
\begin{align}
\hat y^* &= w^{-1}(-\alpha^{-1}\delta)+\mathcal O(r_2),\eqlab{eqx2wIntersection}\\
\mu_2 &= -\beta^{-1} \delta \Omega \nu^{-1}+\mathcal O(r_2).\nonumber
\end{align}

Consider $r_2=0$. The linearization of \eqref{eqnChartK2} about $(x_2^*,\hat y^*)$ is then given by
\begin{align}
 \dot z &= Az,\quad A=\begin{pmatrix}
           0 & \alpha w'(-\alpha^{-1} \delta)\\
           \alpha^{-1} \Omega &   \alpha \beta \Omega^{-1} w'(-\alpha^{-1} \delta )\mu_2
          \end{pmatrix},\eqlab{MatrixA}
\end{align}
where $z$ is the variation of $(x_2,\hat y)$ about $(x_2^*,\hat y^*)$ and 
\begin{align*}
 w'(-\alpha^{-1}\delta) = -\frac{2\phi'(\hat y^*)}{(1+\phi(\hat y^*))^2}=-\frac12 (1-\alpha^{-1}\delta)^2 \phi'(\hat y^*),
\end{align*}
using \eqref{wPrime} and the fact that $$\phi(\hat y^*) = \frac{1+\alpha^{-1}\delta}{1-\alpha^{-1}\delta}.$$ The determinant of the coefficient matrix $A$ is independent of $\mu_2$:
\begin{align}
 \text{det}\,A =
  -\Omega w'(-\alpha^{-1}\delta)
  \eqlab{detA}
\end{align}
and the trace of $A$ is given by
\begin{align}
\text{tr}\,A=\alpha \beta \Omega^{-1} w'(-\alpha^{-1} \delta )\mu_2
\eqlab{trA}
\end{align}
which vanishes for $\mu_2=0$. 
Since $w'<0$ from \eqref{wPrime}, the sign of $\text{det}\,A$ is determined by the sign of $\Omega$. Therefore we conclude the following: 
\begin{lemma}\lemmalab{eqStudyChartK2}
Consider $r_2=0$ and suppose $\alpha\delta<0$. 
\begin{itemize}
\item For $\Omega<0$ we have:
\begin{itemize}
\item The equilibrium $(x_2^*,\hat y^*)$, defined in \eqref{eqx2w}, is a neutral saddle for $\mu_2=\mu_{2,N}\equiv 0$ and a generic saddle for $\mu_2\ne \mu_{2,N}$. 
\item The stable (unstable) eigenspace associated with $(x_2^*,\hat y^*)$ and the linearization \eqref{MatrixA} is asymptotically vertical for $\beta \delta \Omega^{-1}\mu_2\rightarrow \mp \infty$. 
\item The unstable (stable) eigenspace associated with $(x_2^*,\hat y^*)$ and the linearization \eqref{MatrixA} is asymptotically horizontal for $ \beta \delta \Omega^{-1}\mu_2\rightarrow \mp \infty$.
\end{itemize}
\item For $\Omega>0$ we have:
\begin{itemize}
 \item A Hopf bifurcation at $\mu_2=\mu_{2,H}\equiv 0$. 
  \item The equilibrium $(x_2^*,\hat y^*)$ is attracting (repelling) for $\beta \delta \Omega^{-1}\mu_2 \lessgtr 0$.
  \item For $\vert \mu_2 \vert \in (\mu_{2,H},\mu_{2,F})$, where
  \begin{align}
 \mu_{2,F} \equiv \frac{2\Omega^{3/2}}{\sqrt{-w'(-\alpha^{-1}\delta)} \vert \alpha\beta\vert },\eqlab{mu2pm}
\end{align}
the equilibrium $(x_2^*,\hat y^*)$ is a focus.
  \item For $\vert \mu_2 \vert \ge \mu_{2,F}$, $(x_2^*,\hat y^*)$ is a node.
  \item The strong eigenspace is asymptotically vertical for $\beta \delta \Omega^{-1}\mu_2\rightarrow \pm \infty$. 
  \item The weak eigenspace is asymptotically horizontal for $\beta \delta \Omega^{-1}\mu_2\rightarrow \pm \infty$
\end{itemize}
\end{itemize}
       
The quantities $\mu_{2,N}$, $\mu_{2,H}$ and $\mu_{2,F}$ \textit{perturb} by an amount of $\mathcal O(r_2)$ for $r_2\ne 0$ by the implicit function theorem. 
\end{lemma}
\begin{proof}
 From \eqref{MatrixA} we find the following eigenvalues and eigenvectors:
 \begin{align}
  \lambda_{\pm} &= \frac{\text{tr}\,A}{2}\pm \frac12 \sqrt{\text{tr}^2\,A-4\text{det}\,A},\eqlab{lambdapm}\\
  v_{\pm}&=\begin{pmatrix}
            \frac{\alpha}{\lambda_{\pm}}\\
            1
           \end{pmatrix}.\eqlab{vpm}
 \end{align}
For $\mu_2=0$ we have a neutral saddle for $\Omega<0$ and a center for $\Omega>0$, since $\text{tr}\,A=0$ from \eqref{trA} and the sign of $\text{det}\,A$ is determined by the sign of $\Omega$. Also, since $\frac{d}{d\mu_2}\text{tr}\,A \ne 0 $, we have a Hopf bifurcation at $\mu_2=\mu_{2,H} \equiv 0$ when $\Omega>0$. 

For $\Omega>0$ we note that when 
\begin{align*}
 \text{tr}\, A=\pm 2\sqrt{\text{det}\,A},
\end{align*}
that is, when $\mu_2=\mp \mu_{2,F}$ from \eqref{mu2pm}, then the equilibrium $(x_2^*,\hat y^*)$ is a node. It is attracting (repelling) for those $\mu_2$ for which $\text{tr}\,A \lessgtr 0$. So from \eqref{trA}, we conclude that $(x_2^*,\hat y^*)$ is attracting (repelling) for $\beta \delta \Omega^{-1} \mu_2 \lessgtr 0$.

For either sign of $\Omega$ we have, for $\beta \delta \Omega^{-1} \mu_2\rightarrow \pm \infty$, that
\begin{align*}
 \lambda_{+}&\rightarrow \pm \infty,\\
 \lambda_- &\rightarrow 0.
\end{align*} 
%
%
So, for $\Omega<0$, using \eqref{vpm}, this means that the stable (unstable) eigenspace associated with $(x_2^*,\hat y^*)$ and the linearization \eqref{MatrixA} is asymptotically \textit{vertical} for $\beta \delta \Omega^{-1} \mu_2\rightarrow \mp \infty$. On the other hand the unstable (stable) eigenspace is asymptotically \textit{horizontal} for $\beta \delta \Omega^{-1} \mu_2\rightarrow \mp \infty$.

Similarly for $\Omega>0$ we find that the strong (weak) eigenspace\footnote{The strong (weak) eigenspace is the eigenspace associated with the eigenvalue representing the stronger (weaker) contraction or expansion.} associated with $(x_2^*,\hat y^*)$ and the linearization \eqref{MatrixA} is asymptotically \textit{vertical} (\textit{horizontal}) for $\beta \delta \Omega^{-1} \mu_2\rightarrow \pm \infty$. 
\end{proof}

{The results of \lemmaref{eqStudyChartK2} are shown in \figref{lemma8_1}, which can be compared with \figref{prop5_2}. }
\begin{figure}\begin{center}\includegraphics[width=.95\textwidth]{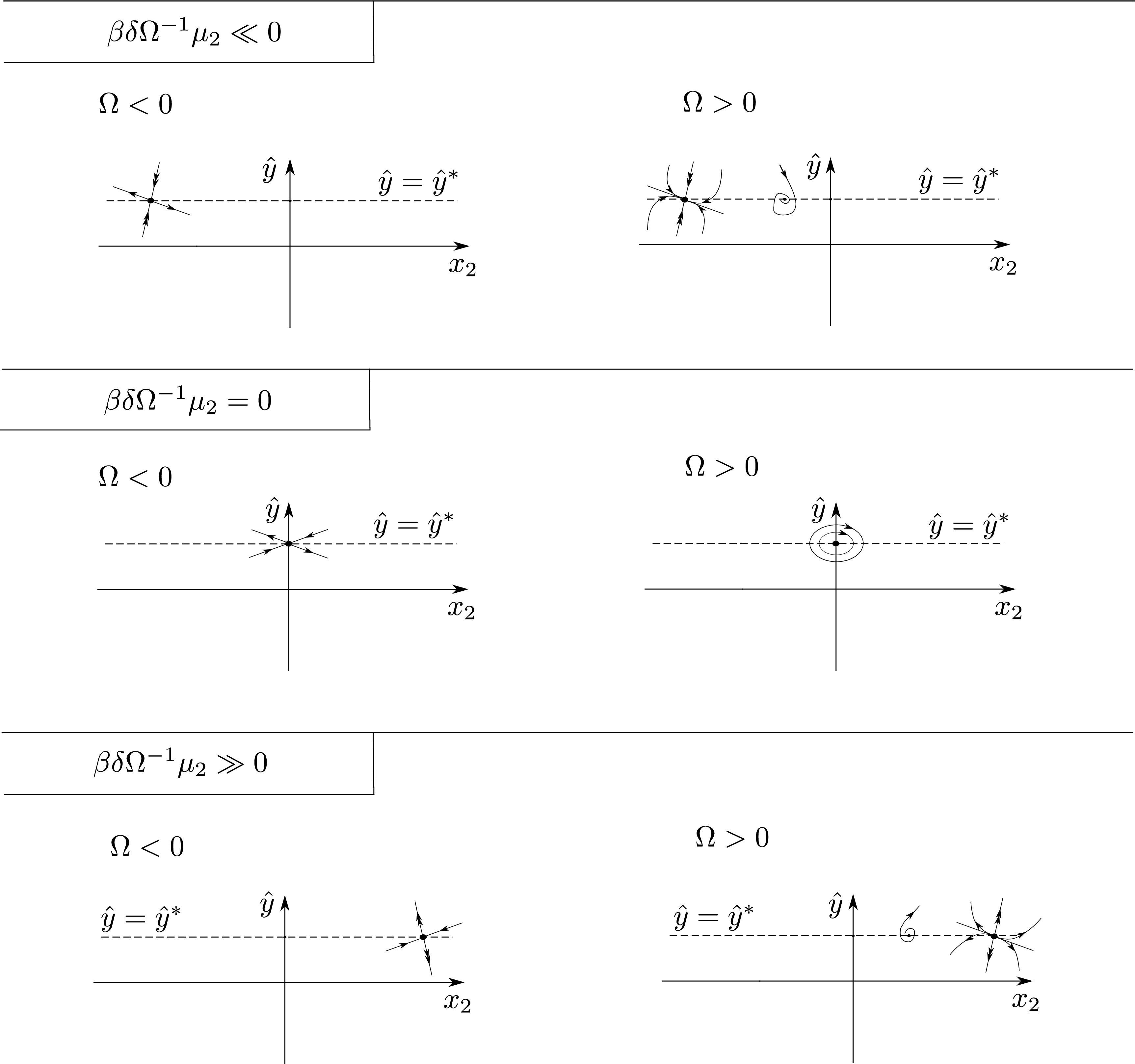}
\caption{Results of \lemmaref{eqStudyChartK2}. 
The dashed line corresponds to $\hat y^*=w^{-1}(-\alpha^{-1}\delta$), from \eqref{eqx2w}. This figure can be compared with \figref{prop5_2}.}\figlab{lemma8_1}
                		\end{center}
              \end{figure}

\subsection{Charts $\kappa_{1,3}$}
The results in \lemmaref{eqStudyChartK2} are in accordance with \propref{pwsProposition} in the limits $\mu_2\rightarrow \pm \infty$. But they occur within chart $\kappa_2$ and everything collapses to $\mu=0$ for $r_2=0$, see \eqref{eqChartKappa2}. To connect the results in chart $\kappa_2$ with the case $\mu=\mathcal O(1)$, we can consider charts $\kappa_{1,3}$. We obtain the following:
\begin{lemma}\lemmalab{equilibriaKappa13}
The equilibria described in \lemmaref{eqStudyChartK2} belong to the same smooth, locally unique family of equilibria as those in \eqref{EqSlowManifold}. The equilibria do not undergo any further bifurcations in passing from $\Lambda_2^\pm$ to $\Lambda^{\pm}$. 
\end{lemma}
\begin{proof}
In chart $\kappa_1$, we find the following family of equilibria:
 \begin{align}
  \hat y &= w^{-1}(-\alpha^{-1} \delta )+ \mathcal O(r_1+\epsilon_1),\nonumber\\
 \mu_1 &= -\beta \delta \Omega^{-1} +\mathcal O(r_1+\epsilon_1),\nonumber
 \end{align}
 within $M_{a,1}$. Using the conservation of $\epsilon$ and $\mu$ it is straightforward to trace this family of equilibria from $\Lambda_{1}^{-}$ to $\Lambda^{-}$ and connect them to the equilibria described in chart $\kappa_2$ with the $\mathcal O(1)$ equilibria described in \eqref{EqSlowManifold}. The analysis in chart $\kappa_3$ is identical. 
\end{proof}

 \subsection{The Hopf bifurcation for $\Omega>0$}\seclab{Hopf}
We now describe in further detail the Hopf bifurcation for $\Omega>0$ in \lemmaref{eqStudyChartK2} and the resulting birth of limit cycles. As mentioned at the start of this section, we return to the time $\tau$, since time $\tilde \tau$ defined in \eqref{tildeTau} and used in \secref{canards} leads to difficulties when the periodic orbits leave the region of regularization $y\in (-\epsilon,\epsilon)$, or $\hat y\in (-1,1)$ (whereas canards do not suffer this fate). To proceed, we write the extended problem \eqref{ExtendedProblem1} in chart $\kappa_2$ to obtain
\begin{align}
 \dot x_2 &=(\delta+r_2\zeta^+ x_2+\mathcal O(r_2(\mu_2+r_2)))(1+\phi(\hat y))+(\alpha+r_2\zeta^-x_2 +\mathcal O(r_2(\mu_2+r_2)))(1-\phi(\hat y)),\eqlab{eqnChartK2y2}\\
\dot{\hat y} &=(x_2+r_2\eta^+ x_2^2+\mathcal O(r_2(\mu_2+r_2))))(1+\phi(\hat y)) \nonumber\\
 &+(-\beta (x_2-\mu_2) +r_2 \eta^-(x_2-\mu_2)^2+\mathcal O(r_2(\mu_2+r_2)))(1-\phi(\hat y).\nonumber
\end{align}
where we have multiplied time $\tau$ by $r_2$. The equilibrium in \eqref{eqx2w} for $\alpha \delta<0$ then becomes 
\begin{align}
(x_2^*,\hat y^*)\equiv \left(\beta \delta \Omega^{-1}\mu_2,\,\phi^{-1}\left(\frac{1+\alpha^{-1} \delta}{1-\alpha^{-1}\delta}\right)\right)+\mathcal O(r_2).\eqlab{eqx2y2}
\end{align}
This equilibrium undergoes a Hopf bifurcation at $\mu_2=\mu_{2,H}=\mathcal O(r_2)$ when $\Omega>0$ (compare with \lemmaref{eqStudyChartK2}).  Let
\begin{align}
 \phi_{i,H} \equiv \phi^{(i)}(\hat y_{0}^*),\quad i=1,2,3,\eqlab{phi_i}
\end{align}
where 
\begin{align}
\hat y_{0}^*=\phi^{-1}\left(\frac{1+\alpha^{-1}\delta}{1-\alpha^{-1}\delta}\right),\eqlab{eqy20Star}
\end{align}
and
\begin{align*}
 \phi^{(i)} \equiv \frac{d^i\phi}{d\hat y^i}.
\end{align*}
The subscript $0$ in \eqref{eqy20Star} is used to emphasize that $\hat y^*$ has been obtained from  \eqref{eqx2y2} with $\mu_2=r_2=0$. By assumption $\phi_{1,H}>0$
 and so we obtain
\begin{proposition}\proplab{HopfBifurcationChartKappa2}
 System \eqref{eqnChartK2y2} has an equilibrium at $(x_2^*,\hat y^*)$, as defined in \eqref{eqx2y2}, which undergoes a Hopf bifurcation at $$\mu_2 = \mu_{2,H}(r_2)$$ where
 \begin{align}
 \mu_{2,H}(r_2) \equiv \left(\frac{2(\alpha (\zeta^++\chi^+)-\delta (\zeta^-+\chi^-)) \Omega}{(\alpha-\delta)^2 \phi_{1,H}}-{(\chi^-+\beta\chi^+)\hat y_{0}^*}\right)\beta^{-1}r_2+\mathcal O(r_2^2).\eqlab{mu2H}
 \end{align}
 The first Lyapunov coefficient is given by
\begin{align*}
 a &= a_2 r_2+\mathcal O(r_2^2),
 \end{align*}
 where
\begin{align}
a_2 &=\frac{(\alpha-\delta)\phi_{1,H}^2}{16\Omega^2}\bigg({(\beta +1)^2}\left(\delta \zeta^--\alpha \zeta^+\right)-{(\alpha-\delta)^2}\left(\eta^-+\beta \eta^+\right)\bigg)+\frac{(\beta+1)\phi_{2,H}}{16\Omega}(\delta\chi^--\alpha\chi^+)\nonumber\\
&+\frac18\left(\frac{\phi_{3,H}}{\phi_{1,H}(\alpha-\delta)}-\frac{\phi_{2,H}^2}{\phi_{1,H}^2(\alpha-\delta)}-\frac{(\beta+1)\phi_{2,H}}{\Omega}\right)\left(\delta (\zeta^-+\chi^-)-\alpha (\zeta^++\chi^+)\right).
%
%
\eqlab{K2}
 \end{align}
If $a_2\ne 0$ then for $\epsilon$ sufficiently small there exists a family of unique periodic solutions bifurcating from $(x_2^*,\hat y^*)$ for 
\begin{align*}
& \beta \delta \Omega^{-1} \mu_2 > \beta \delta \Omega^{-1}\mu_{2,H} \quad \text{when}\quad a_2<0,\\
& \beta \delta \Omega^{-1} \mu_2 < \beta \delta \Omega^{-1}\mu_{2,H} \quad \text{when}\quad a_2>0,
\end{align*}
with amplitude $\mathcal O\left(\sqrt{\vert \mu_2-\mu_{2,H}\vert r_2^{-1}}\right)$. The periodic orbits are attracting for $a_2<0$ and repelling for $a_2>0$, for $r_2$ sufficiently small. 
\end{proposition}
		\begin{proof}
The calculation of $a_2$ is based on classical Hopf bifurcation theory \cite{car1}. 
\end{proof}  

 Note that the Hopf bifurcation is degenerate within $r_2=0$ since $a\equiv 0$ there. The reason for this is that the system \eqref{eqnChartK2y2} with $\mu_2=0$ and $r_2=0$ is Hamiltonian, as we shall now demonstrate. In this case, \eqref{eqnChartK2y2} becomes: 
\begin{align}
 \dot x_2 &=\delta(1+\phi(\hat y)) +\alpha (1-\phi(\hat y)),\eqlab{mu2r2eqn}\\
\dot{\hat y} &=\left( 1+\phi(\hat y)-\beta (1-\phi(\hat y))\right)x_2,\nonumber
 \end{align}
The Hamiltonian function $H=H(x_2, \hat y)$ is given by
 \begin{align}
 H(x_2, \hat y) = \frac12 x_2^2+\int_{\hat y_{0}^*}^{\hat y} \frac{\delta (1+\phi(s))+\alpha (1-\phi(s))}{\beta (1-\phi(s))-(1+\phi(s))}ds.\eqlab{HchartK2}
 \end{align}
The symplectic structure matrix $J(x_2, \hat y)$ is non-canonical: 
\begin{align}
J(x_2, \hat y)=(\beta (1-\phi(\hat y))-(1+\phi(\hat y)))\begin{pmatrix}
                                        0 & 1\\
                                        -1 & 0
                                       \end{pmatrix},\eqlab{J}
                                       \end{align}
which is regular and non-zero near $\hat y=\hat y_{0}^*$ since $\phi(\hat y_{0}^*) = \frac{1+\alpha^{-1}\delta}{1-\alpha^{-1}\delta}$, $\alpha\delta<0$ and \asuref{asu3}. Hence the system with $\mu_2=0$ and $r_2=0$ has a whole family of periodic orbits in the vicinity of \eqref{eqx2y2}$_{\mu_2=r_2=0}$. The Hamiltonian system is not well-defined for $\hat y=\hat y_{c}$ \eqref{phiy2c}, $\beta>0$,  since $J(x_2,\hat y_{c})=0$.  
\begin{remark}
The periodic orbits within the $(x_2,\hat y)$-plane rotate about \eqref{eqx2y2} in the counter clockwise (clockwise) direction if $\alpha>0$ ($\alpha<0$).
\end{remark}

Combining the results in \lemmaref{eqStudyChartK2}, \lemmaref{equilibriaKappa13}, and \propref{HopfBifurcationChartKappa2} we obtain one of our main results, Theorem~\ref{mainHopf}, as follows:

\begin{theorem}\label{mainHopf}
Assuming \eqref{eqAssymption}, the regularized system  \eqref{Xeps} has a smooth and locally unique family of equilibria
\begin{align*}
(x,y)=(x^*,y^*)(\mu,\sqrt{\epsilon}) \equiv \left(\beta \delta\Omega^{-1}\mu+\mathcal O(\mu^2+\epsilon),\mathcal O(\epsilon)\right)
\end{align*}
where $(x^*(\mu,0),y^*(\mu,0))=(x^*(\mu,0),0)$ agrees with the family of pseudo-equilibria for the PWS system (see \propref{pwsProposition}) and, in particular, $\partial_\mu x^*(\mu,0)\ne 0$. For
\begin{itemize}
 \item[$\Omega<0$:] The family of equilibria consists of saddles and does not undergo any bifurcation. 
 \item[$\Omega>0$:] The family of equilibria undergoes a Hopf bifurcation at $\mu=\sqrt{\epsilon}\mu_{2,H}(\sqrt{\epsilon})=\mathcal O(\epsilon)$, where $\mu_{2,H}$ is given by \eqref{mu2H}. The first Lyapunov coefficient is given by $a = a_2 \sqrt{\epsilon}+\mathcal O(\epsilon)$
where $a_2$, given by \eqref{K2}, depends upon the regularization function $\phi$.
 \end{itemize}
\end{theorem}

Note that $a_2$, as given in \eqref{K2}, depends upon the regularization function $\phi$, through $\phi_{1,H}$, $\phi_{2,H}$ and $\phi_{3,H}$ as defined in \eqref{phi_i}, and hence that the criticality of the Hopf bifurcation depends on $\phi$. This observation leads to another one of our main results:
 \begin{theorem}\label{K2sign}
  Suppose that $\Omega>0$ and $\alpha\delta<0$ so that there exists a Hopf bifurcation. Then, provided
  \begin{align}
   \delta (\zeta^-+\chi^-)-\alpha (\zeta^++\chi^+)\ne 0, \eqlab{assumption1}
  \end{align}
 all cases $a_2<0$, $a_2=0$ and $a_2>0$ can be attained by varying $\phi$. 
 \end{theorem}
 
\begin{proof}
We are free to choose $\phi_{1,H}>0$, $\phi_{2,H}$ and $\phi_{3,H}$ in \eqref{K2}. Note, from \eqref{K2}, that the equation $a_2=0$ is linear in $\phi_{3,H}$ and that the coefficient of $\phi_{3,H}$:
\begin{align*}
  \frac{\delta (\zeta^-+\chi^-)-\alpha (\zeta^++\chi^+)}{8\phi_{1,H}(\alpha-\delta)}
\end{align*}
is non-zero by assumption \eqref{assumption1}. Hence $a_2=0$ can be solved for $\phi_{3,H}$. The statement of the proposition therefore follows.
\end{proof}

\begin{remark}
It seems natural to insist that $\phi$ should be an odd function. If $\phi$ were not odd, then one of the vector-fields $X^{\pm}$ would be favoured over the other by the regularization. The functions $\phi$ that we used to prove Theorem~\ref{K2sign} can be odd, at least if $\hat y_{0}^*\ne 0$. If $\hat y_{0}^*=0$ and $\phi$ is odd, then $\phi_{2,H}=0$. Hence, the equation $a_2=0$ should have a solution with $\phi_{3,H}<0$ for $\phi$ to be odd and for Theorem~\ref{K2sign} to apply.
\end{remark}

\begin{remark}
Another natural condition appears to be that $\phi'$ should be strictly increasing within $(-1,0)$ and strictly decreasing within $(0,1)$. The functions used in the proof of Theorem~\ref{K2sign} may also be chosen to satisfy these conditions, at least when $\hat y_{0}^*\ne 0$. The functions then just have $\phi_{2,H}\gtrless 0$ for $\hat y_{0}^*\lessgtr 0$.  For $\hat y_{0}^*=0$ we have $\phi_{2,H}=0$ and the equation $a_2=0$ should have a solution with $\phi_{3,H}<0$ to ensure $\phi''(\hat y)\gtrless 0$ for $\hat y\lessgtr 0$ and that Theorem~\ref{K2sign} apply. 
\end{remark}
\section{Limit cycles of the regularized system}\seclab{limitCycles}
From \secref{limitCyclesPWS}, limit cycles of the PWS system can exist for the case $II_2$, which occurs when
\begin{align}
 \delta = -1,\quad \alpha>0,\quad \beta<0.\eqlab{II2Assumption}
 \end{align}
 See also \tabref{tblPWS}. Limit cycles can also occur in the regularized version of case $VI_3$. This case occurs for 
\begin{align}
\delta=1, \quad \beta>-\alpha>0.\eqlab{VI3Assumption} 
\end{align}
Note that $\beta>0$ in \eqref{VI3Assumption} and hence from Theorem~\ref{canardTheorem} the regularization of $VI_3$ also possesses a canard which is $\mathcal O(\sqrt{\epsilon})$-close to the singular canard of the PWS system. 

The regularization of $II_2$ and $VI_3$, denoted by $II_2^\epsilon$ and $VI_3^\epsilon$ respectively, exhibit significant differences. In case $II_2^\epsilon$, the limit cycles eventually (for $\mu_2$ large enough) cross the region of regularization $\hat y \in (-1,1)$ from $\hat y>1$ to $\hat y<-1$ and back again. There is no sliding and hence no singular canards in the corresponding PWS case $II_2$. Thus there are no canards in the regularized case $II_2^\epsilon$. However, for $VI_3^\epsilon$, the resulting limit cycles interact with the slow manifolds and the maximal canard to produce a scenario almost identical to the canard explosion phenomenon in classical slow-fast theory \cite{krupa_relaxation_2001}.


In this section we present a comprehensive study of the regularized limit cycles that are due to the Hopf bifurcation in chart $\kappa_2$ (see \propref{HopfBifurcationChartKappa2}). In \secref{LCKappa2}, these limit cycles are followed, beyond the validity of the classical Hopf bifurcation theory, into large limit cycles in chart $\kappa_2$. In terms of the original $(x,y)$-variables, from \eqref{eqChartKappa2}, these periodic orbits are, however, still small, only extending $\mathcal O(\sqrt{\epsilon})$ in the $x$-direction and $\mathcal O(\epsilon)$ in the $y$-direction. To follow these orbits to $\mathcal O(1)$-size, and obtain a connection to the PWS system, we must use charts $\kappa_{1,3}$. 

In doing so, we use different techniques for cases $II_2^\epsilon$ and $VI_3^\epsilon$. We split the analysis into separate parts. In \secref{LCO1II2} we study limit cycles of $\mathcal O(1)$-size for case $II_2^\epsilon$ while \secref{LCO1VI3} contains the corresponding analysis for case $VI_3^\epsilon$. The connection of these $\mathcal O(1)$-limit cycles with the limit cycles in chart $\kappa_2$ for cases $II_2^\epsilon$ and $VI_3^\epsilon$ is shown in \secsref{LCCII2}{LCCVI3}, respectively.   



\subsection{Chart $\kappa_2$}\seclab{LCKappa2}
\propref{HopfBifurcationChartKappa2} only guarantees the existence of small periodic orbits for $\mu_2/r_2$ small within chart $\kappa_2$. To follow these periodic orbits within chart $\kappa_2$ for larger values of $\mu_2/r_2$ we follow the Melnikov-based approach of Krupa and Szmolyan \cite{krupa_relaxation_2001}. We will consider both $II_2^\epsilon$ and $VI_3^\epsilon$ in this section.

%

First we define $\hat y_0^h$, $\hat y_1^h$ and $\hat y_2^h$ as follows. Consider the forward solution $\sigma = (x_2,\hat y)(t)$ with initial condition $(0,\hat y_0^h)$ where the Hamiltonian $H$ in \eqref{HchartK2} takes the value $H(0,\hat y_0^h)=h>0$ and
\begin{align}
\hat y_0^h<\hat y^*_0.\eqlab{condy20}
\end{align}
The point $(0,\hat y_2^h)$ is then the second return of $\sigma$ to $x_2=0$ where
\begin{align*}
\hat y_2^h<\hat y^*_0.
\end{align*}
Similarly, we let $(0,\hat y_1^h)$ denote the first return of $\sigma$ to $x_2=0$ where 
\begin{align*}
 \hat y_1^h>\hat y_0^*.
\end{align*}
 The relevant quantities are shown in \figref{HopfCanardR2Mu2} for the case $VI_3^\epsilon$. Notice that by \eqref{VI3Assumption} and $\Omega>0$ it follows from \eqsref{l2Solution}{eqx2w} that the singular canard for $\mu_2=r_2=0$ is above the bifurcating equilibrium. This is illustrated in \figref{HopfCanardR2Mu2} by letting the continuation of $S_{a,\epsilon}$ and $S_{r,\epsilon}$ lie above $\hat y^*_0$. The case $II_2^\epsilon$ is similar but there are no slow manifolds in this case. 

Following \propref{HopfBifurcationChartKappa2} there exists an $h_0>0$ independent of $r_2$ so that for $0\le h\le h_0$ there exists a locally unique family of limit cycles parametrized by $\mu_2$ whose stability is determined by the sign of $a_2$ given in \eqref{K2}. We therefore take $h\ge h_0>0$ 
\begin{figure}\begin{center}\includegraphics[width=0.65\textwidth]{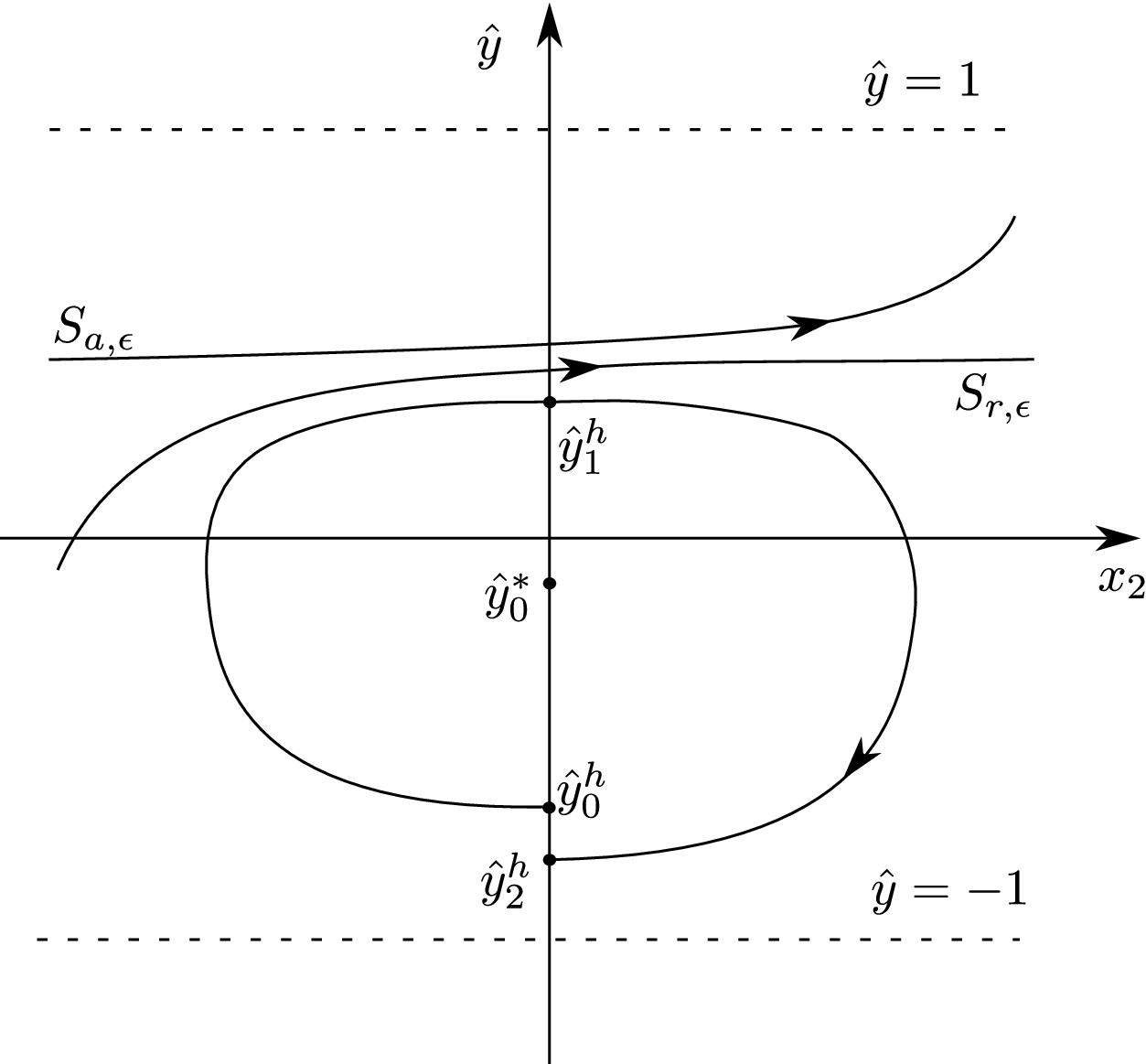}
\caption{Illustration of $\hat y_0^h$, $\hat y_1^h$ and $\hat y_2^h$ for the case $VI_3^\epsilon$. Here $S_{a,\epsilon}$ and $S_{r,\epsilon}$ are in fact the continuation $M_{a,2}(\epsilon)$ and $M_{r,2}(\epsilon)$ of the Fenichel slow manifolds into chart $\kappa_2$. The case $II_2^\epsilon$ is similar but there are no slow manifolds in this case. 
}\figlab{HopfCanardR2Mu2}
                		\end{center}
              \end{figure}
and consider the following distance function
    \begin{align}
     D(r_2,\mu_2,h)  = H(0,\hat y_2^h)-H(0,\hat y_0^h).\eqlab{DistanceFunction}
    \end{align}
From the analysis proceeding \propref{HopfBifurcationChartKappa2}, the system with $\mu_2=r_2=0$ is Hamiltonian and hence $D(0,0,h)=0$ for all $h\ge 0$. Also since $$\partial_{\hat y} H = \frac{\delta (1+\phi(\hat y))+\alpha (1-\phi(\hat y))}{\beta (1-\phi(\hat y))-(1+\phi(\hat y))}\ne 0,$$ for $\hat y< \hat y_{0}^*$ (in accordance with \eqref{condy20}), roots of the equation $D=0$ correspond to periodic orbits.  
    
Let $T^h$ denote the period of the orbit $\sigma$ of the Hamiltonian system with $\mu_2=r_2=0$ satisfying $H(\sigma)=h$.     
We have the following lemma, similar to \cite[Proposition 4.1]{krupa_relaxation_2001}:
\begin{lemma}
 For $h\ge h_0$ we have
 \begin{align}
  D(r_2,\mu_2,h) = D_{r_2}(h)r_2+D_{\mu_2}(h)\mu_2+\mathcal O((r_2+\mu_2)^2),\eqlab{Deqn}
 \end{align}
where
 \begin{align}
D_{r_2}(h) &=2\int_{0}^{T^h/2} {(\zeta^+(1+\phi(\hat y))+\zeta^- (1-\phi(\hat y)))x_2^2} \nonumber\\
&+ \frac{\delta(1+\phi(\hat y))+\alpha(1-\phi(\hat y))}{\beta(1-\phi(\hat y))-(1+\phi(\hat y))} ((\eta^+x_2^2+\chi^+\hat y)(1+\phi(\hat y))+(\eta^-x_2^2+\chi^-\hat y) (1-\phi(\hat y)))dt,\eqlab{Dr2}\\
D_{\mu_2}(h) &=-4\int_{\hat y_0^h}^{\hat y_1^h}   \frac{\beta \phi'(y)}{(\beta (1-\phi(\hat y))-(1+\phi(\hat y)))^2} x_2 d\hat y,\eqlab{Dmu2}
 \end{align}
where $(x_2,\hat y)(t)$ satisfies \eqref{mu2r2eqn} with $(x_2,\hat y)(0)=(0,\hat y_0^h)$, $H(x_2,\hat y)=h$ and $(x_2,\hat y)(T^h/2)=(0,\hat y_1^h)$. 
\end{lemma}

\begin{proof}
Similar calculations to those in \cite[Proposition 4.1]{krupa_relaxation_2001} lead to
\begin{align*}
 D_{r_2}(h) &= \int_0^{T^h} \nabla H(x_2(t),\hat y(t))\cdot G_{r_2}(x_2(t),\hat y(t)) dt,\\
 D_{\mu_2}(h)&=\int_0^{T^h} \nabla H(x_2(t),\hat y(t))\cdot G_{\mu_2}(x_2(t),\hat y(t)) dt,
\end{align*}
where
\begin{align*}
 G_{r_2}(x_2,\hat y) &= \begin{pmatrix}
             (\zeta^+(1+\phi(\hat y))+\zeta^- (1-\phi(\hat y)))x_2\\
(\eta^+x_2^2+\chi^+\hat y)(1+\phi(\hat y))+(\eta^-x_2^2+\chi^-\hat y) (1-\phi(\hat y))
            \end{pmatrix},
\\
 G_{\mu_2}(x_2,\hat y) &=\begin{pmatrix}
              0\\
              (1-\phi(\hat y))\beta 
             \end{pmatrix}.
\end{align*}
The Hamiltonian system possesses a time-reversible symmetry $(x_2,\hat y,t)\mapsto (-x_2,\hat y,-t)$ so:
\begin{align*}
 D_{r_2}(h) &= 2\int_0^{T^h/2} \nabla H(x_2(t),\hat y(t))\cdot G_{r_2}(x_2(t),\hat y(t)) dt,\\
 D_{\mu_2}(h)&=2\int_0^{T^h/2} \nabla H(x_2(t),\hat y(t))\cdot G_{\mu_2}(x_2(t),\hat y(t)) dt.
\end{align*}
For $D_{\mu_2}(h)$ we then use integration by parts.
\end{proof}
\begin{remark}
 If $\vert \hat y_0^h\vert,\vert \hat y_1^h\vert >1$ then $D_{\mu_2}$ can be simplified further:
 \begin{align}
  D_{\mu_2}(h) = -4\int_{-1}^1 \frac{\beta x_2}{(\beta (1-\phi)-(1+\phi))^2}d\phi.\eqlab{Dmu2Large}
 \end{align}
 This is only relevant for case $II_2^\epsilon$. In the case $VI_3^\epsilon$ the maximal canard prevents the local limit cycles from entering $\hat y\ge 1$ (see \figref{HopfCanardR2Mu2}). 
\end{remark}

Since $D_{\mu_2}(h)\ne 0$ we can apply the implicit function theorem to conclude the following:
\begin{proposition}\proplab{largeCyclesK2}
Fix $\nu>0$ small. The family of limit cycles from the Hopf bifurcation, described in \propref{HopfBifurcationChartKappa2}, can be continued into periodic orbits corresponding to roots of $D(r_2,\mu_2,h)$ for $h\le 2 \nu^{-2}$ and $r_2\le r_{20}(\nu)$ sufficiently small. The orbits are $\mathcal O(r_2)$-close to the periodic orbits of the Hamiltonian system $H=H(x_2,\hat y)$ defined in \eqref{HchartK2}.  
\end{proposition}
\begin{proof}
 From \eqref{Deqn} and the implicit function theorem we obtain 
 \begin{align}
  \mu_2 = -\frac{D_{r_2}(h)}{D_{\mu_2}(h)}r_2+\mathcal O(r_2^2).\eqlab{mu2largeCyclesK2}
 \end{align}
\end{proof}

\begin{remark}\remlab{POBeyondLambda2}
 The periodic orbit with $h=2\nu^{-2}$ intersects $\hat y=\hat y_{0}^*$ in $x_2= \mp 2\nu^{-1}$ for $\mu_2=r_2=0$. By \propref{largeCyclesK2} we are therefore able to continue periodic orbits beyond the sections $\Lambda_2^{\mp}$ in chart $\kappa_2$. These orbits belong to a $C^k$-smooth and locally unique family because they are obtained using an argument based on the implicit function theorem for $\mu_2=r_2=0$. 
\end{remark}

\subsection{$\mathcal O(1)$ limit cycles}
We now wish to consider limit cycles of the regularized system with amplitudes that are $\mathcal O(1)$ with respect to $\epsilon$. 
As mentioned above, the analysis is divided into two cases: $II_2^\epsilon$ and $VI_3^\epsilon$.

\subsubsection{$\mathcal O(1)$-limit cycles for case $II_2^\epsilon$}\seclab{LCO1II2}
We start by obtaining $\mathcal O(1)$-periodic orbits in the original $(x,y)$-variables by following fixed points of a Poincar\'e map:
\begin{align}
P_\epsilon:\, \{y=\epsilon,\,x>0\}\rightarrow \{y=\epsilon,\,x>0\},\eqlab{Pepsilon}
\end{align}
where defined under the flow of the regularized system. Since these orbits are $\mathcal O(1)$ with respect to $\epsilon$ and only involve crossing (see \figref{crossingPOS} case $II_2$), the mapping $P_\epsilon$ is smoothly $\mathcal O(\epsilon)$-close to $P_0$, as defined in \eqref{P0Map} for the PWS system. We therefore obtain the following proposition.
\begin{proposition}\proplab{propLimitCyclesReg}
Fix $\mu_0$ small and consider \eqref{II2Assumption}. Suppose that $\vert \mu\vert \ge \mu_0>0$ and $\Delta_{II_2} \ne 0$, where $\Delta_{II_2}$ is defined in \eqref{Ceqn}. Then for $\epsilon\le \epsilon_0(\mu_0)$ sufficiently small, the regularized system has a family of periodic orbits corresponding to fixed points of $P_\epsilon$ of the following form
\begin{align*}
 x_\epsilon(\mu) = x_0(\mu)+\mathcal O(\epsilon),
\end{align*}
with $x_0(\mu)$ given by \eqref{x0mu}, and $\mu \Delta_{II_2}^{-1}<-\mu_0 \vert \Delta_{II_2}\vert^{-1}<0$. The periodic orbits are attracting for $\Delta_{II_2}<0$ and repelling for $\Delta_{II_2}>0$. Moreover, they are continuously $\mathcal O(\epsilon)$-close to periodic orbits of the PWS system.
\end{proposition}
\begin{proof}
 \propref{propLimitCyclesReg} gives non-degenerate fixed points $x_0(\mu)$ of $P_0$ for $\mu \Delta_{II_2}^{-1}<-\mu_0 \vert \Delta_{II_2}\vert^{-1}<0$. Since $P_\epsilon=P_0+\mathcal O(\epsilon)$ we can apply the implicit function theorem to obtain fixed points $x_\epsilon(\mu)=x_0(\mu)+\mathcal O(\epsilon)$ of $P_\epsilon$. The stability of the fixed points $x_\epsilon(\mu)$ is also determined by the stability of $x_0(\mu)$ as a fixed point of $P_0$.  
\end{proof}

\subsubsection{$\mathcal O(1)$-limit cycles for case $VI_3^\epsilon$}\seclab{LCO1VI3}
This case has a canard at $\mu=r_2\mu_{2,c}$ and a Hopf bifurcation at $\mu=r_2\mu_{2,H}$. 
The co-existence of a Hopf bifurcation and a canard leads to the canard explosion phenomenon in which the amplitude of limit cycles undergo $\mathcal O(1)$ variations within an exponentially small parameter regime. In order to prove this statement, we follow the proof of a related assertion in \cite[Proposition 5.1]{krupa_relaxation_2001} for classical planar slow-fast systems. 

Let $\mu=r_2\mu_2=\sqrt{\epsilon}\mu_2$ and consider the original $(x,y)$-variables, in which the region of regularization is $y\in (-\epsilon,\epsilon)$, and let $\gamma$ be the forward orbit with initial condition $(x,y)(0)=(0,y_0)$ where $y_0<0$ small but independent of $\epsilon$. The situation is illustrated in \figref{canardExplosion}. Since the fold of the associated PWS system is invisible from below (see \eqref{II2Assumption} and \propref{visibility}) we know that the first return of $\gamma$ with $x=0$ is a point $(0,y_1)$ with $y_1=\mathcal O(\epsilon)$. Let $\underline \gamma$ be the backward orbit with initial condition $(x,y)(0)=(0,y_0)$. Denote by $(0,\underline y_1$) the first return of $\underline \gamma$ with $x=0$. Here $\underline y_1=\mathcal O(\epsilon)$. Let $(x_1(y_0,\epsilon),-\epsilon)$ and $(\underline x_1(y_0,\epsilon),-\epsilon)$ denote the first intersections of $\gamma$ and $\underline \gamma$, respectively, with $y=-\epsilon$. The functions $x_1(y_0,\epsilon)$ and $\underline x_1(y_0,\epsilon)$ are smooth in $y_0$ and $\epsilon$. In particular, $x_1(y_0,0)$ and $\underline x_1(y_0,0)$ can be obtained from the associated PWS system. 

\begin{figure}\begin{center}\includegraphics[width=0.6\textwidth]{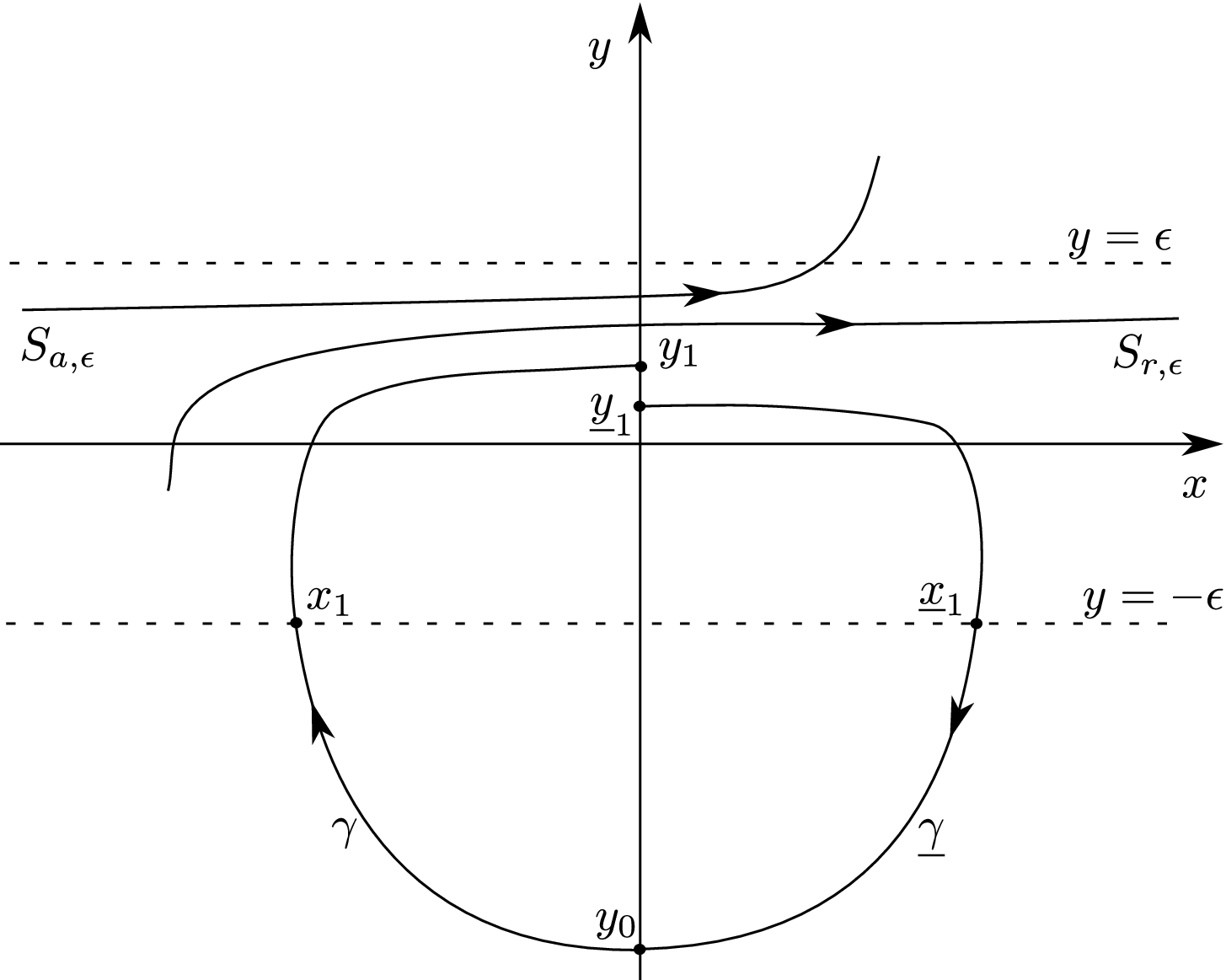}
\caption{Important quantities for case $VI_3^\epsilon$, relevant to \propref{canardExplosion}. The region of regularization $y\in (-\epsilon,\epsilon)$ is exaggerated for illustrative purposes. }\figlab{canardExplosion}
                		\end{center}
              \end{figure}

We consider the distance function:
\begin{align*}
 \mathcal D(y_0,\mu_2) = y_1-\underline y_1.
 \end{align*}
 Roots of $\mathcal D$ correspond to periodic orbits. 
As in \cite{krupa_relaxation_2001} we solve $\mathcal D(y_0,\mu_2)=0$ by noting, from Fenichel theory, that 
\begin{align*}
 \mathcal D(y_0,\mu_{2,c}) = \mathcal O(e^{-c/\epsilon}),\quad c=c(y_0)>0.
\end{align*} 
Since $S_{r,\epsilon}$ and $S_{a,\epsilon}$ are transverse for $\mu_2=\mu_{2,c}$ this then effectively implies the existence of a $\mu_2=\mu_{2,p}(y_0)$ solving $\mathcal D(y_0,\mu_2)=0$ and satisfying $\mu_{2,p}=\mu_{2,c}+\mathcal O(e^{-c/\epsilon})$. Since $x_1$ and $\underline x_1$ are increasing functions of $y_0$ for $y_0$ small, it also follows that $\mu_{2,p}(y_0)$ approaches $\mu_{2,c}$ monotonically as $y_0$ increases, at least for $y_0$ sufficiently small. 
The stability of the periodic orbits is determined by the sign of a \textit{way-in/way-out} function $R=R(y_0)$, see
\cite[Proposition 5.4]{krupa_relaxation_2001}, that measures the contraction and expansion along $S_{a,r}$. In our case the contraction and expansion is determined by the following function 
 \begin{align*}
  (X_2^+(x,0,0)-X_2^-(x,0,0)) \phi' (\hat y).
 \end{align*}
Inserting $\hat y=\hat y_{c}=\phi^{-1}\left(\frac{\beta-1}{\beta+1}\right)$ from \eqref{phiy2c}, which corresponds to $S_{a,r}$ for $\mu_2=r_2=0$, we obtain the function $R=R(y_0)$:
\begin{align*}
R(y_0) = \phi_{1,c} \int_{x_1(y_0,0)}^{\underline x_1(y_0,0)} (X_2^+(x,0,0)-X_2^-(x,0,0)) dx.
\end{align*}
Since $\phi_{1,c}=\phi'(\hat y_{c})>0$, the sign of $R$ coincides with the sign of
\begin{align*}
 \tilde R(y_0)  = \int_{x_1(y_0,0)}^{\underline x_1(y_0,0)} (X_2^+(x,0,0)-X_2^-(x,0,0)) dx.
\end{align*}
We obtain the following proposition:
\begin{proposition}\proplab{canardExplosion}
Consider a point $p=(0,y_0)$ with $y_0\in [-c_1^{-1},-c_2^{-1}]$ with $c_2>c_1$ sufficiently large but fixed. Then for $\epsilon\le \epsilon_0(c_1,c_2)$ sufficiently small there exists a unique periodic orbit through $p$ for $\mu=\mu_p(r_2,y_0)$ where
\begin{align*}
 \mu_p(r_2,y_0) \equiv \mu_{c}+\mathcal O(e^{-c/\epsilon}).
\end{align*}
 The function $\mu_p(r_2,\cdot)$ is monotonic so that $\mu_p(r_2,y_0)$ approaches $\mu_c(r_2)$ as $y_0$ increases.  
 
The periodic orbits are attracting if
\begin{align}
 \Delta_{VI_3} \equiv 
\frac{2}{3\alpha \beta}\left(\alpha (\eta^-+ \beta \eta^+)+\beta(\beta+1)(\zeta^-+\chi^-)\right),\eqlab{Beqn} 
\end{align}
is negative. They are repelling if $\Delta_{VI_3}$ is positive.  
\end{proposition}
\begin{proof}
To verify the statement about stability we need to compute $\tilde R(y_0)$. To do this we invert $x_1=x_1(y_0,0)$ for $y_0$ and parametrize $\underline x_1$ in terms of $x_1$ rather than $y_0$. Then $\underline x_1(x_1)$ is obtained from the map $\sigma_0^-$ in 
\lemmaref{sigmaMaps} with $\mu=0$:
\begin{align*}
 \underline x_1 =\sigma_0^-(x_1)= -x_1+A^-x_1^2+\mathcal O(x_1^3),
\end{align*}
using the backward flow of $X^-$, 
where $A^-$ is given by \eqref{AN}. We therefore consider the following integral
\begin{align*}
 \tilde R(x_1) &=  \int_{x_1}^{\underline x_1(x_1)} (X_2^+(x,0,0)-X_2^-(x,0,0)) dx \\
 &=\int_{x_1}^{\underline x_1(x_1)} ((1+\beta)x+(\eta^+-\eta^-)x^2+\mathcal O(x^3))dx\\
 &=\left( A^-(1+\beta)+\frac23  (\eta^+-\eta^-)\right)\vert x_1\vert^3+\mathcal O(x_1^4).
\end{align*}
Hence the sign of $\tilde R(x_1)$ is determined by 
\begin{align*}
 A^-(1+\beta)+\frac23(\eta^+-\eta^-)=\frac{2}{3\alpha \beta}\left(\alpha (\eta^-+ \beta \eta^+)+\beta(\beta+1)\zeta^-\right),
\end{align*}
where we have used \eqref{AN}. The right hand side is \eqref{Beqn}. 
Since $\tilde R<0$ implies stability while $\tilde R>0$ implies instability the result follows for $x_1$ (and hence $y_0$) sufficiently small. 
\end{proof}

\begin{remark}
In classical planar slow-fast systems \cite{krupa_relaxation_2001}, a canard is generically associated with a Hopf bifurcation \textit{and} a canard explosion. This is not necessarily the case here. For example the PWS case $VV_1$ has a  singular canard but no local limit cycles in either the PWS system or its regularization. 
%
Conversely, the regularized system can undergo a Hopf bifurcation without the presence of a canard. This is demonstrated by case $II_2^\epsilon$. 
%
\end{remark}

We conclude this subsection with \tabref{tblreg} which summarizes properties of the seven regularized two-folds (compare with \tabref{tblPWS}).\\

\begin{center}
\begin{longtable}{| c | c  c | c  c | c  c |}
\caption{1: type of regularized two-fold singularity. 2: value of $\alpha \delta$. 3: equilibrium (x = no, $\checkmark$ = yes). 4: sign of $\Omega$. 5: Hopf (x = no, $\checkmark$ = yes, n.a. = not applicable). 6: value of $\beta$. 7: canard (x = no, $\checkmark$ = yes).}\tablab{tblreg} \\
\hline
1 & 2 & 3 & 4 & 5 & 6 & 7\\
\hline
 \hline
Type & $\alpha \delta$ & Equilibrium & $\Omega$ & Hopf & $\beta$ & Canard \\
\hline
\hline
$VV_1^\epsilon$ & + & x & + & n.a. & + & $\checkmark$ \\
\hline
$VV_2^\epsilon$ & - & $\checkmark$ & - & x & - & x \\
\hline
$VI_1^\epsilon$ & + & x & $\pm$ & n.a. & - & x\\
\hline
$VI_2^\epsilon$ & - & $\checkmark$ & - & x & + & $\checkmark$ \\
\hline
$VI_3^\epsilon$ & - & $\checkmark$ & + & $\checkmark$ & + & $\checkmark$ \\
\hline
$II_1^\epsilon$ & + & x & - & n.a. & + & $\checkmark$ \\
\hline
$II_2^\epsilon$ & - & $\checkmark$ & + & $\checkmark$ & - & x \\
 \hline
 \end{longtable}
\end{center}

 \subsection{Connecting limit cycles}
Having obtained $\mathcal O(1)$ limit cycles in the two cases $II_2^\epsilon$ and $VI_3^\epsilon$, we now analyze the connection between these limit cycles and those described by \propref{largeCyclesK2} that are due to the Hopf bifurcation. The results are summarized in one of our main results:
\begin{theorem}\label{mainLimitCycles} 
For $\epsilon$ sufficiently small:
\begin{itemize}
  \item[$II_2^\epsilon$:] There exists a $C^k$-smooth family of locally unique periodic orbits of the regularized system \eqref{Xeps} that is due to the Hopf bifurcation in Theorem~\ref{mainHopf}. If $a_2< 0$ ($a_2>0$) where $a_2$ is the first Lyapunov coefficient as defined in \eqref{K2}, then the periodic orbits are attracting (repelling) near the Hopf bifurcation. If 
  \begin{align}
  \Delta_{II_2}=\frac{2}{3\alpha \beta} \left(\alpha (\eta^-+\beta \eta^+)+\beta(\zeta^-+\chi^-+\alpha (\zeta^++\chi^+))\right) \tag*{\eqref{Ceqn}},
  \end{align} is negative (positive) then the periodic orbits for $\Delta_{II_2}^{-1} \mu \le - c\sqrt{\epsilon}$, $c>0$ sufficiently large, are attracting (repelling).
  The periodic orbits for $\mu=\mathcal O(1)$ (with respect to $\epsilon$) are continuously $\mathcal O(\epsilon)$-close to the PWS periodic orbits in \propref{propLimitCyclesPWS}.
  \item[$VI_3^\epsilon$:] There exists a $C^k$-smooth family of \textnormal{small} periodic orbits of the regularized system \eqref{Xeps} that is due to the Hopf bifurcation in Theorem~\ref{mainHopf}. There also exists a $C^k$-smooth family of periodic orbits that are $\mathcal O(1)$ (with respect to $\epsilon$) in amplitude and which undergo a canard explosion, where the amplitude changes by $\mathcal O(1)$ within an exponentially small parameter regime around the canard value $\mu=\sqrt{\epsilon}\mu_{2,c}(\sqrt{\epsilon})$ (see Theorem~\ref{canardTheorem}). If $a_2< 0$ ($a_2>0$) where $a_2$ is the first Lyapunov coefficient as defined in \eqref{K2}, then the periodic orbits are attracting (repelling) near the Hopf bifurcation. If 
 \begin{align}
 \Delta_{VI_3}=\frac{2}{3\alpha \beta}\left(\alpha (\eta^-+ \beta \eta^+)+\beta(\beta+1)(\zeta^-+\chi^-)\right) \tag*{\eqref{Beqn}}
 \end{align} 
 is negative (positive) then the $\mathcal O(1)$-periodic orbits are attracting (repelling).
 \end{itemize}
\end{theorem}
The proof of Theorem~\ref{mainLimitCycles} is divided into two cases: $II_2^\epsilon$ and $VI_3^\epsilon$.

\subsubsection{Connecting limit cycles for case $II_2^\epsilon$ using charts $\kappa_{1,3}$}\seclab{LCCII2}


To connect the periodic orbits described in chart $\kappa_2$ by \propref{largeCyclesK2} with the $\mathcal O(1)$ periodic orbits in \propref{propLimitCyclesReg} we first return to the original $(x,y)$ variables in which the region of regularization is $y\in (-\epsilon,\epsilon)$ and consider the mappings $\sigma_\epsilon^{\pm}$ taking $(x_0,\epsilon)$ to $(\sigma_\epsilon^+(x_0),\epsilon)$ and $(x_1,-\epsilon)$ to $(\sigma_\epsilon^-(x_1),-\epsilon)$, respectively. See \figref{crossingPOSReg}. 
{We have:
\begin{lemma}\lemmalab{sigmaEpsilonMap0}
Fix $c>0$ large. The maps $\sigma_\epsilon^{\pm}$ are defined for 
\begin{align}
x\in (x_f^{+}(\mu,\epsilon),c^{-1}]&\quad \text{where}\quad x_f^+(\mu,\epsilon) \equiv -\chi^+\epsilon + \mathcal O(\epsilon (\mu+\epsilon)),\eqlab{xfPlus}\\
&\text{and}\nonumber\\
 x\in [-c^{-1},x_f^-(\mu,\epsilon))&\quad \text{where}\quad x_f^-(\mu,\epsilon)  \equiv \mu-\beta^- \chi^- \epsilon + \mathcal O(\epsilon (\mu+\epsilon)),\eqlab{xfMinus}
\end{align}
respectively and for those $x$ the maps $\sigma_\epsilon^{\pm}$ satisfy
\begin{align}
 \sigma_\epsilon^{\pm} = \sigma_0^{\pm}+\mathcal O(\epsilon),\eqlab{sigmaEpsilonMap}
\end{align}
where $\sigma_0^\pm$ are described in \lemmaref{sigmaMaps}.  
\end{lemma}
\begin{proof}
 The mappings $\sigma_0^\pm$ map $\{y=0\}$ to itself by the forward flow of $X^{\pm}$ for $x>0$ and $x<\mu$, respectively. The mappings $\sigma_\epsilon^{\pm}$, on the other hand, map $\{y=\pm \epsilon\}$ to itself by the forward flow of $X^{\pm}$, respectively. Here $\sigma_\epsilon^+$ is defined for $x>x_f^+(\mu,\epsilon)$ where $X_2^+(x_f^+(\mu,\epsilon),\epsilon,\mu)=0$ while $\sigma_\epsilon^+$ is defined for $x<x_f^-(\mu,\epsilon)$ where $X_2^-(x_f^-(\mu,\epsilon),-\epsilon,\mu)=0$. Using \eqsref{normplus2}{normminus2} and the implicit function theorem gives \eqsref{xfPlus}{xfMinus}, respectively, for $c$ sufficiently large. Equation \eqref{sigmaEpsilonMap} therefore follows by standard regular perturbation theory.
\end{proof}}
\begin{remark}\remlab{sigma13Remark}
 {The mappings $\sigma_{\epsilon}^{\pm}$ do not depend upon the regularization. They are due to \eqref{SotomayorProperty} determined by $X^{\pm}$.}
\end{remark}
\begin{figure}
\begin{center}
\includegraphics[width=.75\textwidth]{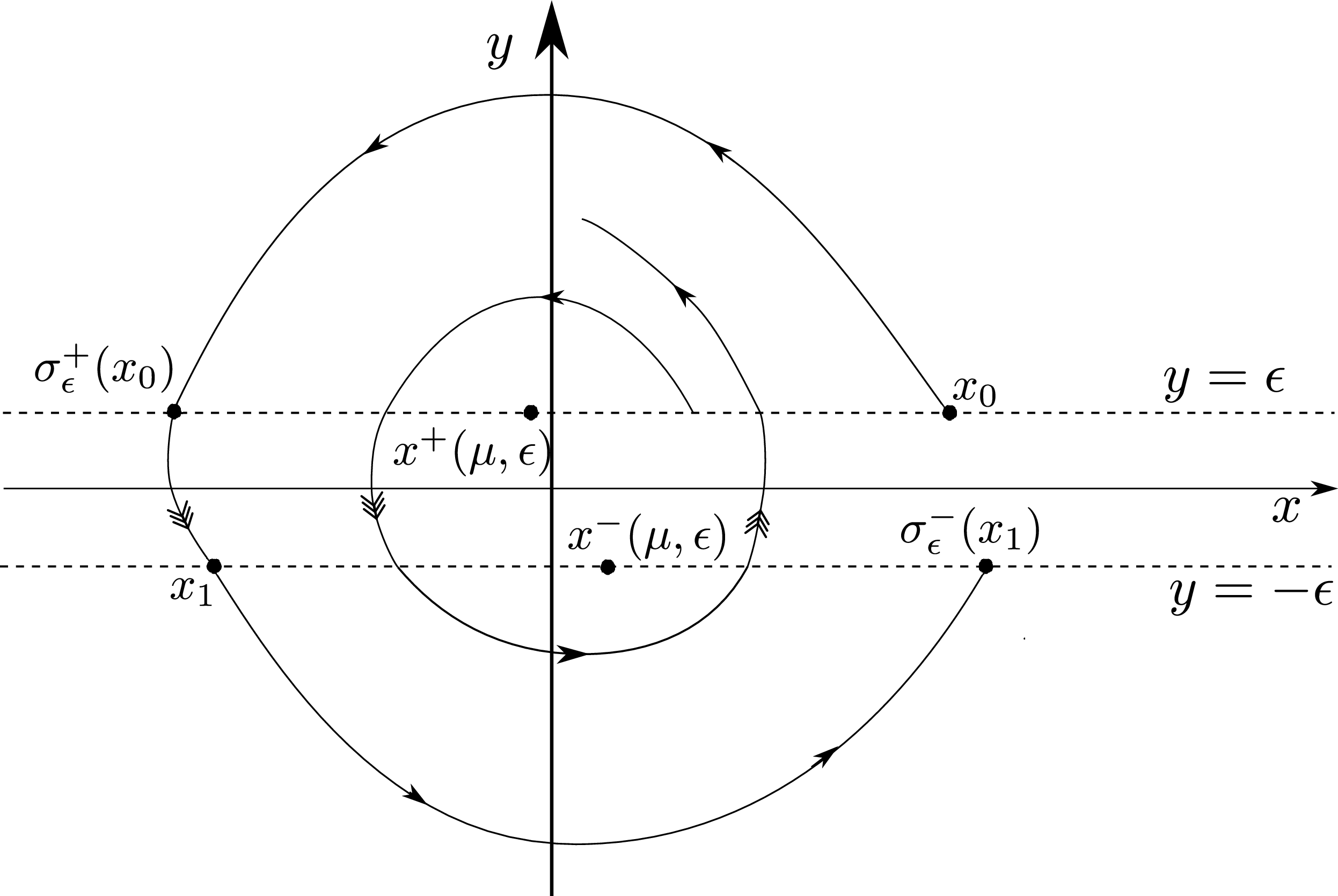}
\caption{The mappings $\sigma_\epsilon^\pm$ associated with the regularization of the invisible two-fold singularity $II_2$, where $\delta=-1$, $\alpha>0$ and $\beta<0$. The points $(x,y)=(x^{\pm},\pm \epsilon)$ are points where $X_2^{\pm}(x,\pm \epsilon,\mu)=0$. Note that the mappings $\xi_1^-$ and $\xi_3^+$ in \eqsref{xi1Negative}{xi3Positive} describe the fast transitions indicated in the figure by the triple-headed arrows, using the charts $\kappa_1$ and $\kappa_3$. These transitions would appear nearly vertical in the stretched coordinate system $(x,\hat y)$. }\figlab{crossingPOSReg}
\end{center}
\end{figure}

We now write these mappings $\sigma_\epsilon^\mp$ in terms of the charts $\kappa_{1,3}$. The resulting mappings will be denoted by 
\begin{align*}
\sigma_{31}^-=\kappa_3 \circ \sigma_\epsilon^- \circ \kappa_1^{-1},
\end{align*}
and 
\begin{align*}
\sigma_{13}^+ = \kappa_1 \circ \sigma_\epsilon^- \circ \kappa_3^{-1},
\end{align*}
respectively, using the subscripts to highlight that these mappings are from $\kappa_1$ ($\kappa_3$) to $\kappa_3$ ($\kappa_1$), respectively. 
\begin{lemma}\lemmalab{sigmaEpsilonMap}
Consider \eqref{II2Assumption}.
 In terms of the charts $\kappa_{1,3}$ the mappings $\sigma_\epsilon^{\mp}$ take the following forms:
 \begin{align*}
  \sigma_{31}^-:\quad \kappa_1 \cap \{\hat y=-1\}&\rightarrow \kappa_3 \cap \{\hat y=-1\},\\
  r_3 &=r_1(1+2\mu_1+A^- r_1+\mathcal O(r_1(\epsilon_1+\mu_1+r_1)),\\
  \epsilon_3 &=\epsilon_1(1-2(2\mu_1+A^-r_1)+\mathcal O(r_1(\epsilon_1+\mu_1+r_1)+\mu_1^2)),\\
  \mu_3 &=\mu_1(1-(2\mu_1+A^-r_1)+\mathcal O(r_1(\epsilon_1+\mu_1+r_1)+\mu_1^2)),
 \end{align*}
and
 \begin{align*}
  \sigma_{13}^+:\quad \kappa_3 \cap \{\hat y=1\}&\rightarrow \kappa_1 \cap \{\hat y=1\},\\
  r_1 &=r_3(1-A^+ r_3+\mathcal O(r_3(\epsilon_3+\mu_3+r_3)),\\
  \epsilon_1 &=\epsilon_3(1+2A^+r_3+\mathcal O(r_3(\epsilon_3+\mu_3+r_3))),\\
  \mu_1 &=\mu_3(1+A^+r_3+\mathcal O(r_3(\epsilon_3+\mu_3+r_3))),
 \end{align*}
respectively, for $r_1$, $\epsilon_1$ and $\mu_1$ sufficiently small. 

\end{lemma}
\begin{proof}
 Consider the case $\sigma_\epsilon^-$ (the case $\sigma_\epsilon^+$ is identical). We then use \eqref{sigmaEpsilonMap} and \lemmaref{sigmaMaps} and set $x_1=-r_1$ and $x_3=r_3$ as described by the charts $\kappa_{1,3}$, respectively. The expressions for $\epsilon_3$ and $\mu_3$ then follow from the conservation of $\epsilon$ and $\mu$, respectively. The condition \eqsref{xfPlus}{xfMinus} are satisfied for $r_3$, $\epsilon_3$, $\mu_3$ and $r_1$ $\epsilon_1$, $\mu_1$, respectively, sufficiently small.
\end{proof}

We then consider the Poincar\'e mapping $P_\epsilon$ from \eqref{Pepsilon} used in \propref{propLimitCyclesReg} and write this mapping in chart $\kappa_3$. The resulting mapping is given by
\begin{align}
 P_3=\kappa_3 \circ P_\epsilon \circ \kappa_3^{-1} :\quad \kappa_3 \cap \{\hat y=1\}&\rightarrow \kappa_3 \cap \{\hat y=1\},\eqlab{P_3}\\
 (r_3,\epsilon_3,\mu_3)&\mapsto (r_3^+,\epsilon_3^+,\mu_3^+),\nonumber
\end{align}
We will compose $P_3$ into four different mappings $\sigma_{13}^+$, $\xi_1^-$, $\sigma_{31}^-$ and $\xi_3^+$ so that:
$$P_3=\xi_3^+\circ \sigma_{31}^-\circ \xi_1^- \circ \sigma_{13}^+.$$
The mappings $\sigma_{13}^+$ and $\sigma_{31}^-$ are described in \lemmaref{sigmaEpsilonMap}
while $\xi_{1}^{-}$ and $\xi_3^+$ are given in terms of the forward flow associated with the differential equations in charts $\kappa_1$ and $\kappa_3$ (see \eqsref{EqsChartKappa1}{EqsChartKappa3}) and map $\{y=\pm \epsilon\}$ ($\{\hat y=\pm 1\}$) to $\{y=\mp \epsilon\}$ ($\{\hat y=\mp 1\}$), respectively. Hence the mappings:
\begin{align}
 \xi_1^-:\{\hat y=1\}&\rightarrow \{\hat y=-1\},\eqlab{xi1Negative}\\
 (r_1,\epsilon_1,\mu_1)&\mapsto (r_1^-,\epsilon_1^-,\mu_1^-),
 \nonumber
 \end{align}
 and
 \begin{align}
 \xi_3^+:\{\hat y=-1\}&\rightarrow \{\hat y=1\},\eqlab{xi3Positive}\\
 (r_3,\epsilon_3,\mu_3)&\mapsto (r_3^+,\epsilon_3^+,\mu_3^+),
 \nonumber
\end{align}
are defined by the forward flow of the following equations:
\begin{align}
\dot r_1 &=-r_1 \epsilon_1 \tilde F_1(r_1,\hat y,\epsilon_1,\mu_1),\eqlab{newEqsChartKappa1}\\
\dot{\hat y} &= (-1+\mathcal O(r_1))(1+\phi(\hat y))+(\beta(1+ \mu_1)+\mathcal O(r_1))(1-\phi(\hat y)),\nonumber\\
\dot \epsilon_1 &=2\epsilon_1^2 \tilde F_1(r_1,\hat y,\epsilon_1,\mu_1),\nonumber\\
\dot \mu_1 &= \epsilon_1 \mu_1 \tilde F_1(r_1,\hat y,\epsilon_1,\mu_1),\nonumber\\ 
\end{align}
and
\begin{align}
 \dot r_3 &=r_3 \epsilon_3\tilde F_3(r_3,\hat y,\epsilon_3,\mu_3),\nonumber\\
 \dot{\hat y} &=(1+\mathcal O(r_3))(1+\phi(\hat y))+(-\beta(1-\mu_3)+\mathcal O(r_3))(1-\phi(\hat y))),\nonumber\\
 \dot \epsilon_3 &=-2\epsilon_3^2\tilde F_3(r_3,\hat y,\epsilon_3,\mu_3),\nonumber\\
 \dot \mu_3&=-\epsilon_3\mu_3  \tilde F_3(r_3,\hat y,\epsilon_3,\mu_3),\nonumber
 \end{align}
 respectively. These equations are equations \eqsref{EqsChartKappa1}{EqsChartKappa3}, respectively, written in terms of $\hat y$ rather than $w$, where
 \begin{align*}
 \tilde F_1(r_1,\hat y,\epsilon_1,\mu_1) &= (-1+\mathcal O(r_1))(1+\phi(\hat y))+(\alpha +\mathcal O(r_1))(1-\phi(\hat y)),\\
 \tilde F_3(r_3,\hat y,\epsilon_3,\mu_3) &= (-1+\mathcal O(r_3))(1+\phi(\hat y))+(\alpha +\mathcal O(r_3))(1-\phi(\hat y)).
\end{align*}
\begin{lemma}\lemmalab{xipmMap}
 Consider \eqref{II2Assumption} and let
 \begin{align*}
  E&=\int_{-1}^1 \frac{-(1+\phi(\hat y))+\alpha(1-\phi(\hat y))}{1+\phi(\hat y)-\beta(1-\phi(\hat y))}d\hat y.
 \end{align*}
 Then the mappings $\xi_1^-$, $\xi_3^+$ in \eqsref{xi1Negative}{xi3Positive} take the following form:
 \begin{align*}
  \xi_1^-:\,r_1^- &= r_1(1-E\epsilon_1+\mathcal O(r_1\epsilon_1)),\\
  \epsilon_1^-&=\epsilon_1(1+2E\epsilon_1 +\mathcal O(r_1\epsilon_1)),\\
  \mu_1^- &=\mu_1(1+E\epsilon_1 +\mathcal O(r_1\epsilon_1)),
  \end{align*}
  and 
   \begin{align*}
  \xi_3^+:\,r_3^+ &= r_3(1+E\epsilon_3+\mathcal O(r_3\epsilon_3)),\\
  \epsilon_3^+&=\epsilon_3(1-2E\epsilon_3 +\mathcal O(r_3\epsilon_3)),\\
  \mu_3^+ &=\mu_3(1-E\epsilon_3 +\mathcal O(r_3\epsilon_3)).
  \end{align*}
\end{lemma}
\begin{proof}
 Consider $\xi_1^-$ and \eqref{newEqsChartKappa1} (the analysis for $\xi_3^+$ is identical). Since $\beta<0$ we have that $\dot{\hat y}<0$ for $\hat y\in [-1,1]$ for $r_1=\mu_1=0$ and hence we can replace time by $\hat y$ by dividing the equations for $\dot r_1$, $\dot \epsilon_1$ and $\dot \mu_1$ by $\dot{\hat y}$. The point $(r_1,\epsilon_1,\mu_1)=0$ is a fixed point of these equations. Solving the second order variational equations gives the desired result.
\end{proof}

We then have:
\begin{lemma}
 Consider \eqref{II2Assumption}. The Poincar\'e map $P_3:\,(r_3,\epsilon_3,\mu_3)\mapsto (r_3^+,\epsilon_3^+,\mu_3^+)$ takes the following form:
 \begin{align}
  P_3:\,r_3^+ &= r_3(1+2\mu_3+\Delta_{II_2}r_3+\mathcal O(r_3(\epsilon_3+\mu_3+r_3)), \eqlab{r3Plus}\\
  \epsilon_3^+&=\epsilon_3(1-2(2\mu_3+\Delta_{II_2}r_3)+\mathcal O(r_3(\epsilon_3+\mu_3+r_3)),\nonumber\\
  \mu_3^+&=\mu_3(1-(2\mu_3+\Delta_{II_2} r_3)+\mathcal O(r_3(\epsilon_3+\mu_3+r_3)),\nonumber
 \end{align}
where $\Delta_{II_2}$ is given by \eqref{Ceqn}.
\end{lemma}
\begin{proof}
 We use \lemmaref{xipmMap} and \lemmaref{sigmaEpsilonMap}.
\end{proof}
\begin{remark}
 Note that to leading order \eqref{r3Plus} is independent of $E$ and hence of $\phi$, the regularization function. Hence $\phi$ does not induce bifurcations in the transition from sufficiently large limit cycles (meaning $\nu$ sufficiently small) in chart $\kappa_2$ to the $\mathcal O(1)$ limit cycles. 
\end{remark}

Next, we solve for fixed points of $P_3$ and obtain:
\begin{proposition}\proplab{fixPointP3}
 Suppose that $\Delta_{II_2}$ in \eqref{Ceqn} is non-zero. Then for $\mu_3,\,\epsilon_3$ sufficiently small and $\mu_3\Delta_{II_2}^{-1}<0$, the mapping $P_3$ has a locally unique family of fixed points:
 \begin{align}
  r_3&=-2\mu_3 \Delta_{II_2}^{-1}+\mathcal O(\mu_3(\epsilon_3+\mu_3)).\eqlab{r3FixPoint}
 \end{align}
 The family of fixed point of $P_3$ corresponds to a $C^k$-smooth family of periodic orbits which are attracting (repelling) for $\Delta_{II_2}<0$ ($\Delta_{II_2}>0$).
\end{proposition}
\begin{proof}
 Suppose that $(r_3,\mu_3,\epsilon_3)$ is a fixed point of $P_3$. Then since $\mu=r_3\mu_3$ and $\epsilon=r_3\epsilon_3$ we solve for $r_3=r_3(\mu_3,\epsilon_3)$ by setting $r_3^+=r_3$ in \eqref{r3Plus}. This gives
 \begin{align*}
  2\mu_3+\Delta_{II_2}r_3+\mathcal O(r_3(\epsilon_3+\mu_3+r_3))=0.
 \end{align*}
We solve this equation by the implicit function theorem and obtain \eqref{r3FixPoint}. The statement about stability follows from the fact that the sign of $\Delta_{II_2}$ determines the sign of $\partial_{r_3}r_3^+=r_3(\Delta_{II_2}+\mathcal O(\epsilon_3+\mu_3+r_3))$.
\end{proof}

The periodic orbits in chart $\kappa_2$, described by \propref{largeCyclesK2}, that are due to the Hopf bifurcation, are locally unique since they are obtained by the implicit function theorem for $\mu_2=r_2=0$. These orbits can be continued all the way up to the section $\Lambda_2^+$, defined in \eqref{Lambda2pm} (see also \remref{POBeyondLambda2}). The periodic orbits that are $\mathcal O(1)$ with respect to $\mathcal O(\epsilon)$, described by \propref{propLimitCyclesReg}, are also locally unique by virtue of the implicit function theorem. Therefore by setting $\epsilon_3=\nu^2$, corresponding to the section $\Lambda_2^+$ \eqref{Lambda2pm}, and taking $\nu$ sufficiently small, we obtain:
  \begin{align}
 x_2=\nu^{-1},\,\hat y=1,\,\mu_2 = -\frac{\Delta_{II_2}}{2\nu^2}r_2+\mathcal O(r_2^2),\eqlab{P3Lambda2Plus}
 \end{align}
using \eqsref{kappa32}{r3FixPoint}. Therefore we can conclude that the periodic orbits described by \eqref{r3FixPoint} coincide with the locally unique ones in chart $\kappa_2$ described by \propref{largeCyclesK2}. Similarly, setting $r_3=\rho$, corresponding to section $\Lambda^+$ \eqref{Lambdapm}, shows that the periodic orbits in \eqref{r3FixPoint} coincide with those in \propref{propLimitCyclesReg}, where these are defined. This gives a $C^k$-smooth and locally unique family of periodic orbits as described by Theorem~\ref{mainLimitCycles}, case $II_2^\epsilon$. 

\begin{remark}\remlab{novelityCharts}
We believe that the application of the directional charts $\kappa_{1,3}$ in this section to describe the Poincar\'e-mapping $P_\epsilon$ is a novelty. The coordinates of $\kappa_{1,3}$ enabled us to connect the periodic orbits in chart $\kappa_2$ with the larger periodic orbits without the need for careful estimation. This is a general advantage of the blowup method and the phase directional charts $\kappa_{1,3}$. Having said that, it might be possible to prove the connection of the limit cycles in $\kappa_2$ with the larger limit cycles in \propref{propLimitCyclesReg} by working in chart $\kappa_2$ alone. To do this one would, however, have to perform a careful estimation of the function $D$ in \eqref{DistanceFunction}. The following remark, \remref{finalRemark}, contains a discussion of this issue.
\end{remark}
\begin{remark}\remlab{finalRemark}
 By \eqref{kappa32} it follows that the family of fixed points of $P_3$ described in \propref{fixPointP3} intersects $\Lambda_2^+$ in 
\eqref{P3Lambda2Plus}.
By Theorem~\ref{mainLimitCycles}, case $II_2^\epsilon$, this fixed point of $P_3$ corresponds to a periodic orbit obtained from \propref{largeCyclesK2} for a corresponding value of the energy constant $h$. Hence the value of $\mu_2$ in \eqref{P3Lambda2Plus} must agree with the value given in \eqref{mu2largeCyclesK2}. This is a corollary of Theorem~\ref{mainLimitCycles}, case $II_2^\epsilon$. One may expect that there could be a more direct way of showing this. We now outline a formal derivation of the result by approximation the integrals in $D_{r_2}$ and $D_{\mu_2}$ in \eqref{mu2largeCyclesK2}. 

Since $x_2=\nu^{-1}$ in \eqref{P3Lambda2Plus} we take $h=\frac{1}{2\nu^2}$. To compute $D_{r_2}$ \eqref{Dr2} we first substitute $dt = ((1+\phi)-\beta(1-\phi))x_2dy$ and integrate from $\hat y_0^h$ to $\hat y_1^h$. We then split this integration into (a) an integration from $\hat y=\hat y_0^h$ to $\hat y=0$ and (b) an integration from $\hat y=0$ to $\hat y=\hat y_0^h$. We then ignore the contribution from the region of regularization and simply set $\phi=\mp 1$ in the integrations (a) and (b), respectively. We apply the same approximation to the Hamiltonian $H(x_2,\hat y)$ and obtain the value of $\hat y_0^h$ and $\hat y_1^h$  from the equation $H(0,\hat y)=h$. Combining this gives the following approximation of $D_{r_2}$:
\begin{align*}
 D_{r_2} \approx \frac{2}{3\nu^3\beta \alpha}\left(\alpha(\eta^-+\beta \eta^+)+\beta (\zeta^-+\chi^-+\alpha(\zeta^++\chi^+))\right),
\end{align*}
which is valid for $\nu$ small. Hence
\begin{align*}
 D_{r_2} \approx \frac{\Delta_{II_2}}{\nu^3},
\end{align*}
 from \eqref{Ceqn}. For $D_{\mu_2}$, as defined in \eqref{Dmu2}, we use \eqref{Dmu2Large} and approximate $x_2$ by the constant value $\nu^{-1}$. This gives
\begin{align*}
 D_{\mu_2} \approx \frac{2}{\nu}.
\end{align*}
Hence, from \eqref{mu2largeCyclesK2}, we have
\begin{align*}
 \mu_2 \approx -\frac{\Delta_{II_2}}{2\nu^2}r_2,
\end{align*}
which agrees with \eqref{P3Lambda2Plus}. We have not pursued a rigorous result of this kind. 
\end{remark}
%

\subsubsection{Connecting limit cycles for case $VI_3^\epsilon$}\seclab{LCCVI3}
Krupa and Szmolyan \cite{krupa_relaxation_2001} describe the classical canard explosion phenomenon, as observed in the van der Pol system. They prove that the periodic orbits within their chart $\kappa_2$ belong to a smooth and unique family of local periodic orbits that also includes $\mathcal O(1)$ periodic orbits that arise from their canard explosion. 
Their proof involves careful estimation on the dependency of the function $D(r_2,\mu_2,h)$, as defined in \eqref{DistanceFunction}, on the distance to the singular canard (measured by the energy constant $h$). %
It seems plausible that a similar analysis could be performed to our system. However, the situation here is complicated by the fact that our Hamiltonian function $H$ depends non-trivially on the regularization function $\phi$. The statements in Theorem~\ref{mainLimitCycles}, case $VI_3^\epsilon$, therefore summarize the previous results in \propref{largeCyclesK2} and \propref{canardExplosion}. Instead we conjecture on the connection of small periodic orbits in \propref{largeCyclesK2} with the larger ones in \propref{canardExplosion} as follows:
\begin{con}\label{mainConjecture1}
The two families of periodic orbits in $VI_3^\epsilon$  belong to the same family of locally unique periodic orbits. 
\end{con} 

The limit cycles in \propref{largeCyclesK2} and \propref{canardExplosion} do not seem to be present in the PWS case $VI_3$.

\subsection{Saddle-node bifurcation}
We conclude this section with the following main result
\begin{theorem}\label{saddleNode}
 Suppose \eqref{assumption1} and \eqref{limitCycleAssumption}. Then for $\epsilon$ sufficiently small:
  \begin{itemize}
   \item[$II_2^\epsilon$:] There exists an open set of regularization functions such that the periodic orbits in Theorem~\ref{mainLimitCycles}, case $II_2^\epsilon$, undergo at least one saddle-node bifurcation.
   \item[$VI_3^\epsilon$:] Suppose, in addition, that Conjecture~\ref{mainConjecture1} holds. Then there exists an open set of regularization functions such that the periodic orbits in Theorem~\ref{mainLimitCycles}, case $VI_3^\epsilon$, undergo at least one saddle-node bifurcation.
  \end{itemize}
\end{theorem}
\begin{proof}
A corollary of Theorem~\ref{K2sign} is that we can always achieve $a_2\Delta_{II_2}<0$ or $a_2\Delta_{VI_3}<0$. The result therefore follows from the statements in Theorem~\ref{mainLimitCycles} and Conjecture~\ref{mainConjecture1}. 
\end{proof}

\section{Numerics}\seclab{numerics}
In this section we illustrate the results in Theorem~\ref{saddleNode}, and provide further support for Conjecture~\ref{mainConjecture1}, by computing limit cycles for two model systems, for cases $II_2^\epsilon$ and $VI_3^\epsilon$. 
\subsection{Case $II_2^\epsilon$}\seclab{II2Numerics}
In this section we consider the regularization of the following model system for case $II_2$:
\begin{align}
 X^+(x,y) = \begin{pmatrix}
             -1-7x\\
             x+2x^2
            \end{pmatrix},\quad X^-(x,y,\mu) = \begin{pmatrix}
            1-6 x\\
            (x-\mu)-2(x-\mu)^2
            \end{pmatrix},\eqlab{II2model}
\end{align}
corresponding to the following parameters:
\begin{align*}
 \delta = -1,\,\alpha = 1,\,\beta = -1,\,\zeta^+=-7,\,\zeta^-=-6,\,\eta^{\pm} = \pm 2,\,\chi^\pm = 0,\,\Omega=2,
\end{align*}
in \eqsref{normplus2}{normminus2}. 
The constant $\Delta_{II_2}$ given in \eqref{Ceqn} takes the value:
\begin{align*}
 \Delta_{II_2} = -6.
\end{align*}
According to \propref{propLimitCyclesPWS} the limit cycles of the PWS system are therefore all stable. According to Theorem~\ref{mainLimitCycles} the $\mathcal O(1)$ (with respect to $\epsilon$) limit cycles of the regularized system are also stable. 

We consider two regularization functions\footnote{The superscripts $l$ and $c$ refer to linear and cubic, respectively.}:
\begin{align}
 \phi^l(\hat y) =\hat y\quad \text{for}\quad \hat y\in (-1,1),\eqlab{phiLinear}
\end{align}
and 
\begin{align}
 \phi^c(\hat y) = \frac32\hat y-\frac12 \hat y^3 \quad \text{for}\quad  \hat y\in (-1,1).\eqlab{phiCubic}
\end{align}
From \eqref{eqy20Star} we obtain $\hat y_0^*=0$ in both cases. Inserting the corresponding values of $\phi_{1,H}=\phi'(0),\, \phi_{2,H}=\phi''(0)=0,\,\phi_{3,H}=\phi^{(3)}(0)$ into \eqsref{mu2H}{K2} gives the following values for $\mu_{2,H}$ and $a_2$:
\begin{align}
 \mu_{2,H}^l &= 13r_2+\mathcal O(r_2^2),\eqlab{mu2Hlc}\\
 a_2^l&=\frac12,\nonumber\\
 \mu_{2,H}^c &= \frac{26}{3}r_2+\mathcal O(r_2^2)= (8+\frac23) r_2 + \mathcal O(r_2^2),\nonumber\\
 a_2^c  &= -\frac12.\nonumber
\end{align}
Since $a_2^l\Delta_{II_2} <0$ we have from Theorem~\ref{saddleNode} that the \textit{linear} regularization function \eqref{phiLinear} introduces a saddle-node bifurcation. We demonstrate this as follows.

Using the numerical bifurcation software AUTO we continued the periodic orbits in the two regularizations of $II_2$. In \figref{II2Amplitude} we show the amplitude (measured as $\max(\hat y)$) of the periodic orbits as a function of the parameter $\mu_2$ for $r_2=\sqrt{\epsilon}=0.01$. The full line shows the result of using $\phi_l$, as given in \eqref{phiLinear}, while the dotted line shows the result of using $\phi_c$, as given in \eqref{phiCubic}. The linear regularization function introduces a saddle-node bifurcation, in agreement with Theorem~\ref{saddleNode}. On the other hand, the \textit{cubic} regularization function does not introduce any saddle-node bifurcation. 

The Hopf bifurcations were numerically found to occur at
\begin{align*}
 \mu_{2,H}^l \approx 0.1298,\quad \mu_{2,H}^c \approx 0.0866,
\end{align*}
which are in good agreement with \eqref{mu2Hlc} for $r_2=0.01$.
Also, in agreement with \propref{propLimitCyclesReg}, we observed that the two family of limit cycles agree for larger values of $\mu_2$ since the limit cycles for both regularizations must be $\mathcal O(\epsilon)$-close to the limit cycles of the PWS system for $\mu=\mathcal O(1)$.

\begin{figure}\begin{center}\includegraphics[width=.65\textwidth]{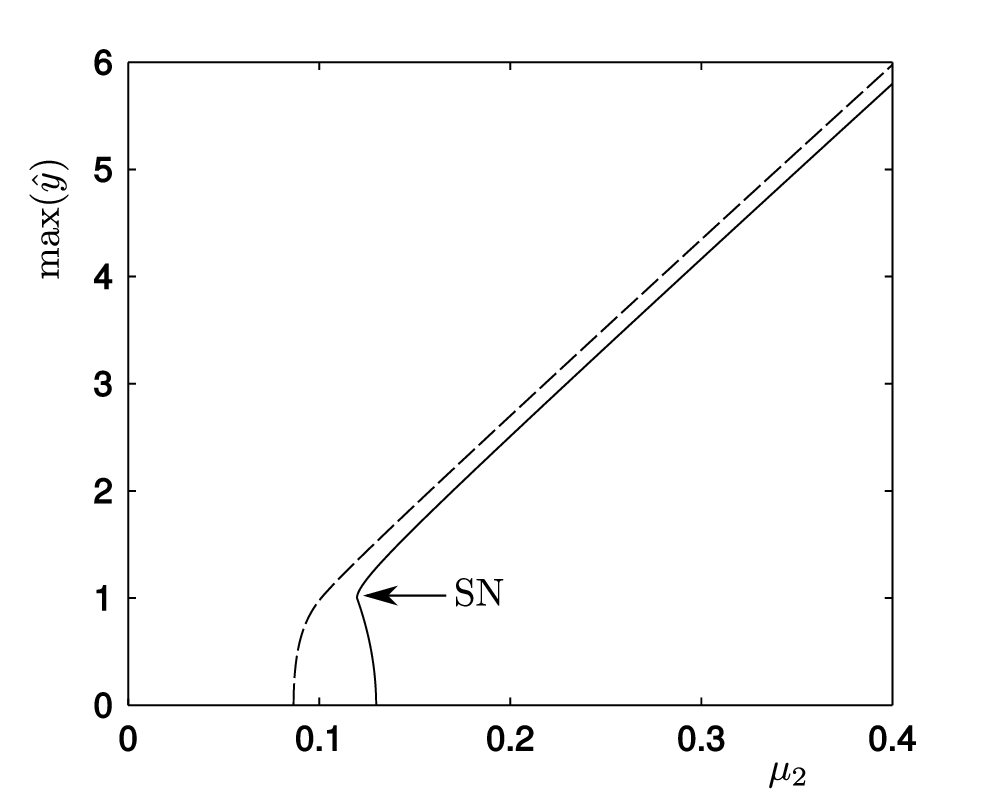}
\caption{Periodic orbit amplitudes as a function of the parameter $\mu_2$ for $r_2=0.01$. The full line shows the result of using the linear regularization function \eqref{phiLinear}, the dotted line shows the result of using the cubic regularization function \eqref{phiCubic}. As can be seen, the linear regularization function induces a saddle-node (SN) bifurcation, in agreement with Theorem~\ref{saddleNode}. 
}\figlab{II2Amplitude}
                		\end{center}
              \end{figure}

\begin{remark}
 We obtained system \eqref{II2model} by fixing \eqsref{phiLinear}{phiCubic} and the parameters $\delta,\,\alpha,\beta$ and solving $a_2^l=-a_2^c$ for $\zeta^\pm$ and $\eta^\pm$ setting $\chi^\pm=0$ for simplicity. 
\end{remark}

\subsection{Case $VI_3^\epsilon$}\seclab{numericsVI3}
In this section we consider the following model system for the case $VI_3$:
\begin{align}
 X^+ = \begin{pmatrix} 
        1 + \frac12 x\\
        x-x^3
       \end{pmatrix},\quad X^- = \begin{pmatrix}
       -1\\
       -2(x-\mu)+(x-\mu)^2
       \end{pmatrix}.
       \eqlab{system1}
\end{align}
Here 
\begin{align*}
 \delta = 1,\,\alpha = -1,\,\beta = 2,\,\zeta^+=\frac12,\,\zeta^-=0,\,\eta^+ = 0,\,\eta^-=1,\,\chi^\pm = 0,\,\Omega=1.
\end{align*}
We sketch the PWS system in \figref{sketch} for $\mu=0$. It is very similar to Fig. 12 in \cite{Kuznetsov2003}. However, as opposed to \cite{Kuznetsov2003} we have included the cubic term in $X^+$ which gives rise to an invisible tangency at $(x,y)=(1,0)$ and a return mechanism from $\Sigma^+$ to $\Sigma^-$.
\begin{figure}\begin{center}\includegraphics[width=.65\textwidth]{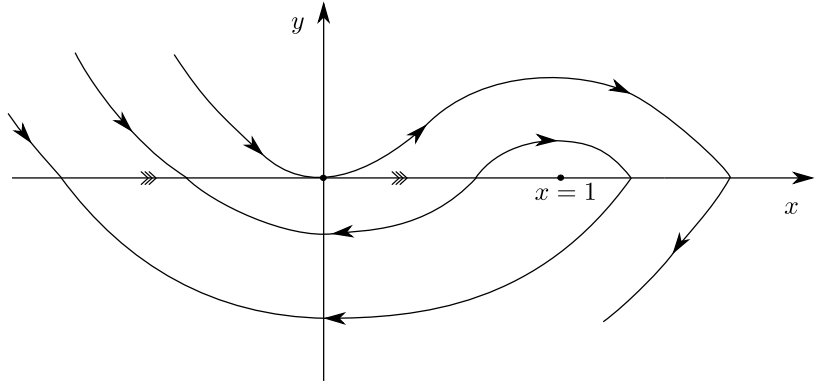}
\caption{Sketch of the PWS system \eqref{system1} for $\mu=0$. The triple-headed arrows within $\Sigma$ indicate the direction of the sliding vector field.}\figlab{sketch}
                		\end{center}
              \end{figure}

We then regularize $X$ in \eqref{system1} using the cubic function in \eqref{phiCubic}. Since $\Omega=1>0$, we can apply \propref{HopfBifurcationChartKappa2} to the regularized system and conclude that the system has an equilibrium \eqref{eqx2y2}
which undergoes a Hopf bifurcation at
\begin{align}
 \mu_2=\mu_{2,H}^c\equiv -\frac{1}{12}r_2+\mathcal O(r_2^2).\eqlab{mu2Hsystem1}
\end{align}
The first Lyapunov coefficient is obtained from \eqref{K2}:
\begin{align}
a&= -\frac{5}{64}r_2+\mathcal O(r_2^2).\eqlab{a2csystem1}
\end{align}
Since $a_2^c\equiv -\frac{5}{64}<0$ we conclude that the periodic orbits are attracting for $r_2$ (and hence $\epsilon$) sufficiently small and appear for $\mu_2>\mu_{2,H}^c$. 

This example also has a maximal canard (see Theorem~\ref{canardTheorem}). The parameter value at which it occurs is, from  \eqref{muc}, given by
\begin{align}
 \mu_{2,c}^c &= -\frac{1}{9\phi_{1,c}} r_2+\mathcal O(r_2^2)\approx -0.07806 r_2 + \mathcal O(r_2^2),\eqlab{mu2cSystem1New}
\end{align}
In the last expression we have used \eqref{phiy2c} to obtain
\begin{align*}
\phi_{1,c}
\approx 1.4233.
\end{align*}


There is a related example for the case $VI_3$ in \cite{Kuznetsov2003} on p. 2169 (after reversing time and reflecting $x\mapsto -x$) with the same values of $\delta$, $\alpha$ and $\beta$. The system in \cite{Kuznetsov2003} has $\zeta^{\pm} = \mp 1$ as the only non-zero coefficients in \eqref{zetaeta}. The reason for modifying the system given in \cite{Kuznetsov2003} is that their system gives $a_2=0$ from \eqref{K2}, for all regularization functions $\phi$. In fact a detailed calculation shows that $a\equiv 0$. The example in \cite{Kuznetsov2003} is therefore codimension two for the regularization.

In \figref{amplitude} we have used the numerical bifurcation software AUTO to track the amplitudes of the limit cycles of \eqref{system1} emanating from the equilibrium \eqref{eqx2w}. We considered $r_2=0.1$. The amplitude of the limit cycles is now measured in \figref{amplitude} using $\text{max}(x)$ instead of $\max(\hat y)$ used above. This proved to be more illustrative in this case. A dramatic increase in amplitude is seen near $\mu_{2}\approx -7.836574\times 10^{-3}$.  In \figref{duck} we have illustrated three different limit cycles within the original $(x,y)$-plane. The largest limit cycle looks like a \textit{canard}. The three limit cycles occur for the following parameters:
              \begin{align*}
              \mu_2& = -7.85\times 10^{-3},\\
              \mu_2&=-7.8365738\times 10^{-3},\\
            \mu_2&=-7.8365737\times 10^{-3}.
              \end{align*} 
The difference between the last two parameters is $10^{-11}$. The dramatic increase of amplitude is due to the canard explosion phenomenon described in \propref{canardExplosion}.
Numerically we found the following canard value
\begin{align}
 \mu_{2,c}^c= -7.8365738\times 10^{-3}.\eqlab{mu2cNum}
\end{align}
This value is in good agreement with \eqref{mu2cSystem1New} for $r_2=0.1$. Note that in comparison to the classical \textit{canard} relaxation oscillation in the van der Pol system, the duck's head and chest are in our case, not due to motion along a curved slow manifold. Instead they are due to regular motion within $\Sigma_{\pm}$ following the regular vector fields $X^\pm$, respectively. It is the motion along the slow manifold that creates the straight back of the duck. Also in the present case these different types of motions occur on an identical time-scale. There is a slow-fast behaviour but it is hidden and only visible through the scaling $\hat y = y/\epsilon$.  



\begin{figure}\begin{center}\includegraphics[width=.65\textwidth]{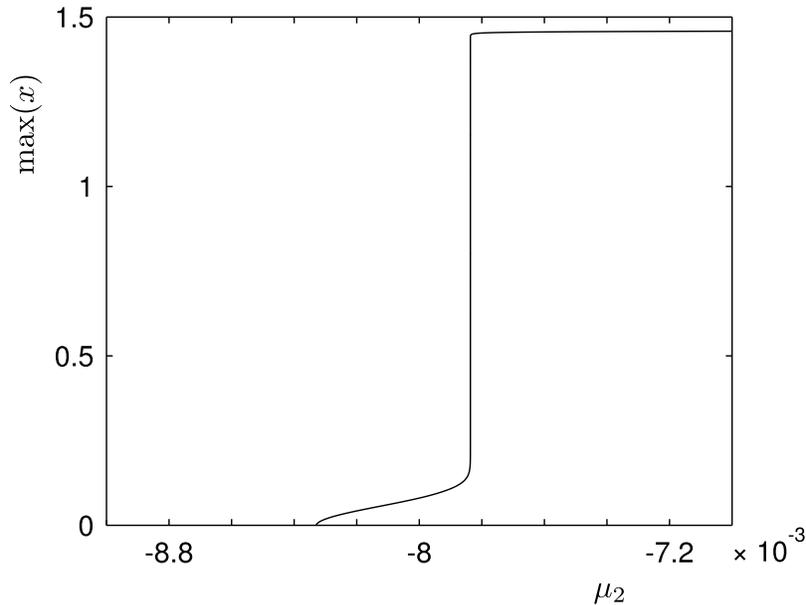}
\caption{The amplitude of the limit cycles, measured in terms of $\text{max}(x)$, as a function of $\mu_2$. The amplitude \textit{explodes} due to the presence of a maximal canard around $\mu_{2}\approx-7.8366 \times 10^{-3}$. The flat part beyond the canard explosion gives rise to the \textit{canard}-like limit cycles similar the one shown in \figref{duck}.}\figlab{amplitude}
                		\end{center}
              \end{figure}

              \begin{figure}\begin{center}\includegraphics[width=.65\textwidth]{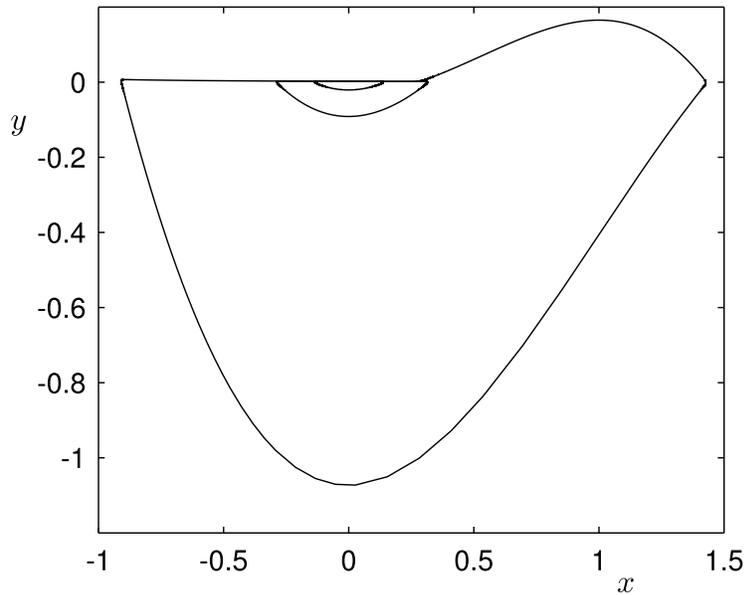}
\caption{Three different limit cycles. Due to the canard explosion phenomenon (see \figref{amplitude}) the difference in parameter between the largest and the second largest limit cycle is extremely small: $10^{-11}$.  }\figlab{duck}
                		\end{center}
              \end{figure}


              Now we replace the cubic regularization function in \eqref{phiCubic} by the following septic $C^1$ regularization function\footnote{The superscript $s$ now stands for \textit{septic}.}:
              \begin{align}
             \phi^s(\hat y) = -\frac{55}{54}\hat y^7+\frac{83}{54}\hat y^5-\frac{14}{27}\hat y^3+\hat y,\quad \hat y\in (-1,1).\eqlab{phiNew}
              \end{align}
                            This regularization function has been constructed so that $a_2$, using \eqref{K2}, becomes
              \begin{align*}
               a_2^s = \frac{5}{64}.
              \end{align*}
This value is just the negative of the previous value $a_2^c$ in \eqref{a2csystem1}. Hence periodic orbits emanating from the Hopf bifurcation are repelling and appear for $\mu<\mu_{2,H}$ where now
\begin{align}
\mu_{2,H}=\mu_{2,H}^s \equiv -\frac18 \sqrt{\epsilon}+\mathcal O(\epsilon).\eqlab{mu2Hsystem1New}
\end{align}
The canard value $\mu_{2,c}$ also changes and becomes
\begin{align}
 \mu_{2,c}^s&\approx -0.12188 \sqrt{\epsilon} + \mathcal O(\epsilon).\eqlab{mu2cSystem1Newer}
\end{align}
We again use AUTO with $r_2=0.1$ to continue periodic orbits from the Hopf bifurcation at $\mu=\mu_{2,H}^s$ \eqref{mu2Hsystem1New}. We obtain the bifurcation diagram in \figref{amplitudeNew}. As opposed to \figref{amplitude} we now observe a saddle-node (SN) bifurcation, which occurs before the canard explosion phenomenon. 

\begin{figure}\begin{center}\includegraphics[width=.65\textwidth]{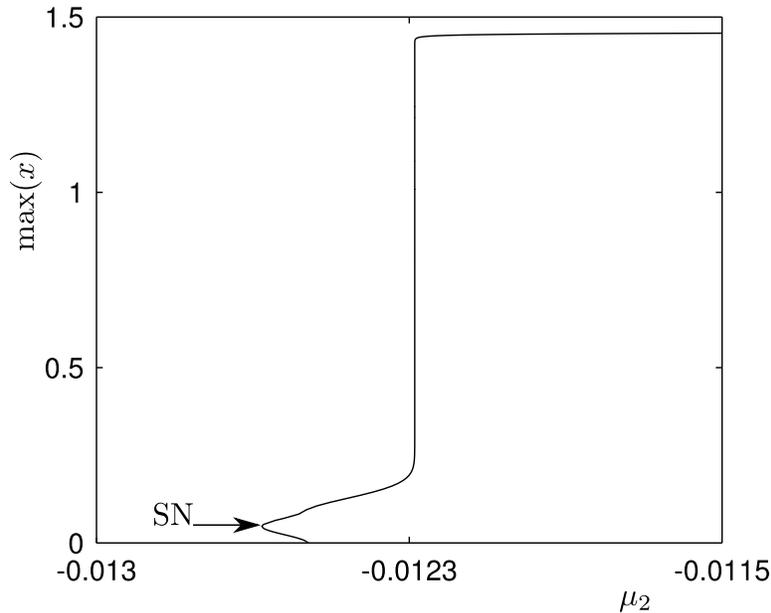}
\caption{This diagram shows the amplitude of the limit cycles, measured in terms of $\text{max}(x)$, as a function of $\mu_2$ using the septic regularization function \eqref{phiNew} to regularize the system \eqref{system1}. The amplitude \textit{explodes} due to the presence of a maximal canard around $\mu_{2}\approx--1.22369 \times 10^{-2}$. }\figlab{amplitudeNew}
                		\end{center}
              \end{figure}

\begin{remark}
Fixing the values of $\delta$, $\alpha$ and $\beta$ it is straightforward to construct a family of model systems for $VI_3$ where the Lyapunov coefficients corresponding to regularization functions $\phi^l$ and $\phi^c$ (see \eqsref{phiLinear}{phiCubic}) have opposite signs: $a_2^la_2^c<0$. A simple example is the following:
\begin{align}
 X^+(x,y) = \begin{pmatrix}
             1+7x\\
             x+8x^2
            \end{pmatrix},\quad X^-(x,y,\mu) = \begin{pmatrix}
             -1+6x\\
             -\frac32 (x-\mu)+8(x-\mu)^2
            \end{pmatrix},\nonumber
\end{align}
where 
\begin{align*}
 a_2^l&=-\frac58,\quad 
 a_2^c=\frac{7}{32}.
\end{align*}
However, we have not presented the details of this case since the saddle-node bifurcation occurs very close to the canard value and therefore it is not as clearly visible as the saddle-node in \figref{amplitudeNew}.
\end{remark}

\begin{remark}
This ``duck'' part in \figref{duck} is not covered by our results. However, it seems very plausible that the results in \cite{krupa_relaxation_2001} can be extended to this case too.
\end{remark}

\section{Discussion and Conclusions}\seclab{conclusions}
In this paper, we have considered the regularization of the codimension one two-fold bifurcation in planar PWS systems. The PWS two-fold bifurcation is dynamically very interesting as it may include singular canards, pseudo-equilibria and limit cycles. Using the blowup method of Krupa and Szmolyan \cite{krupa_extending_2001}, we continued these objects into the regularization and we related the PWS bifurcations to standard smooth bifurcations. 
Perhaps most interestingly, we were able to show that the regularization can induce saddle-node bifurcations of the limit cycles. The results were illustrated by numerical examples.

{There are two questions that emerge from this work that we feel are worthy of further discussion. What light can regularization shed on the original PWS system? Is the introduction of saddle-node bifurcations a necessary consequence of regularization? 

For the first question, it is clear that singular canards and pseudo-equilibria of the PWS system are limits of equivalent objects in the regularized system. Similarly, limit cycles in the PWS two-fold case $II_2$ are limits as $\epsilon\rightarrow 0$ of limit cycles in the regularized case $II_2^{\epsilon}$, at least ``macroscopically''; the saddle-node bifurcations occur ``microscopically'' within chart $\kappa_2$. In comparison, the PWS case $VI_3$ is more singular. It possesses backwards and forwards non-uniqueness of orbits due to the presence of stable and unstable sliding. In particular, it is possible to identify closed ``singular cycles'', reminiscent of singular cycles in slow-fast systems such as the van der Pol system (see \figref{VI3vdP}). Our analysis showed that these singular cycles are limits of periodic orbits of the regularization. Interestingly, the quantity $\Delta_{VI_3}$ defined in \eqref{Beqn} only depends on the PWS system, giving us an insight into the stability of a very singular object. The limit cycles of the regularization undergo a canard explosion phenomenon which gives rise to a very rapid amplitude increase of local periodic orbits. This can lead to global limit cyles as it was shown in \secref{numericsVI3} and \figref{duck}. 
 

The second question is much broader. In this paper, we have considered planar two-folds, subject to the Sotomayor and Teixeira \cite{Sotomayor96} regularization. We have shown that the criticality of Hopf bifurcations depends on the regularization function and generically it is possible to induce saddle-node bifurcations by varying the regularization function. But we have not shown how many saddle-node bifurcations may exist. 
Perhaps there are other PWS systems where the regularization does not induce bifurcations. Or there may be systems where other types of behaviour occur upon regularization. In addition there are other regularizations that could be considered. 
}



\bibliography{refs}
\bibliographystyle{plain}


\end{document}